\documentclass[12pt]{article}

\usepackage[english]{babel}
\usepackage[utf8]{inputenc}

\usepackage[margin=1.13 in, top=1.1 in, bottom= 1.2 in]{geometry}
\makeatletter
\g@addto@macro\normalsize{%
  \setlength\abovedisplayskip{7pt}
  \setlength\belowdisplayskip{7pt}
  \setlength\abovedisplayshortskip{7pt}
  \setlength\belowdisplayshortskip{7pt}
}
\makeatother

\usepackage{tocloft}

\usepackage{amsfonts}
\usepackage{mathrsfs}
\usepackage{bbm}
\usepackage{latexsym}
\usepackage{amssymb}
\usepackage{mathtools}
\usepackage{relsize}
\usepackage{lipsum}

\usepackage{todonotes}
\usepackage{enumitem}
\setlist{nolistsep} 	
\usepackage{amsthm}

\usepackage{xcolor}
\definecolor{Color1}{rgb}{0.0, 0.42, 0.47}
\definecolor{Color2}{rgb}{0.78, 0.11, 0.0}

\usepackage{titlesec}
\titleformat{\section}
  {\large\center\bfseries}
  {\thesection.}{.7em}{}
\titlespacing*{\section}{0pt}{3.5ex plus 0ex minus 0ex}{1.5ex plus 0ex}
\titleformat{\subsection}
  {\center\bfseries}
  {\thesubsection.}{.7em}{}
\titlespacing*{\subsection}{0pt}{3.5ex plus 0ex minus 0ex}{1.5ex plus 0ex}
\titleformat{\subsubsection}
  {\center\bfseries}
  {\thesubsubsection.}{.7em}{}
\titlespacing*{\subsubsection}{0pt}{3.5ex plus 0ex minus 0ex}{1.5ex plus 0ex}

\addto\captionsenglish{}

\usepackage{titling}
\makeatletter
\renewenvironment{abstract}{
\begin{center}
{\bfseries \large\abstractname\vspace{\z@}}
\end{center}
\quotation
}
\makeatother

\usepackage[linktocpage=true]{hyperref}
\usepackage[capitalize]{cleveref}
\hypersetup{citecolor = Color1,colorlinks,
			linkcolor = black,
			urlcolor = Color2}

\newtheoremstyle{plain}{3mm}{3mm}{\slshape}{}{\bfseries}{.}{.5em}{}
\newtheoremstyle{definition}{2mm}{2mm}{}{}{\bfseries}{.}{.5em}{}
\theoremstyle{plain} 
	
\newtheorem{theorem}{Theorem}[section]

\newtheorem{proposition}[theorem]{Proposition}

\newtheorem{lemma}[theorem]{Lemma}
\newtheorem{corollary}[theorem]{Corollary}
\theoremstyle{definition} 
\newtheorem{definition}[theorem]{Definition}
\newtheorem{remark}[theorem]{Remark}
\newtheorem{example}[theorem]{Example}

\theoremstyle{plain} 
\newcounter{MainTheoremCounter}

\newtheorem{maintheorem}[MainTheoremCounter]{Theorem}

\theoremstyle{plain} 

\usepackage{chngcntr}

\numberwithin{equation}{section}

\allowdisplaybreaks

\newcommand{\Cesaro}{Ces\`{a}ro }

\newcommand{\N}{\mathbb{N}}
\newcommand{\Z}{\mathbb{Z}}
\newcommand{\R}{\mathbb{R}}
\newcommand{\C}{\mathbb{C}}
\newcommand{\Q}{\mathbb{Q}}
\newcommand{\T}{\mathbb{T}}

\renewcommand{\epsilon}{\varepsilon}
\renewcommand{\leq}{\leqslant}
\renewcommand{\geq}{\geqslant}
\renewcommand{\setminus}{\backslash}

\renewcommand{\subset}{\subseteq}

\usepackage[normalem]{ulem}

\usepackage[makeroom]{cancel}

\usepackage{fixme}
\fxsetup{status=draft}

\newcommand{\E}{\operatorname{\mathbb{E}}}

\newcommand{\Poly}{\operatorname{Poly}}
\newcommand{\oo}{\infty}

\newcommand{\floor}[1]{\lfloor #1 \rfloor}

\newcommand{\round}[1]{[ #1 ]}
\newcommand{\unif}{\operatorname{unif}}
\renewcommand{\mod}{\operatorname{mod}}
\newcommand{\psim}{\sim_{\text{PET}}}

\author{{\scshape Vitaly Bergelson and Michael Reilly}}
\date{\small \today}
\title{\bfseries
Weighted averages and applications to sets of multiple recurrence}

\begin{document}
\maketitle




\begin{abstract}
We introduce new techniques for determining combinatorial properties of sets of multiple recurrence by considering weighted averages with quickly growing weights. Our main result is a far-reaching generalization of Szemer\'edi's Theorem which additionally confirms a conjecture of Bergelson-Moreira-Richter and contains as special cases both a variant of the Polynomial Szemer\'edi Theorem due to Bergelson-Leibman-Lesigne and the fact that if $f$ belongs to a broad class of smooth functions and satisfies $x^{d-1}\prec f(x)\prec x^d$ for some $d\in \mathbb{N}$ then for any $\ell\in \mathbb{N}$, any invertible measure preserving system $(X,\mathscr{B},\mu,T)$, and any $A\in \mathscr{B}$ with $\mu(A)>0$, the set $\{n\in \mathbb{N}: \mu(A\cap T^{-\round{f(n)}}A\cap T^{-2\round{f(n)}}A\cap \cdots\cap T^{-\ell\round{f(n)}}A )>0\}$ is thick, meaning that it contains arbitrarily long intervals of natural numbers. Additionally, we formulate and prove a generalization to weighted averages of Boshernitzan's criterion for uniform distribution which we use in the proof of our main result.
\end{abstract}
\section{Introduction}\label{section:introduction}

Classically, ergodic theory studies various sequences arising from dynamical systems using \emph{\Cesaro averages}, that is, averages of the form 
\begin{equation}\label{eq:Cesaro_def}
\E_{n\leq N}x_n :=\frac{1}{N}\sum_{n=1}^Nx_n.
\end{equation}
For instance, if $(X,\mathscr{B},\mu,T)$ is a measure preserving system we may take $x_n=f(T^nx)$ or $x_n = \mu(A\cap T^{-n} A)$ for $n\in\N= \{1,2,\dots\}$, $f\in L^{2}(X)$, and $A\in \mathscr{B}$. In recent years, many applications to combinatorics and number theory have been obtained by considering more general weighted ergodic averages, namely, averages having the form
\begin{equation}\label{eq:W_avg_def}
\E_{n\leq N}^Wx_n:=\frac{1}{W(N)}\sum_{n=1}^N\Delta W(n)x_n
\end{equation}
where $W$ is a function which eventually increases to $\oo$, $\Delta W$ is defined by $\Delta W(n) = W(n)-W(n-1)$ for $n\geq 2$ and $\Delta W(1)=W(1)$ (note that (\ref{eq:Cesaro_def}) and (\ref{eq:W_avg_def}) are the same when $W(N)=N$).\footnote{Some sources define $\Delta W(n) = W(n+1)-W(n)$.}
For example, \cite{Bergelson_Moreira_Richter_2020}, \cite{Fran22}, \cite{Richter23}, \cite{BMR24}, \cite{LM25}, \cite{BKS2025} each contain results concerning sequences $(x_n)_{n\in \N}$ whose \Cesaro averages do not converge while the weighted averages $\E_{n\leq N}^Wx_n$ do converge so long as $W$ grows slowly enough. This illustrates a general phenomenon: for functions $W_1, W_2$ such that $W_1$ grows much faster than $W_2$, if $\lim_{N\to\oo}\E_{n\leq N}^{W_1}x_n$ exists for some bounded sequence $(x_n)_{n\in \N}$ then $\lim_{N\to\oo}\E_{n\leq N}^{W_2}x_n$ also exists (cf. item (2) in Theorem \ref{thm:mikey_thm} below), but the converse is often not true.

\begin{example}\label{ex:ud}
    Let $W(N) = e^{\sqrt{N}}$ for $N\in \N$. Theorem \ref{thm:W_ud} below implies the following.
    \begin{itemize}
        \item For the sequence $x_n = e^{2\pi i \log(n)}$ for $n\in \N$
        \begin{equation*}
            \lim_{N\to\oo}\E^{\log}_{n\leq N}x_n=0, \quad  
        \lim_{N\to\oo}\E_{n\leq N}x_n \text{ does not exist, }\quad  
        \lim_{N\to\oo}  \E^W_{n\leq N}x_n \text{ does not exist.}
        \end{equation*}
        \item For the sequence $y_n = e^{2\pi i n\log(n)}$ for $n\in \N$
        \begin{equation*}
            \lim_{N\to\oo}\E^{\log}_{n\leq N}y_n=0,\quad \quad  
        \lim_{N\to\oo}\E_{n\leq N}y_n=0 \quad \quad 
        \lim_{N\to\oo}  \E^W_{n\leq N}y_n \text{ does not exist.}
        \end{equation*}
        \item For the sequence $z_n = e^{2\pi  i n\sqrt{2}}$ for $n\in \N$,
        \begin{equation*}
            \lim_{N\to\oo}\E^{\log}_{n\leq N}z_n=0,\quad \quad  
        \lim_{N\to\oo}\E_{n\leq N}z_n=0 \quad \quad 
        \lim_{N\to\oo}  \E^W_{n\leq N}z_n=0 .
        \end{equation*}
    \end{itemize}
\end{example}

In this paper, we consider the opposite regime, where the function $W$ grows fast enough so that the weighted averages $\E_{n\leq N}^Wx_n$ might not converge even when the \Cesaro averages do converge (for example $W(N) = e^{\sqrt{N}}$ in Example \ref{ex:ud}). For quickly growing weights, results about convergence of weighted averages yield new information that cannot be obtained through convergence of \Cesaro averages. We prove weighted generalizations of statements concerning  multiple ergodic averages, which in turn allows us to improve known theorems and obtain new amplifications of classical results.

For instance, one application of our results improves a theorem of Frantzikinakis and Wierdl (Theorem \ref{thm:Hardy_szemeredi} below) which shows that return times in the ergodic form of Szemer\'{e}di's Theorem can be found along nonpolynomial functions from a Hardy field. Before we can give the full formulation, we must recall some definitions concerning Hardy fields.

    For a smooth function $f:[0,\oo)\rightarrow \R$, the germ of $f$ is the equivalence class of smooth functions 
    $$
    \{g:[0,\oo)\rightarrow \R: \text{ there exists } c>0 \text{ such that } f(x)=g(x) \text{ for all }x>c\}.
    $$
    Let $\mathbb{B}$ be the set of all germs of smooth functions. $\mathbb{B}$ is a ring under pointwise addition and pointwise multiplication and we call $\mathcal{H}\subset \mathbb{B}$ a \emph{Hardy field} if $\mathcal{H}$ is a sub-field of $\mathbb{B}$ which is closed under differentiation. We call $f$ a Hardy function, and write $f\in \mathcal{H}$, if the germ of $f$ belongs to a Hardy field $\mathcal{H}$. A \emph{maximal Hardy field} is a Hardy field which is not a proper subset of any other Hardy field. For more detailed discussion on Hardy fields, see (\cite{Boshernitzan81}, \cite{Boshernitzan}, \cite{Fran09}). In particular we list some important facts that we will frequently make use of.
\begin{itemize}
    \item If $f$ is a Hardy function then $f$ is eventually monotone and hence $\lim_{x\to\oo}f(x)$ exists in $\R\cup\{\oo,-\oo\}$ \cite[Section 1.1]{FW09}.
\end{itemize}
 For functions $f,g$ we use the notation $f\prec g$ to mean that $\lim_{x\to\oo}f(x)/g(x)=0$ and we write\footnote{Some sources use $f \ll g$ to mean $\limsup_{x\to\oo}{f(x)}/{g(x)}<\oo$.} $f\preceq g$ to mean that $\limsup_{x\to\oo}f(x)/g(x)<\oo$.
 \begin{itemize}
    \item If $f$ and $g$ are contained in the same Hardy field then $f/g$ is a Hardy function and so $\lim_{x\to\oo}\frac{f(x)}{g(x)}$ exists in $\R\cup \{\oo,-\oo\}$. Hence, either $f\preceq g$ or $g\preceq f$.
    \item The Hardy field of logarithmico-exponential functions is the smallest field which contains $x, \log, \exp$ and is closed under composition.
    \item The Hardy field of logarithmico-exponential functions is contained in every maximal Hardy field.
    \item If $\mathcal{H}$ is a maximal Hardy field, $g\in \mathcal{H}$ with $\lim_{x\to\oo}g(x)=\oo$, and $f$ belongs to each maximal Hardy field then $f\circ g\in \mathcal{H}$ \cite[Proposition 6.9]{Boshernitzan81}.
    \item If $\mathcal{H}$ is a maximal Hardy field and $f\in \mathcal{H}$ then $\int_0^xf(t)~dt\in \mathcal{H}$ \cite[Theorem 5.3]{Boshernitzan81}.
\end{itemize}
\begin{definition}\label{def:degree}
    Suppose that $f$ belongs to a Hardy field. 
    If there exists a $d\in \N$ such that $|f(x)|\preceq x^{d}$, then we define $\deg(f)$ to be the minimal such value of $d$ and we say that $f$ is \emph{subpolynomial}.
    Additionally, we put $\deg^*(f) = \min\{\deg(f-q):q(x)\in \Q[x]\}$.
\end{definition}

In the following theorem and throughout the paper, $[\cdot ]:\R\rightarrow\Z$ denotes the {rounding function} satisfying $[x] = \floor{x+1/2}$ for $x\in \R$, where $\floor{x} = \sup\{n\in \Z: n\leq x\}$.

\begin{theorem}[{\cite[Theorem 6.1]{FW09}}]\label{thm:Hardy_szemeredi}
    Let $\mathcal{H}$ be a Hardy field. Suppose that $f\in \mathcal{H}$ satisfies $x^{d-1}\prec f(x)\prec x^{d}$ for some $d\in \N$. Then for any $\ell\in \N$ and any invertible measure preserving system $(X,\mathscr{B},\mu,T)$ and any $A\in \mathscr{B}$ with $\mu(A)>0$, the set 
\begin{equation}\label{eq:Hardy_sz_returns_1}
   \{n:\mu(A\cap T^{-[ f(n)]}A\cap T^{-2[ f(n)]}A\cap \cdots\cap T^{-\ell[ f(n)]}A )>0\}
\end{equation}
is nonempty.
\end{theorem}

With our methods, we are able to obtain the following improvement of Theorem \ref{thm:Hardy_szemeredi}.

\begin{theorem}\label{thm:thick_szemeredi}
    Let $\mathcal{H}$ be a Hardy field. Suppose that $f\in \mathcal{H}$ satisfies $x^{d-1}\prec f(x)\prec x^{d}$ for some $d\in \N$. Then for any $\ell\in \N$ and any invertible measure preserving system $(X,\mathscr{B},\mu,T)$ and any $A\in \mathscr{B}$ with $\mu(A)>0$, the set 
\begin{equation}\label{eq:Hardy_sz_returns}
   \{n:\mu(A\cap T^{-[ f(n)]}A\cap T^{-2[ f(n)]}A\cap \cdots\cap T^{-\ell[ f(n)]}A )>0\}
\end{equation}
is thick, meaning that it contains arbitrarily long intervals of natural numbers.
\end{theorem}


Theorem \ref{thm:thick_szemeredi} is reminiscent of results found in \cite{Bergelson_Moreira_Richter_2020} (see Theorem \ref{thm:BMR_1} below), except that Theorem \ref{thm:thick_szemeredi} requires $f$ to belong to a Hardy field and allows for multiplication outside of the rounding function, e.g. $2\cdot \round{f(n)}$ instead of $\round{2\cdot f(n)}$.

We will see in Section \ref{section:applications} that Theorem \ref{thm:thick_szemeredi} follows from a general result about multiple recurrence, Theorem \ref{thm:general_thick_szemeredi} below. Another special case of Theorem \ref{thm:general_thick_szemeredi} is a variant of the Polynomial Szemer\'{e}di Theorem, which we formulate after the following definition.
\begin{definition}
 Let $R\subset \R[x]$ be a finite set of real polynomials. We say that $R$ is \emph{jointly intersective} if there exists a finite set $P\subset \Z[x]$ with $R\subset \text{Span}_{\mathbb{R}}(P)$ such that for each $r\in \N$ there exists $n\in \N$ such that $p(n)$ is divisible by $r$ for all $p(x)\in P$. 
\end{definition}
\begin{remark}
    Equivalently, $R$ is jointly intersective if and only if $R$ is contained in the principle ideal of $\R[x]$ generated by $p_0(x)$, where $p_0(x)\in \Z[x]$ is a polynomial such that for each $r\in \N$ there is $n\in \N$ with $p_0(n)$ divisible by $r$ (cf. \cite[Proposition 6.1]{BLL08}). Moreover, if each polynomial contained in $R$ has constant term equal to $0$, then it is trivial to see that $R$ is jointly intersective because $R\subset (x)$.
\end{remark}

\begin{theorem}[Polynomial Szemer\'edi Theorem \cite{BLL08}]\label{thm:polynomial_sz}
    Let $P= \{p_1,\dots, p_{\ell}\}\subset \Z[x]$. $P$ is jointly intersective if and only if for any invertible measure preserving system $(X,\mathscr{B},\mu,T)$ and any $A\in \mathscr{B}$ with $\mu(A)>0$, 
\begin{equation}\label{eq:polynomial_sz}
R_A(p_1\dots, p_{\ell})=\{n\in \N: \mu(A\cap T^{-p_1(n)}A\cap\cdots \cap T^{-p_{\ell}(n)}A)>0\}
\end{equation}
is {syndetic}, meaning that it has bounded gaps.
\end{theorem}

We will generalize this version of the Polynomial Szemer\'{e}di Theorem to nonpolynomial functions by considering a generalization of the notion of syndeticity.

\begin{definition}\label{def:W_syndetic}
    Let $W$ be a function with $1\prec W(x)\preceq x$. We say that $S\subset \N$ is $W$\emph{-syndetic} if 
    \begin{equation}
        \liminf_{W(N)-W(M)\to\oo}\frac{1}{W(N)-W(M)}\sum_{n=M}^N\Delta W(n)1_{S}(n)>0.
    \end{equation}
\end{definition}

\begin{theorem}\label{thm:general_thick_szemeredi}
    Let $\mathcal{H}$ be a Hardy field and let $f_1,\dots, f_{\ell}\in \mathcal{H}$ be subpolynomial. Let $W\in \mathcal{H}$ with $\log x\prec  W(x)\prec x$ such that 
    \begin{equation}
    \lim_{x\to\oo}\frac{|f^{(d)}(x)-q(x)|^{1/d}}{W'(x)}>0
    \end{equation}
    for each $q(x)\in \Q[x]$ and each unbounded $f\in \text{Span}_{\mathbb{R}}\{f_1,\dots, f_{\ell}\}$, where $d = \deg^*(f)$. Define 
    \begin{equation}\label{eq:poly_def}
         \Poly\{f_1,\dots, f_{\ell}\} := \{p(x)\in \R[x]:\lim_{x\to\oo}|f(x)-p(x)|= 0\text{ for some } f\in  \text{Span}_{\mathbb{R}}\{f_1,\dots, f_{\ell}\}\}.
\end{equation}
Suppose that $\Poly\{f_1,\dots, f_{\ell}\}$ is jointly intersective. Then for any invertible probability measure preserving system $(X,\mathscr{B},\mu,T)$ and $A\in \mathscr{B}$ with $\mu(A)>0$, the set 
\begin{equation}\label{eq:general_Hardy_sz_returns}
    R_A(f_1,\dots, f_{\ell}):=\{n:\mu(A\cap T^{-[ f_1(n)]}A\cap T^{-[ f_2(n)]}A\cap \cdots\cap T^{-[ f_{\ell}(n)]}A )>0\}.
\end{equation}
is $W$-syndetic. In particular, $R_A(f_1,\dots, f_{\ell})$ is nonempty (a fact shown in \cite{BMR24}).
\end{theorem}

We show in Section \ref{section:applications} that the nontrivial direction of Theorem \ref{thm:polynomial_sz} follows immediately as a special case of Theorem \ref{thm:general_thick_szemeredi} by noting that (\ref{eq:polynomial_sz}) is $W$-syndetic for each $W$ with $\log x\prec \log W(x)\prec x$, from which it follows that (\ref{eq:polynomial_sz}) is syndetic. Additionally, we will use Theorem \ref{thm:general_thick_szemeredi} to prove Theorem \ref{thm:thick_szemeredi} in Section \ref{section:applications} as well.

Similarly to the treatment of Theorem \ref{thm:polynomial_sz} in \cite{BLL08}, the proof of our main result can be reduced to the case when $(X,\mathscr{B},\mu,T)$ is a nilsystem (Definition \ref{def:nilsystem}). The arguments in \cite{BLL08} are phrased in terms of uniform \Cesaro averages, which are well suited to working with polynomial sequences. In this paper, we utilize more technical nilsystem arguments developed in \cite{Richter23} and \cite{BMR24}, which allow us to use general weighted averages of the form (\ref{eq:W_avg_def}) when dealing with non-polynomial functions. Using these methods we obtain our main result, Theorem \ref{thm:main} below, with Theorem \ref{thm:general_thick_szemeredi} as a special case. In order to state Theorem \ref{thm:main}, we need another definition and some preliminary facts.

 \begin{definition} \label{def:W_compatible}
 Let $\mathcal{H}$ be a Hardy field and let $W,f\in \mathcal{H}$. Suppose that $1\prec \log W(x)\prec x$ and that $f$ is subpolynomial. Let $d=\deg^*(f)$. We say that $W$ is \emph{compatible} with $f$ if either $d=0$ or $d>0$ and
 \begin{equation}\label{eq:W_compatible}
     \lim_{x\to\oo}\frac{|f^{(d)}(x)-p(x)|^{1/d}}{(\log W)'(x)}=\oo.\tag{$\star$}
 \end{equation}
for all $p(x)\in \Q[x]$. Let $f_1,\dots, f_{\ell}\in \mathcal{H}$. We say that $\{f_1,\dots, f_{\ell}\}$ satisfies property (\ref{eq:W_compatible}) if $W$ is compatible with each $f\in \text{Span}_{\mathbb{R}}\{f_1,\dots, f_{\ell}\}$.
\end{definition}
\begin{example}
    Let $f(x) = x^c$ for $c\in (0,\oo)\setminus \N$. Then $W$ is compatible with $f$ so long as $\log W(x)$ grows slower than $x^{c/\floor{c}}$. For instance, $W(x) = e^{x^{(3/4-\epsilon)}}$ is compatible with $f(x)=x^{3/2}$ for each $\epsilon>0$.
\end{example}

Using the following remark, observe that $W(x)=x$ is compatible with a subpolynomial Hardy function $f$ with $\deg^*(f)>0$ if and only if $f$ satisfies Boshernitzan's criterion for uniform distribution modulo 1, namely $\lim_{x\to\oo}\frac{|f(x)-p(x)|}{\log x}= \oo$ for all $p(x)\in \Q[x]$
 \cite[Theorem 1.3]{Boshernitzan}.

\begin{remark}\label{remark:WP}
In \cite[Section 6]{BMR24} it is said that $W,f_1,\dots, f_{\ell}$ satisfy property (WP) if each $f_i$ is subpolynomial, $1\prec W(x)\preceq x$, and 
\begin{align*}
    &\text{ For each }f\in \text{Span}_{\mathbb{R}}\{f_1,\dots, f_{\ell}\} \text{ and } p(x)\in \R[x],\text{ either } |f(x)-p(x)|\preceq 1\\ &\text{ or } \log W(x)\prec |f(x)-p(x)|\tag{WP}.
\end{align*}
We observe that (WP) holds for any $f$ with $\deg^*(f)\geq 2$. If $1\prec W(x)\preceq x$ (in fact, if $1\prec \log W(x)\preceq \log x$) then property (\ref{eq:W_compatible}) also holds for any $f$ with $\deg^*(f)\geq 2$, which can be verified using L'H\^opital's rule. Indeed, for $\deg^*(f)= 2$ and $1\prec \log W(x)\preceq \log x$, and for any $p(x)\in \Q[x]$, 
\begin{align*}
    &\lim_{x\to\oo}\frac{|f''(x)-p(x)|^{1/2}}{(\log W)'(x)} = \lim_{x\to\oo}\frac{|f''(x)-p(x)|^{1/2}}{x^{-1}}\cdot \frac{x^{-1}}{(\log W)'(x)}\\ 
    =& \left(\lim_{x\to\oo}\frac{|f''(x)-p(x)|}{x^{-2}}\cdot \left(\frac{x^{-1}}{(\log W)'(x)}\right)^2\right)^{1/2}.
\end{align*}
We know that 
\begin{equation*}
    \lim_{x\to\oo}\frac{|f''(x)-p(x)|}{x^{-2}}=\lim_{x\to\oo}\frac{|f(x)-P(x)|}{\ln(x)} =\oo,
\end{equation*}
where $P''(x)=p(x)$, since $\deg^*(f)=2$, and $\lim_{x\to\oo}\frac{x^{-1}}{(\log W)'(x)}>0$ by the assumption that $\log W(x)\preceq \log x$. This proves the case when $\deg^*(f)=2$ and the case when $\deg^*(f)>2$ follows since $|f^{(3)}(x)-p(x)|^{1/3}\succ |(f')^{(2)}-p(x)|^{1/2}$ whenever $|f^{(3)}(x)-p(x)|$ tends to $0$. This shows that property (\ref{eq:W_compatible}) holds when $\deg^*(f)\geq 2$ and $1\prec \log W(x)\preceq \log x$. 
Additionally, note that property (\ref{eq:W_compatible}) and (WP) are the same condition when $\deg^*(f)\leq 1$. Thus, it follows that in the case $1\prec W(x)\preceq x$, property (WP) as defined in \cite{BMR24} is equivalent to property (\ref{eq:W_compatible}) defined above.
\end{remark}

We are now able to state our main result, Theorem \ref{thm:main}. 

\begin{maintheorem}\label{thm:main}
     Let $\mathcal{H}$ be a Hardy field. Let $f_1,\dots, f_{\ell}\in \mathcal{H}$ be  subpolynomial and let $W\in \mathcal{H}$ satisfy $1\prec \log W(x)\prec x$. Suppose that $\{f_1,\dots, f_{\ell}\}$ satisfies property (\ref{eq:W_compatible}). Let $(X,\mathscr{B},\mu,T)$ be an invertible measure preserving system.
     \begin{enumerate}[label = (\roman*)]
         \item For each $h_1,\dots, h_{\ell}\in L^{\oo}(X)$,
         \begin{equation}\label{eq:Thm_A_1}
             \lim_{N\to\oo}\E_{n\leq N}^W(T^{\round{f_1(n)}}h_1\cdots T^{\round{f_{\ell}(n)}}h_{\ell})
         \end{equation}
         exists in $L^2(X)$.
        \item \label{condition_1}Suppose that $\Poly(f_1,\dots, f_{\ell})\cap \Z[x] = \{0\}$. 
        Then for each $h_1,\dots, h_{\ell}\in L^{\oo}(X)$,
        \begin{equation}\label{eq:Thm_A_2}
             \lim_{N\to\oo}\E_{n\leq N}^W(T^{\round{f_1(n)}}h_1\cdots T^{\round{f_{\ell}(n)}}h_{\ell}) = \prod_{i=1}^{\ell}h_i^*
         \end{equation}
         where $h_i^*$ is the projection in $L^2(X)$ of $h_i$ onto the subspace of $T$ invariant functions.
         \item\label{condition_2} Suppose that $\Poly\{f_1,\dots, f_{\ell}\}$ is jointly intersective.
         Then for any $A\in \mathscr{B}$ with $\mu(A)>0$,
    \begin{equation}\label{eq:Thm_A_3}
             \lim_{N\to\oo}\E_{n\leq N}^W\mu(A\cap T^{-\round{f_1(n)}}A\cap \cdots \cap T^{-\round{f_{\ell}(n)}}A) >0.
         \end{equation}
     \end{enumerate}
\end{maintheorem}
The condition that $\Poly(f_1,\dots, f_{\ell})\cap \Z[x] = \{0\}$ is equivalent to the condition that $\lim_{x\to\oo}|f(x)-p(x)|= \oo$ for any $p(x)\in \Z[x]$ and any nonzero $f\in \text{Span}_{\mathbb{R}}\{f_1,\dots, f_{\ell}\}$. This observation, along with Remark \ref{remark:WP}, shows that Theorem \ref{thm:main} confirms \cite[Conjecture 6.4]{BMR24}.
\begin{theorem}[{\cite[Conjecture 6.4]{BMR24}}]\label{thm:BMR_conj}
    Let $\mathcal{H}$ be a Hardy field. Let $f_1,\dots, f_{\ell}\in \mathcal{H}$ be  subpolynomial and let $W\in \mathcal{H}$ satisfy $1\prec W(x)\preceq x$. Suppose that $\{f_1,\dots, f_{\ell}\}$ satisfies (WP) (see Remark \ref{remark:WP}). Then each part of Theorem \ref{thm:main} holds.
\end{theorem}
Additionally, it was noted in \cite{BMR24} that taking $W(x)=x$ in Theorem \ref{thm:BMR_conj} gives \cite[Theorem 1.12]{TSINAS2023a} as a special case.
\begin{theorem}[{\cite[Theorem 1.12]{TSINAS2023a}}]
    Let $\mathcal{H}$ be a Hardy field and let $f_1,\dots, f_{\ell}\in \mathcal{H}$ be subpolynomial. Suppose that if $f$ is a nontrivial linear combination of $f_1,\dots, f_{\ell}$ then $f$ satisfies $\lim_{x\to\oo}\frac{|f(x)-q(x)|}{\log x}=\oo$ for all $q(x)\in \Z[x]$. Then for any ergodic measure preserving system and any $h_1,\dots, h_{\ell}\in L^{\oo}(X)$, the limit
    \[
    \lim_{N\to\oo}\frac{1}{N}\sum_{n=1}^N T^{\round {f_1(n)}}h_1\cdots T^{\round {f_{\ell}(n)}}h_{\ell}
    \] 
    converges to $\prod_{i=1}^{\ell}\int h_i~d\mu$ in $L^2$.
\end{theorem}

Next, we give an equivalent form of Theorem \ref{thm:main} in the style of the results appearing in \cite{Bergelson_Moreira_Richter_2020} by using \emph{uniform $W$-averages}, which are averages of the form 
\begin{equation} 
\E_{\text{unif}}^Wx_n:=\lim_{W(N)-W(M)\to\oo}\frac{1}{W(N)-W(M)}\sum_{n=M}^N\Delta W(n)x_n
\end{equation}
where the limit is taken over all sequences of intervals $[M_k,N_k]$ for which $\lim_{k\to\oo}W(N_k)-W(M_k)=\oo$. In fact, we will use the following Theorem \ref{thm:uniform_main} to recover each of the main theorems in \cite{Bergelson_Moreira_Richter_2020} with some added assumptions (see Theorem \ref{thm:BMR_1}).

\begin{maintheorem}\label{thm:uniform_main}
     Let $\mathcal{H}$ be a Hardy field, let $W\in \mathcal{H}$ satisfy $
     \log x\prec  W(x)\prec x$, and let $f_1,\dots, f_{\ell}\in \mathcal{H}$ be subpolynomial. Suppose that 
     \begin{equation}\label{eq:thm_B_condition}
         \lim_{x\to\oo}\frac{|f^{(d)}(x)-p(x)|^{1/d}}{W'(x)}>0
     \end{equation}
     for each $p(x)\in \Q[x]$ and each unbounded $f\in \text{Span}_{\mathbb{R}}\{f_1,\dots, f_{\ell}\}$, where $d = \deg^*(f)$.  
     Let $(X,\mathscr{B},\mu,T)$ be an invertible measure preserving system.
     \begin{enumerate}[label = (\roman*)]
         \item For each $h_1,\dots, h_{\ell}\in L^{\oo}(X)$, the limit
         \begin{equation}\label{eq:uniform_A_1}
         \E_{\text{unif}}^{ W}(T^{\round{f_1(n)}}h_1\cdots T^{\round{f_{\ell}(n)}}h_{\ell})
         \end{equation}
         exists in $L^2(X)$.
        \item Suppose that $\Poly(f_1,\dots, f_{\ell})\cap \Z[x] = \{0\}$. Then for any $h_1,\dots, h_{\ell}\in L^{\oo}(X)$,
        \begin{equation}\label{eq:uniform_A_2}
            \E_{\text{unif}}^{ W}(T^{\round{f_1(n)}}h_1\cdots T^{\round{f_{\ell}(n)}}h_{\ell}) = \prod_{i=1}^{\ell}h_i^*
         \end{equation}
         where $h_i^*$ is the projection in $L^2(X)$ of $h_i$ onto the subspace of $T$ invariant functions. It follows that for $A\in \mathscr{B}$ with $\mu(A)>0$ and for any $\epsilon>0$, the set 
         $$
 \{n\in \N:\mu(A\cap T^{-\round{f_1(n)}}A\cap\cdots \cap T^{-\round{f_{\ell}(n)}}A)>\mu(A)^{k+1}-\epsilon\}
         $$
         is $W$-syndetic.
         \item Suppose that $\Poly\{f_1,\dots, f_{\ell}\}$ is jointly intersective. Then for any $A\in \mathscr{B}$ with $\mu(A)>0$,
        \begin{equation}\label{eq:uniform_A_3}
             \E_{\text{unif}}^{ W}\mu(A\cap T^{-\round{f_1(n)}}A\cap \cdots \cap T^{-\round{f_{\ell}(n)}}A) >0.
         \end{equation}
     \end{enumerate}
\end{maintheorem}

In Section \ref{section:applications} we show how Theorem \ref{thm:uniform_main} can be used to obtain combinatorial corollaries. For instance, consider the following example.

\begin{example}
    Let $W(N)={\sqrt{N}}$ and let $f_1(x) = x^{0.51}$ (or any power of $x$ larger than $x^{0.5}$) and $f_2(x)= x\cdot \log(x)^2$. It is straightforward to verify that (\ref{eq:thm_B_condition}) holds and that $\Poly\{f_1,f_2\}=\{0\}$. Then by Theorem \ref{thm:uniform_main}(iii)
    \begin{equation}\label{eq:intro_example}
    \lim_{\sqrt{N}-\sqrt{M}\to\oo}\frac{1}{\sqrt{N}-\sqrt{M}}\sum_{n=M}^N\frac{1}{2\sqrt{n}}\mu(A\cap T^{-\round{n^{0.51}}}A\cap  T^{-\round{n\cdot \log(n)^2}}A) >0
    \end{equation}
    for any invertible $(X,\mathscr{B},\mu,T)$ and any $A\in \mathscr{B}$ with $\mu(A)>0$. Let $E\subset \N$ be a set with $\overline{d}(E):=\limsup_{N\to\oo}\E_{n\leq N}1_E(n) > 0$. Using (\ref{eq:intro_example}), Theorem \ref{thm:mikey_thm}, and Furstenberg's correspondence principle (Theorem \ref{thm:Furstenberg_correspondence} below), there exists a value $a\in E$ such that for all large enough $N\in \N$, there exists $n\in [N-N^{0.51},N]$ with 
    \begin{equation}\label{eq:example}
   \{a,a+\round{n\cdot \log(n)^2},a+\round{n^{0.51}}\}\subseteq \E.
   \end{equation}
  Moreover, let $A$ be the set of all $a\in E$ for which (\ref{eq:example}) holds. Then $\overline{d}(A)>0$.
    \end{example}

\subsection{Acknowledgments}
The authors would like to thank Florian Richter for providing the inspiration behind the proof of Theorem \ref{thm:gaussian_dne}, Sa\'ul Rodr\'iguez Mart\'in for giving helpful comments, and an anonymous referee for detecting faulty arguments in a previous version of this manuscript.

\subsection{Outline of the paper}
In Section \ref{section:applications} we give applications of our results to combinatorics. Additionally, we describe how Theorem \ref{thm:uniform_main} implies Theorem \ref{thm:general_thick_szemeredi} and we prove that Theorem \ref{thm:uniform_main} and Theorem \ref{thm:main} are equivalent. Then we show how Theorem \ref{thm:polynomial_sz} and Theorem \ref{thm:thick_szemeredi} follow from Theorem \ref{thm:general_thick_szemeredi}. In Section \ref{section:uniform_distribution} we formulate and prove a generalization of Boshernitzan's criterion for uniform distribution, Theorem \ref{thm:W_ud}, which is of independent interest and which serves as an essential tool in our proof of Theorem \ref{thm:main}.
The remainder of the paper is dedicated to providing a proof of Theorem \ref{thm:main}. In Section \ref{section:preliminaries} we give preliminaries on nilmanifolds and weighted averages. In Section \ref{section:BMR_results} we prove Theorem \ref{thm:main} using certain facts about uniform distribution on nilmanifolds. In Section \ref{section:richter_results}, we prove these facts using the results of Section \ref{section:uniform_distribution}.

\section{Corollaries and applications of Theorem \ref{thm:main}}\label{section:applications}

The following theorem is a special case of \cite[Theorem C]{reilly26} and we will use it to deduce Theorem \ref{thm:uniform_main} from Theorem \ref{thm:main} and derive several applications.
\begin{theorem}[cf. {\cite[Theorem C]{reilly26}}]\label{thm:mikey_thm}
Let $\mathcal{H}$ be a maximal Hardy field. Let $W\in \mathcal{H}$ and suppose that $\log(x)\prec  \log W(x)\prec x$. Let $Y$ be a Banach space, let $(x_n)_{n\in \N}\subset Y$ be a bounded sequence, and let $L\in Y$. Then the following are equivalent.
\begin{enumerate}[label = (\arabic*)]
\item $\E_{\text{unif}}^{\log W}x_n = L$.
\item $\lim_{N\to\oo}\E_{n\leq N}^Ux_n = L$ for each $U\in\mathcal{H}$ with $1\prec \log U(x)\prec \log W(x)$.
\item $\lim_{N\to\oo} \frac{1}{s(N)}\sum_{n=N-s(N)}^Nx_n= L$ for each nondecreasing function $s:\N\rightarrow\N$ satisfying 
$
\lim_{N\to\oo}s(N) \cdot (\log W)'(N)=\oo
$
and $s(N)\leq N-1$ for all sufficiently large $N\in \N$.
\item $\lim_{k\to\oo}\limsup_{N\to\oo}|Y_{N,k}-L|=0$, where $Y_{N,0} = x_{N}$ and $Y_{N,k+1}=\E^{W}_{n\leq N}Y_{n,k}$ for $N\in \N$ and $k\geq 0$.
\end{enumerate}
Moreover, if $\lim_{N\to\oo}\E^{W}_{n\leq N}x_n=L$ then each of the above statements hold.
\end{theorem}
\begin{remark}\label{remark:1precW}
    The assumption that $\log x\prec \log W(x)$ is necessary for statement (3) to not be vacuous, but the implications $(1)\implies (4)\implies (2)$ are each proven in \cite[Corollary 2.2, Corollary 3.3, Theorem 4.1, Theorem 5.5]{reilly26} under the more general assumption that $1\prec \log W(x)$.
\end{remark}
Now we prove that Theorem \ref{thm:uniform_main} and Theorem \ref{thm:main} are equivalent using the 
the equivalence of items $(1)$ and $(2)$ in Theorem \ref{thm:mikey_thm}. 
\begin{proof}[Proof of the equivalence of Theorem \ref{thm:uniform_main} and  Theorem \ref{thm:main}]
    Pick $W\in \mathcal{H}$ such that (\ref{eq:thm_B_condition}) holds. Let $V(x)= \exp(W(x))$, which is contained in any maximal Hardy field containing $\mathcal{H}$. Let $U\in \mathcal{H}$ be any function which satisfies $1\prec \log U(x)\prec \log V(x)$. Then by definition of $W$, we have 
    $$
    \lim_{x\to\oo}\frac{|f^{(d)}(x)-p(x)|^{1/d}}{(\log U)'(x)}=\oo
    $$
    for each $p(x)\in \Q[x]$ and each unbounded $f\in \text{Span}_{\mathbb{R}}\{f_1,\dots, f_{\ell}\}$, where $d = \deg^*(f)$ (see Definition \ref{def:degree}). So, if each part of Theorem \ref{thm:main} holds for $\E^U$ averages then by Theorem \ref{thm:mikey_thm} each part of Theorem \ref{thm:uniform_main} holds for $\E_{\unif}^W$ averages, and vice versa.
\end{proof}

We define now the notion of $W$-thickness, which is closely related to $W$-syndeticity (Definition \ref{def:W_syndetic}).

\begin{definition}\label{def:W_thick}
    Let $W$ be a function with $1\prec W(x)\preceq x$. We say that $S$ is $W$\emph{-thick} if 
    \begin{equation}
        \limsup_{W(N)-W(M)\to\oo}\frac{1}{W(N)-W(M)}\sum_{n=M}^N\Delta W(n)1_{S}(n)=1.
    \end{equation}
\end{definition}
Next, we give equivalent forms of $W$-syndeticity and $W$-thickness in terms of averages along intervals.
\begin{lemma}\label{lem:W_syndetic_equivalence}
     Let $W$ be a Hardy function which satisfies $\log(x)\prec W(x)\preceq x$ and let $S\subset \N$. The following are equivalent.
    \begin{enumerate}[label = (\roman*)]
        \item $S$ is $W$-thick,
        \item $S^c$ is not $W$-syndetic,
        \item There exists a nondecreasing function $s:\N\rightarrow \N$ with $\lim_{N\to\oo}s(N)\cdot  W'(N)=\oo$ and $s(N)\leq N-1$ for all sufficiently large $N\in \N$ such that   $$
        \limsup_{N\to\oo}\frac{1}{s(N)}\sum_{n=N-s(N)}^{N}1_S(n)=1.
        $$
        \item There is a sequence of intervals of natural numbers $(I_k)_{k\in \N}$ of the form $I_k  = [a_k,b_k]$ with $\lim_{k\to\oo}(b_k-a_k)\cdot W'(b_k) = \oo$ such that $\lim_{k\to\oo}\frac{|S\cap I_k|}{|I_k|}=1$.
    \end{enumerate}
\end{lemma}
\begin{proof}
Each of these statement are identical when $\lim_{x\to\oo}\frac{W(x)}{x}\in(0,\oo)$, so we assume that $W(x)\prec x$. The equivalence $(i)\iff (ii)$ is clear from the definition, since 
    \begin{align*}
    &\limsup_{W(N)-W(M)\to\oo}\frac{1}{W(N)-W(M)}\sum_{n=M}^N\Delta W(n)1_{S}(n) \\
    = 1-&\liminf_{W(N)-W(M)\to\oo}\frac{1}{W(N)-W(M)}\sum_{n=M}^N\Delta W(n)1_{S^c}(n).
    \end{align*}
    Similarly, statements $(iii)$ and $(iv)$ are restatements of each other.
    
    Next, we show the implication $(i)\implies (iii)$.
     Suppose that $A$ and $B$ are nondecreasing integer valued functions such that $\lim_{N\to\oo}B(N)-A(N)=\oo$ and 
     $$
     \lim_{N\to\oo}\frac{1}{W(B(N))-W(A(N))}\sum_{n=A
(N)}^{B(N)}\Delta W(n)1_{S}(n)=1.
$$
Let $s:\N\rightarrow \N$  be a nondecreasing function which satisfies $\lim_{N\to\oo}s(N)\cdot  W'(N)=\oo$ and $s(N)\leq N-1$ for all sufficiently large $N\in \N$. We will show that 
$$
\limsup_{N\to\oo}\frac{1}{s(N)}\sum_{n=N-s(N)}^{N}1_S(n)=1.
$$
In \cite[Equation 4.2]{reilly26}, it is shown that there exists a sequence of nonnegative constants $(c_{N,n})_{N,n\in \N}$ with $\lim_{N\to\oo}\sum_{n\in \N}c_{N,n}=1$ such that 
\begin{equation}\label{eq:(4)_implies_(5)}
    \sum_{k=A(N)}^{B(N)}c_{N,k}\left(\frac{1}{s(k)}\sum_{n=k-s(k)}^k1_S(n)\right)=\frac{\sum_{n=A(N)}^{B(N)}\Delta W(n)1_{S}(n)}{W(B(N))-W(A(N))}+o_{N\to \oo}(1).
    \end{equation}
Taking $\limsup$ of both sides shows that 
\begin{align*}
    1= \limsup_{N\to\oo}\frac{\sum_{n=A(N)}^{B(N)}\Delta W(n)1_{S}(n)}{W(B(N))-W(A(N))} =& \limsup_{N\to\oo}\sum_{k=A(N)}^{B(N)}c_{N,k}\left(\frac{1}{s(k)}\sum_{n=k-s(k)}^k1_S(n)\right)\\
    \leq & \limsup_{N\to\oo}\frac{1}{s(N)}\sum_{n=N-s(N)}^N1_S(n),
\end{align*}
and so $\limsup_{N\to\oo}\frac{1}{s(N)}\sum_{n=N-s(N)}^N1_S(n)=1$. This shows the implication $(i)\implies (iii)$.

    Lastly, we show the implication $(iv)\implies (i)$.
     Suppose that there is a sequence of intervals $I_N = [A(N),B(N)]$ such that $\lim_{N\to\oo}|I_N|\cdot W'(B(N))=\oo$ and for each $\epsilon>0$, $\lim \frac{|S\cap I_N|}{|I_N|}>1-\epsilon$.

Since $\log(x)\prec W(x)$, we know that $\lim_{N\to\oo}B(N)\cdot W'(B(N))=0$ and so (by shortening the intervals $I_N$ if necessary) we assume that $\lim_{N\to\oo}\frac{A(N)}{B(N)}=1$ and hence
$\lim_{N\to\oo}\frac{W'(A(N))}{W'(B(N))} = 1$.

Next, note that $W(B(N))-W(A(N)) = W'(\xi)\cdot |I_N|$ for some for some $\xi\in I_N$, by the Mean Value Theorem. Then 
\begin{align}\label{eq:MVT}
W'(B(N))\cdot |I_N|\leq W(B(N))-W(A(N)) \leq &W'(A(N))\cdot |I_N|
\end{align}
and hence 
\begin{align}\label{eq:MVT_2}
W(B(N))-W(A(N)) = W'(A(N))\cdot |I_N|\cdot (1+o_{N\to\oo}(1))
\end{align}
since $\lim_{N\to\oo}\frac{W'(A(N))}{W'(B(N))} = 1$. Additionally, 
$$
\lim_{N\to\oo}W(B(N))-W(A(N))=\oo
$$
because $\lim_{N\to\oo}W'(B(N))\cdot |I_N|= \oo$.

For each $\epsilon>0$ and for all large enough $N$, $I_N$ contains at most $(1-\epsilon)\cdot |I_N|$ elements of $S$. So 
\begin{align*}
    &\lim_{N\to\oo}\frac{1}{W(B(N))-W(A(N))}\sum_{n=A
(N)}^{B(N)}\Delta W(n)1_{S}(n)\\
 \geq &\lim_{N\to\oo}\frac{1}{W(B(N))-W(A(N))}\sum_{n=A(N)+\lfloor{\epsilon |I_N|\rfloor}}^{B(N)}\Delta W(n)\\
=& \lim_{N\to\oo}\frac{W(B(N))-W(A(N)+{\epsilon |I_N|})}{W(B(N))-W(A(N))}\\
=&1- \lim_{N\to\oo}\frac{W(A(N)+\epsilon |I_N|)-W(A(N))}{W(B(N))-W(A(N))}\\
\geq &1- \lim_{N\to\oo}\frac{\epsilon |I_N|\cdot W'(A
(N))}{W(B(N))-W(A(N))} \geq 1-\epsilon
\end{align*}
by (\ref{eq:MVT}). This holds for each $\epsilon>0$, therefore
\begin{align*}
1=&\lim_{N\to\oo}\frac{1}{W(B(N))-W(A(N))}\sum_{n=A
(N)}^{B(N)}\Delta W(n)1_{S}(n)\\\leq &\limsup_{W(A)-W(B)\to\oo}\frac{1}{W(B)-W(A)}\sum_{n=A
}^{B}\Delta W(n)1_{S}(n)
\end{align*}
which shows that $(iv)\implies(i)$. This completes the proof.
\end{proof}

Taking $W(N)=N$, the notions of $W$-syndetic and $W$-thick specialize to the usual notions of syndetic and thick. By Lemma \ref{lem:W_syndetic_equivalence}, we observe that if $W_1$ grows faster than $W_2$, then for $S\subset \N$
\begin{align*} 
S\text{ is } W_1\text{-syndetic}\implies &S\text{ is } W_2\text{-syndetic}\\
\text{ and} \\
S\text{ is } W_2\text{-thick}\implies &S\text{ is } W_1\text{-thick}.
\end{align*}
These implications make it clear that any syndetic set is $W$-syndetic for every $1\prec W(x)\prec x$ and any set which is $W$-thick for some $1\prec W(x)\prec x$ must be thick in the usual sense.

Taking $h_1=\cdots = h_{\ell} = 1_{A}$, Theorem \ref{thm:uniform_main} gives the following.
\begin{maintheorem}\label{thm:W_syndetic_main}
    Let $W,f_1,\dots, f_{\ell}$ be functions which satisfy the conditions of Theorem \ref{thm:uniform_main}. Let $(X,\mathscr{B},\mu,T)$ be an invertible measure preserving system and let $A\in \mathscr{B}$ with $\mu(A)>0$.
     \begin{enumerate}[label = (\roman*)]
        \item Suppose that $\Poly(f_1,\dots, f_{\ell})\cap \Z[x] = \{0\}$. Then for each $\epsilon>0$, the set 
        \begin{equation}
            \{n\in \N:\mu(A\cap T^{-\round{f_1(n)}}A\cap\cdots \cap T^{-\round{f_{\ell}(n)}}A)>\mu(A)^{k+1}-\epsilon\}
         \end{equation}
         is $W$-syndetic.
         \item Suppose that $\Poly(f_1,\dots, f_{\ell})$ is jointly intersective.
         Then the set
        \begin{equation}
            \{n\in \N:\mu(A\cap T^{-\round{f_1(n)}}A\cap\cdots \cap T^{-\round{f_{\ell}(n)}}A)>0\}
         \end{equation}
         is $W$-syndetic.
     \end{enumerate}
\end{maintheorem}
Observe that Theorem \ref{thm:general_thick_szemeredi} is precisely part (ii) of Theorem \ref{thm:W_syndetic_main}. Next, we show how Theorem \ref{thm:polynomial_sz} and Theorem \ref{thm:thick_szemeredi} follow from Theorem \ref{thm:general_thick_szemeredi}.

\begin{proof}[Proof of Theorem \ref{thm:polynomial_sz}]
    Suppose that $P$ is not jointly intersective. Pick $r\in \N$ such that there is no $n\in \N$ with $p_i(n)\equiv 0\mod r$ for all $i\in \{1,\dots, \ell\}$. Let $X = \{0,1,\dots, r-1\}$, let $\mathscr{B} = \mathscr{P}(X)$, let $\mu$ be the normalized counting measure on $X$, and let $Tx = x+1\mod r$ for $x\in X$. Taking $A = \{0\}$, it is clear that $R_A(p_1,\dots, p_{\ell})$ is empty.

    Now we suppose that $P$ is jointly intersective and that $(X,\mathscr{B},\mu,T)$ is an invertible measure preserving system. We will show that $R_A(p_1,\dots, p_{\ell})$ is syndetic. Let $W$ be any function which belongs to a Hardy field and satisfies $\log(x)\prec  W(x)\prec x$, and let $f_i = p_i$ for all $i\in \{1,\dots, \ell\}$. Let $f\in \text{Span}_{\mathbb{R}}\{f_1,\dots, f_{\ell}\}$ and note that $f\in \Q[x]$ and so $\deg^*(f)=0$. Then $W$ is compatible with $f$ and hence property (\ref{eq:W_compatible}) is satisfied. By Theorem \ref{thm:general_thick_szemeredi}, $R_A(p_1,\dots, p_{\ell})$ is $W$-syndetic. This holds for any function $W$ which belongs to a Hardy field and satisfies $\log(x)\prec \log W(x)\prec x$ (and it is known that we can take $\log W(x)$ to have growth rate arbitrarily close to $x$, see for instance \cite[Theorem A]{ADH24}) so by Lemma \ref{lem:W_syndetic_equivalence} we have that $\liminf_{k\to\oo}\frac{|R_{A}(p_1,\dots, p_{\ell})\cap I_k|}{|I_k|}>0$ for any sequence of intervals $(I_k)_{k\in \N}$ with $\lim_{k\to\oo}|I_k|=\oo$. It follows that $R_{A}(p_1,\dots, p_{\ell})$ is syndetic, since $|R_{A}(p_1,\dots, p_{\ell})\cap I|>0$ for any long enough interval $I$.  
\end{proof}

\begin{proof}[Proof of Theorem \ref{thm:thick_szemeredi}]
   The case $d=1$ immediately reduces to the ergodic Szemer\'edi Theorem. Indeed, in this case the set $\{\round{f(n)}:n\in \N\}$ contains all but finitely many elements of $\N$ and for each $m\in \N$ and each sufficiently large $N\in\N$ there is an $n$ such that $\round{f(n)},\round{f(n+1)},\dots, \round{f(n+m)}$ are each equal to $N$. So, it suffices to consider the case $d>1$. Let $m\in \N$, let $(X,\mathscr{B},\mu,T)$ be an invertible measure preserving system and let $A\in \mathscr{B}$ with $\mu(A)>0$. We will show that 
   \begin{equation}\label{eq:thick_sz_proof}
        R_A(f_1,\dots, f_{\ell})=\{n:\mu(A\cap T^{-[ f(n)]}A\cap T^{-2[ f(n)]}A\cap \cdots\cap T^{-\ell[ f(n)]}A )>0\}
   \end{equation}
   contains an interval of the form $\{n-m,\dots, n\}$. Consider the collection of functions $\mathscr{F} = \{n\mapsto i\cdot f(n-j): 1\leq i\leq \ell, 0\leq j\leq m\}$. We will show that $\Poly(\mathscr{F}) = \{0\}$.

   Recall the identity $f(n-i) = (1-\Delta)^if(n)$, which makes it clear that \begin{equation}
       \text{Span}_{\mathbb{R}}(\mathscr{F}) \subset  \text{Span}_{\mathbb{R}}\{\Delta^if:i\geq 0\}.
   \end{equation} 
   Recall that $x^{d-1}\prec f\prec x^d$. It follows that for each $i$, $\Delta^if$ tends to either $0$ or $\pm\oo$ and so each nonzero element of $\text{Span}_{\mathbb{R}}\{\Delta^if:i\geq 0\}$ tends to either $0$ or $\pm \oo$. 
   
   Suppose for the sake of contradiction that there exists a nonzero $p(x)\in \R[x]$ and a $g\in \text{Span}_{\mathbb{R}}\{\Delta^if:i\geq 0\}$ such that $\lim_{x\to\oo}|g(x)-p(x)| = 0$. We know that $\lim_{x\to\infty}g(x) \in \{0,\pm\oo\}$, but $p(x)$ is a nonzero polynomial and so $\lim_{x\to\infty}g(x)=\lim_{x\to\infty}p(x)\neq 0$. It follows that $\lim_{x\to\oo}p(x)=\pm\oo$.

   Pick $K$ such that $\Delta^Kp(x)$ is constant and apply the Stolz-\Cesaro theorem (see Theorem \ref{thm:stolz_cesaro} in Section \ref{section:preliminaries}) $K$ times to see that
   \begin{equation}\label{eq:K_stolz_cesaro}
       1 = \lim_{x\to\oo}\frac{g(x)}{p(x)}=\lim_{x\to\oo}\frac{\Delta g(x)}{\Delta p(x)}=\cdots =\lim_{x\to\oo}\frac{\Delta^K g(x)}{\Delta^K p(x)}.
   \end{equation}
   We know that $\Delta^K g(x) \in  \text{Span}_{\mathbb{R}}\{\Delta^if:i\geq 0\}$ and so $\lim_{x\to\oo}|\Delta^K g(x)| \in \{0,\oo\}$ which contradicts (\ref{eq:K_stolz_cesaro}) and the fact that $\Delta^Kp(x)$ is constant. Therefore, we have shown that $\Poly(\mathscr{F}) \subset \Poly(\{\Delta^if:i\geq 0\})= \{0\}$, which is jointly intersective. Next, consider the following claim.

   \textbf{Claim:} Let $\epsilon>0$ be arbitrarily small and let $\alpha\in (\ell\epsilon, 1/\ell)$. Suppose that for some $n\in \N$, each of $\round{f(n)}\alpha \text{ mod 1},\round{2f(n)}\alpha \text{ mod 1},\dots, \round{\ell f(n)}\alpha \text{ mod 1}$ are contained in $[0,\epsilon)$. Then $f(n)\mod 1\in [0,1/2\ell)\cup [1-1/2\ell,1)$ and hence $\round{if(n)} = i \round{f(n)}$ for $i=1,2,\dots, \ell$.

   To see why this claim is true, observe that if $f(n)\mod 1\not\in [0,1/2\ell)\cup [1-1/2\ell,1)$ then we would have $\round{kf(n)} = k\round{f(n)}\pm i$ for some $i,k\in \{1,\dots, \ell\}$, which would mean that $\round{kf(n)}\alpha \text{ mod 1}\not\in [0,\epsilon)$ by choice of $\alpha$.

   Now let $\epsilon>0$ be arbitrarily small and define an invertible measure preserving system $(X',\mathscr{B}',\mu',T')$ by $X' = X\times \R/\Z$ with induced sigma algebra $\mathscr{B}'$ and induced measure $\mu'$ and put $T'(x,y) = (Tx,y+\alpha\mod 1)$, where $\alpha$ is given by the statement of the claim. Let $A' = A\times [0,\epsilon)$. Pick $W\in \mathcal{H}$ satisfying $W'(N) = (\Delta^{d}f(N))^{1/d}$, so that $\log x\prec \log W(x)$ since $d\geq 2$. Now, we apply Theorem \ref{thm:general_thick_szemeredi} to the system $(X',\mathscr{B}',\mu',T')$ to see that the set 
   \begin{equation}
       R=\left\{n\in \N:\mu'\left(A'\cap \left(\bigcap_{\substack{1\leq i\leq \ell\\ 0\leq j\leq m}}T^{-\round{if(n-j)}}A'\right)\right)>0\right\}
   \end{equation}
   is nonempty. Additionally, note that $R$ is contained in $R_1\cap R_2$ where  
   \begin{equation}
       R_1=\left\{n\in \N:\mu\left(A\cap \left(\bigcap_{\substack{1\leq i\leq \ell\\ 0\leq j\leq m}}T^{-\round{if(n-j)}}A\right)\right)>0\right\}
   \end{equation}
   and
   \begin{equation}
       R_2=\left\{n\in \N:\mu_{\R/\Z}\left([0,\epsilon)\cap \left(\bigcap_{\substack{1\leq i\leq \ell\\ 0\leq j\leq m}}T^{-\round{if(n-j)}}[0,\epsilon)\right)\right)>0\right\}
   \end{equation}
   (here $\mu_{\R/\Z}$ is the induced Lebesgue measure on $\R/\Z$). Pick any $n\in R\subset R_1\cap R_2$. Since $n\in R_2$ we have $\round{if(n-j)} = i\cdot \round{f(n-j)}$ for any $1\leq i\leq \ell$ and $0\leq j\leq m$ by the above claim, and since $n\in R_1$ we have
   \begin{equation}
       \mu(A\cap T^{-\round{f(n-j)}}A\cap T^{-2\round{f(n-j)}}A\cap\cdots \cap T^{-\ell\round{f(n-j)}}A)>0
   \end{equation}
   for each $0\leq j\leq m$. This shows that  $R_A(f_1,\dots, f_{\ell})$ contains an interval of the form $\{n-m,\dots, n\}$, which concludes the proof.
\end{proof}

We may adapt the proof of Theorem \ref{thm:thick_szemeredi} in order to obtain the main results of \cite{Bergelson_Moreira_Richter_2020} with the added assumptions that the function belongs to a Hardy field and is \emph{tempered}, meaning that there exists $d\in \N$ such that $f^{(d-1)}$ tends to $0$ and $\lim_{x\to\oo}x\cdot f^{(d)}(x) = \oo$.
\begin{theorem}[cf. {\cite[Theorems A,B,D]{Bergelson_Moreira_Richter_2020}}]\label{thm:BMR_1}
    Suppose that $f$ is a tempered function which belongs to a Hardy field. Let $p_1(x),\dots, p_k(x)\in \Z[x]$ and put $W = \Delta^{\deg(f)-1}f$. Then for any invertible measure preserving system $(X,\mathscr{B},\mu,T)$
    \begin{itemize}
        \item For any $h_1,\dots, h_{\ell}\in L^{\oo}(X)$, the limit 
        \begin{equation}
            \E_{\text{unif}}^WT^{-\round{p_1(\Delta)f(n)}}h_1\cdots T^{-\round{p_k(\Delta)f(n)}}h_{k}
        \end{equation}
        exists in $L^2(X)$.
        \item For any $A\in \mathscr{B}$ with $\mu(A)>0$,
        \begin{equation}
            \E_{\text{unif}}^W
           \mu(A\cap T^{-\round{p_1(\Delta)f(n)}}A\cdots T^{-\round{p_k(\Delta)f(n)}}A)>0.
        \end{equation}
        \item For any $A\in \mathscr{B}$ with $\mu(A)>0$,
        \begin{equation}
           \{n\in \N: \mu(A\cap T^{-\round{p_1(\Delta)f(n)}}A\cdots T^{-\round{p_k(\Delta)f(n)}}A)>0\}
        \end{equation}
        is thick and $W$-syndetic.
    \end{itemize}
\end{theorem}
\begin{proof}
    Let $d = \deg(f)$. It is clear that $W=\Delta^{d
    -1}f$ satisfies equation (\ref{eq:thm_B_condition}). Additionally, since $f$ is tempered, we know that $x^{d-1}\log(x)\prec f(x)\prec x^{d}$ and so $\log(x)\prec W(x)$. 
   Let $\mathscr{F} =  \{n\mapsto p_i(\Delta)f(n-j):1\leq i\leq k, 0\leq j\leq m\}$ and recall from the proof of Theorem \ref{thm:thick_szemeredi} that $
   \mathscr{F}\subset\text{Span}_{\mathbb{R}}\{\Delta^{i}f:i\geq 0\}$ and $\Poly( \mathscr{F})\subset \Poly(\{\Delta^{i}f:i\geq 0\}) = \{0\}$, which is jointly intersective. Then each conclusion in the theorem statement follows from Theorem \ref{thm:uniform_main} and Theorem \ref{thm:W_syndetic_main}.
\end{proof}

Using a standard argument involving Furstenberg's correspondence principle, we will transform our ergodic results into combinatorial results.
\begin{theorem}[Furstenberg’s correspondence principle]\label{thm:Furstenberg_correspondence}
    For any $E\subset \N$ with 
    $$
    \overline{d}(E):=\limsup_{N\to\oo}\E_{n\leq N}1_E(n) > 0
    $$
    there exists an invertible measure preserving system $(X,\mathscr{B},\mu,T)$ and a set $A \in\mathscr{B}$ with $\mu(A) = \overline{d}(E)$ such that for all $n_1,\dots,n_{\ell}\in \Z$
    \[
    \overline{d}(E\cap(E-n_1)\cap\dots \cap (E-n_{\ell})) \geq \mu(A\cap T^{-n_1}A\cap \dots \cap T^{-n_{\ell}}A).
    \]
\end{theorem}
From Furstenberg correspondence and Theorem \ref{thm:W_syndetic_main}, we obtain the following.
\begin{maintheorem}
  Let $W,f_1,\dots, f_{\ell}$ be functions which satisfy the conditions of Theorem \ref{thm:main}. Suppose that $\{f_1,\dots, f_{\ell}\}$ satisfies property (\ref{eq:W_compatible}). Let $E\subset \N$ be a set with $\overline{d}(E)>0$.
     \begin{enumerate}[label = (\roman*)]
        \item Suppose that the conditions of Theorem \ref{thm:main}\ref{condition_1} are satisfied. Then 
        \begin{equation}
            \E_{\text{unif}}^W(\overline{d}(E\cap (E-\round{f_1(n)})\cap \cdots \cap (E-\round{f_{\ell}(n)})) = \overline{d}(E)^{k+1}.
        \end{equation}
        In particular, for each $\epsilon>0$, the set 
        \begin{equation}
            \{n:\overline{d}(E\cap (E-\round{f_1(n)})\cap \cdots \cap (E-\round{f_{\ell}(n)})>\overline{d}(E)^{k+1}-\epsilon\}
        \end{equation}
        is $W$-syndetic. This means that there are ``many" values of $a,n\in \N$ such that $\{a, a+\round{f_1(n)},\dots,a+\round{f_{\ell}(n)}\}\subset E$.
         \item Suppose that the conditions of Theorem \ref{thm:main}\ref{condition_2} are satisfied. Then 
        \begin{equation}
            \E_{\text{unif}}^W(\overline{d}(E\cap (E-\round{f_1(n)})\cap \cdots \cap (E-\round{f_{\ell}(n)})) >0.
        \end{equation}
        In particular,
        \begin{equation}
            \{n:\overline{d}(E\cap (E-\round{f_1(n)})\cap \cdots \cap (E-\round{f_{\ell}(n)})>0\}
        \end{equation}
        is $W$-syndetic. Again, this means that there are ``many" values of $a,n\in \N$ such that $\{a, a+\round{f_1(n)},\dots,a+\round{f_{\ell}(n)}\}\subset E$. 
     \end{enumerate}
\end{maintheorem}

\section{Weighted Uniform Distribution}\label{section:uniform_distribution}

In Sections \ref{section:BMR_results} and \ref{section:richter_results} we reduce the proof of Theorem \ref{thm:main} to a statement about uniform distribution modulo 1, Theorem \ref{thm:W_ud} below. In this section, we state and prove Theorem \ref{thm:W_ud} and give applications independent of its role in proving Theorem \ref{thm:main}.

Recall that a sequence $(x_n)_{n\in \N}\subset \R$ is \emph{uniformly distributed modulo 1} (or u.d. mod 1) if 
\begin{equation}\label{eq:u.d._def}
\lim_{N\to\infty}\frac{1}{N}\sum_{n=1}^NF(x_n\text{ mod }1)= \int_{[0,1)} F
\end{equation}
for each continuous function $F\in C([0,1))$, where $x\text{ mod }1$ is the fractional part of $x\in \R$, $x\text{ mod }1 := x-\floor{x}$. The Weyl criterion for uniform distribution states that a sequence $(x_n)_{n\in \N}$ is u.d. mod 1 if and only if
\begin{equation}\label{eq:u.d._def_2}
\lim_{N\to\infty}\frac{1}{N}\sum_{n=1}^Ne^{2\pi i kx_n}= 0
\end{equation}
for all nonzero $k\in \mathbb{Z}$. In \cite{Boshernitzan}, Boshernitzan gave a criterion for $(f(n))_{n\in \N}$ to be u.d. mod 1 when $f$ is a subpolynomial Hardy function.
\begin{theorem}[{\cite[Theorem 1.3]{Boshernitzan}}]\label{thm:Boshernitzan_ud}
    Suppose that $f$ is a subpolynomial Hardy function. Then the following are equivalent.
    \begin{enumerate}[label=(\arabic*)]
        \item $(f(n))_{n\in \N}$ is u.d. mod 1,
        \item $\lim_{x\to\oo}\frac{|f(x)-p(x)|}{\log(x)}=\oo$ for all $p(x)\in \Q[x]$.
    \end{enumerate}
\end{theorem}
For example, $(f(n))_{n\in \N}$ is u.d. mod 1 when $f$ is any function of the form $f(x) = \alpha x^c$ for $\alpha\in \R\setminus \Q$, $c>0$. When $c\in \N$ much more is true; $(f(n))_{n\in \N}$ is \emph{well distributed modulo 1} (w.d. mod 1), namely
    \begin{equation}
    \lim_{N-M\to\infty}\frac{1}{N-M}\sum_{n=M}^Ne^{2\pi i kf(n)}= 0 \text{ for all nonzero }k\in \mathbb{Z}.
\end{equation}
In fact, the only Hardy functions $f$ such that $(f(n))_{n\in \N}$ is w.d. mod 1 are those of the form $f(x) = \alpha x^n+q(x)+o_{x\to\oo}(x^n)$ for some $\alpha\in \R\setminus \Q, n\in \N, q(x)\in \Q[x]$.
\begin{theorem}[{\cite[Theorem 1.10]{Boshernitzan}}]\label{thm:Boshernitzan_wd}
    Suppose that $f$ is a subpolynomial Hardy function. Then the following are equivalent.
    \begin{enumerate}[label=$(\arabic*)'$]
        \item $(f(n))_{n\in \N}$ is w.d. mod 1,
        \item There exists $q(x)\in \Q[x]$ and $m\in \N$ such that $\lim_{x\to\oo}\frac{|f(x)-q(x)|}{x^{m}}$ is finite and irrational.
    \end{enumerate} 
\end{theorem}

\begin{remark}
    In \cite{Boshernitzan}, the above theorem is stated but only the implication $(2)'\implies (1)'$ is proven. The forward implication is incorrectly cited as being contained in \cite{Boshernitzan_wrong_citation} and it is likely that the correct citation is the preprint \cite{Boshernitzan_preprint}, which was never published. The methods contained in \cite{Boshernitzan_preprint} are largely disjoint from the methods that we consider in this paper, as Boshernitzan uses the existence of Hardy functions which tend to infinity very slowly to show that when condition $(2)'$ does not hold, the sequence $(f(n),f(n+1),\dots, f(n+k))_{n\in \N}$ is dense modulo 1 in $[0,1]^k$ for any $k\in \N$. The authors are unaware of any full proof of Theorem \ref{thm:Boshernitzan_wd} currently contained in published literature.

\end{remark}

This theorem demonstrates that some Hardy functions have ``better" uniform distribution properties than others. Presently, we characterize Hardy functions by their quality of uniform distribution. 
\begin{definition}
Let $(x_n)_{n\in \N}\subset \R$. We say that $(x_n)$ is u.d. mod 1 with respect to $W$-averages if $\lim_{N\to\oo}\E_{n\leq N}^We^{2\pi i kx_n} = 0$ for all nonzero $k\in \Z$. 
\end{definition}
\begin{maintheorem}\label{thm:W_ud}
     Let $\mathcal{H}$ be a Hardy field. Let $f\in \mathcal{H}$ be subpolynomial and let $W\in \mathcal{H}$ satisfy $1\prec \log W(x)\prec x$. Consider the following statements.
     \begin{enumerate}[label=(\roman*)]
         \item $\deg^*(f)>0$ and $W$ is compatible with $f$ (Definition \ref{def:W_compatible}).
         \item $(f(n))_{n\in \N}$ is u.d. mod 1 with respect to $W$-averages.
     \end{enumerate}
     Then (i)$\implies $(ii). Additionally, when $\deg^*(f)\leq 2$ we also have (ii)$\implies $(i).
\end{maintheorem}

\begin{remark}
Let $\mathcal{H}$ be a Hardy field and let $f,W\in\mathcal{H}$. Put $d= \deg^*(f)$ and suppose that $d>0$. Then $W$ is compatible with $f$ if $\lim_{x\to\oo}\frac{|f^{(d)}(x)-p(x)|^{1/d}}{(\log W)'(x)}=\oo$ for all $p(x)\in \Q[x]$. 
Taking $W(x)=x$ we find that $W$ is compatible with $f$ if and only if $\lim_{x\to\oo}\frac{|f^{(d)}(x)-p(x)|}{x^{-d}}=\oo$ for all $p(x)\in \Q[x]$. After repeated applications of L'H\^opital's rule, this becomes $\lim_{x\to\oo}\frac{|f(x)-p(x)|}{\log(x)}=\oo$ for all $p(x)\in \Q[x]$. So, Theorem \ref{thm:Boshernitzan_ud} is a special case of Theorem \ref{thm:W_ud}. 
Additionally, \cite[Theorem 1.6]{BKS2025} (see also \cite[Theorem 5.1]{Richter23}) shows that conditions (i) and (ii) of Theorem \ref{thm:W_ud}, along with several other statements, are equivalent but contains the added assumption that $W'$ is nonincreasing.
\end{remark}
\begin{remark}
    It is natural to conjecture that the implication (ii)$\implies$(i) holds for $\deg^*(f)>2$, but our methods are unable to treat this case.
\end{remark}

The following corollary is immediate from Theorem \ref{thm:W_ud} and Theorem \ref{thm:mikey_thm}.

\begin{corollary}\label{cor:gen_wd_condition}
     Let $f$ be a subpolynomial Hardy function. Put $d=\deg^*(f)$ and assume $d>0$.
Suppose that $s:\N\rightarrow\N$ is a nondecreasing  function satisfying $s(N)\leq N-1$ for all sufficiently large $N\in \N$ and
\begin{equation}\label{eq:gen_wd_condition_1}
         \lim_{N\to\oo}s(N)\cdot |f^{(d)}(N)-p(N)|^{1/d}=\oo
\end{equation}
         for each $p(x)\in \Q[x]$. Then
         \begin{equation}\label{eq:gen_wd_condition_2}
         \lim_{N\to\oo}\frac{1}{s(N)}\sum_{n=N-s(N)}^Ne^{2\pi i kf(n)} = 0
         \end{equation}
        for all nonzero $k\in \mathbb{Z}$.

\end{corollary}
\begin{proof}
Pick $q(x)\in \Q[x]$ such that $d=\deg^*(f)=\deg(f-q)$. Then 
\begin{equation} \label{eq:slowest_q}
|f^{(d)}(x)-q^{(d)}(x)|\preceq |f^{(d)}(x)-p(x)|
\end{equation}
for all $p(x)\in \Q[x]$. Let $\mathcal{H}$ be a maximal Hardy field containing $f$. Pick $W\in \mathcal{H}$ such that 
\begin{equation}\label{eq:same_growth_as_f^d}
    \lim_{x\to\oo}\frac{|f^{(d)}(x)-q^{(d)}(x)|}{(\log W)'(x)}=1.
\end{equation}
For example, take 
\begin{equation}\label{eq:example_hardy}
W(x) = \exp\left(\int_0^x |f^{(d)}(t)-q^{(d)}(t)|~dt\right)
\end{equation}
which is contained in $\mathcal{H}$ since maximal Hardy fields are closed under integration and exponentiation. 
For any $U\in \mathcal{H}$ with $1\prec \log U(x)\prec \log W(x)$, $f$ is compatible with $U$ by (\ref{eq:same_growth_as_f^d}) and (\ref{eq:slowest_q}), and moreover $\lim_{N\to\oo}s(N)\cdot (\log U)'(N) = \oo$ by (\ref{eq:gen_wd_condition_1}) and (\ref{eq:slowest_q}). By Theorem \ref{thm:W_ud}, 
\[
\lim_{N\to\oo}\E_{n\leq N}^Ue^{2\pi i k f(n)}=0
\]
for any nonzero $k\in \Z$ and any $U\in \mathcal{H}$ with $1\prec \log U(x)\prec \log W(x)$, and hence
\[
\lim_{N\to\oo}\frac{1}{s(N)}\sum_{n=N-s(N)}^N e^{2\pi i k f(n)}=0
\]
for any nonzero $k\in \Z$. This concludes the proof.
\end{proof}

\begin{remark}
In the case $\deg^*(f)\leq 2$, we can note that Theorem \ref{thm:Boshernitzan_wd} follows from Theorem \ref{thm:mikey_thm} and Corollary \ref{cor:gen_wd_condition}. To see why this is true, observe that  $(f(n))_{n\in \N}$ is w.d. mod 1 if and only if equation (\ref{eq:gen_wd_condition_2}) holds for all $s:\N\rightarrow \N$ with $\lim_{N\to\oo}s(N)=\oo$ and $s(N)\leq N-1$ for all sufficiently large $N\in \N$. Now, if $f$ satisfies condition $(2)'$ in Theorem \ref{thm:Boshernitzan_wd} and $\lim_{N\to\oo}s(N)=\oo$ then equation (\ref{eq:gen_wd_condition_1}) automatically holds because $f^{(d)}$ does not tend to $0$.

For the reverse direction, suppose condition $(2)'$ in Theorem \ref{thm:Boshernitzan_wd} does not hold. Then there is a Hardy function $W$ such that $f$ and $W$ are contained in the same maximal Hardy field, such that 
\[
\lim_{x\to\oo}\frac{|f^{(d)}(x)-p(x)|^{1/d}}{(\log W)'(x)}<\oo
\]
for some $p(x)\in \Q[x]$ (take for example, the function $W$ in (\ref{eq:example_hardy})). Then for this function $W$ we have that $(f(n))_{n\in \N}$ is not u.d. mod 1 with respect to $W$-averages and hence by Theorem \ref{thm:mikey_thm} there is a function $s$ with $\lim_{N\to\oo}s(N)=\oo$ for which $(\ref{eq:gen_wd_condition_2})$ does not hold.
\end{remark}

\begin{example}
Let $f(x) = x^{3/2}$ and let $s$ satisfy $\lim_{N\to\oo}\frac{s(N)}{N^{1/4}}=\oo$, for example $s(N)=N^{1/4+\epsilon}$ for some $\epsilon>0$. Then equation (\ref{eq:gen_wd_condition_1}) holds with $d=\deg^*(f)=2$. By the usual proof of the Weyl Criterion, equation (\ref{eq:gen_wd_condition_2}) implies that 
\begin{equation}\label{eq:example_2}
         \lim_{N\to\oo}\frac{1}{s(N)}\sum_{n=N-s(N)}^N1_{(a,b)}(f(n)\mod 1) = b-a
\end{equation}
         for all $(a,b)\subset [0,1]$, where $f(n)\mod 1 = f(n)-\lfloor f(n)\rfloor$ denotes the fractional part of $f(n)$. It follows that for all large enough $N$, there is an $n\in [N-s(N),N]$ such that $f(n)\mod 1\in (a,b)$. This improves \cite[Example 1.10]{reilly26}, and  
         moreover, this is the best possible result of this form since Theorem \ref{thm:W_ud} says that (\ref{eq:example_2}) does not hold if $s(N)$ grows like $N^{1/4}$ or slower. 
\end{example}

In the remainder of this section, we prove Theorem \ref{thm:W_ud}. We first treat the case $\deg^*(f)=1$, then we prove the forward implication when $\deg^*(f)=2$, and then the converse implication.

\subsection{The case deg\texorpdfstring{$^*(f)=1$}{*(f)=1}}\label{section:d=1}
\begin{lemma}\label{lem:deg=1}
    Suppose that $f$ and $W$ are continuously differentiable and eventually monotone functions such that $1\prec \log W(x)\prec x$, $\lim_{x\to\oo}f'(x)=0$, and the limit $\lim_{x\to\oo}\frac{f'(x)}{(\log W)' (x)}$ exists in $(0,\oo)$. Then there is a constant $C\in \C$ with $|C|<1$ such that 
    \begin{equation}
        \E_{n\leq N}^W(e^{2\pi i f(n)}) = C\cdot e^{2\pi i f(N)}+o_{N\to\oo}(1).
    \end{equation}
\end{lemma}
\begin{proof}

We begin by considering the special case $f(x) = c\cdot \log W(x)$ for some $c\in (0,\oo)$.
    \begin{align*}
         \E_{n\leq N}^W(e^{2\pi i 
        f(n)})=&\E_{n\leq N}^W(e^{2\pi i 
        c\log W(n)}) = \frac{1}{W(N)}\sum_{n=1}^{N}\Delta W(n)e^{2\pi i c\log W(n)}\\ = &e^{2\pi i c\log W(N)}\cdot \sum_{n=1}^{N}\frac{\Delta W(n)}{W(N)}e^{2\pi i c\log\left(\frac{W(n)}{W(N)}\right)}.
    \end{align*}
The sum $\sum_{n=1}^{N}\frac{\Delta W(n)}{W(N)}e^{2\pi i c\log\left(\frac{W(n)}{W(N)}\right)}$ is a Riemann sum with partition $\{0<\frac{W(1)}{W(N)}<\frac{W(2)}{W(N)}<\cdots <\frac{W(N)}{W(N)}=1\}$ for the integral $\int_0^1e^{2\pi i c\log(x)}dx$. Therefore
\begin{equation}\label{eq:riemann_sum}
\sum_{n=1}^{N}\frac{\Delta W(n)}{W(N)}e^{2\pi i c\log\left(\frac{W(n)}{W(N)}\right)}=\int_{0}^1e^{2\pi i c\log(x)}dx = \int_0^1x^{2\pi i c}dx = \frac{1}{1+2\pi i c}.
\end{equation}
It is worth noting that equation (\ref{eq:riemann_sum}) is precisely where we use the assumption that $\log W(x) \prec x$ (or equivalently that $\lim_{x\to\oo}f'(x)=0$), since otherwise $\frac{W(N)-W(N-1)}{W(N)}$ would not tend to $0$ and so the Riemann sum would not tend to the integral. Taking $C = \frac{1}{1+2\pi i c}$ we have 
   \[
   \E_{n\leq N}^W(e^{2\pi i f(n)}) = C\cdot e^{2\pi i f(N)}+o_{N\to\oo}(1)
   \]
as desired. For the general case, we have that
\begin{equation} \label{eq:f'_asymptotics}
f'(x) = c\cdot (\log W)'(x)+E(x)\cdot (\log W)'(x)
\end{equation} 
for some $c\in (0,\oo)$ and some function $E$ with $\lim_{x\to\oo}E(x)=0$. Let $\epsilon>0$ and let $N\in \N$ be arbitrarily large. Pick the smallest $N_0\in \N$ such that $\frac{W(N_0)}{W(N)}>\epsilon/2$ and note that $N_0$ tends to $\oo$ as $N$ tends to $\oo$. Then 
\begin{equation}\label{eq:remove_epsilon_weight}
\left|\frac{1}{W(N)}\sum_{n=1}^{N_0}\Delta W(n)e^{2\pi i f(n)}\right|\leq \frac{1}{W(N)}\sum_{n=1}^{N_0}\left|\Delta W(n)\right|<\epsilon.
\end{equation}
Integrating both sides of equation (\ref{eq:f'_asymptotics}) gives 
\[
\int_n^Nf'(x)dx = f(N)-f(n)
\]
and 
\[
\int_{n}^Nc\cdot (\log W)'(x)+E(x)\cdot (\log W)'(x)~dx = c\log\left(\frac{W(N)}{W(n)}\right) + \int_n^NE(x)(\log W)'(x)~dx. 
\]
Let $u = \log W(x)$ and define the function $\tilde{E}$ by $\tilde{E}(t) = E(W^{-1}(e^t))$. Then $\tilde{E}(t)\to 0$ as $t\to\oo$ and 
\[
\int_n^NE(x)(\log W)'(x)~dx = \int_{\log W(n)}^{\log W(N)}\tilde{E}(u)~du.
\]
For $n\in [N_0,N]$, we have 
\[
\log W(N)-\log W(n) \leq \log W(N)-\log W(N_0) =  \log\left(\frac{W(N)}{W(N_0)}\right)
\]
which is bounded uniformly in $N$ by our assumption on $N_0$. Since $\tilde{E}$ tends to $0$, it follows that $\int_{\log W(n)}^{\log W(N)}\tilde{E}(u)~du = o_{N\to\oo}(1)$ uniformly for $n\in [N_0,N]$. Altogether, we have 
\begin{equation}\label{eq:f_difference_asymptotic}
f(n)-f(N) = c\cdot \log\left(\frac{W(n)}{W(N)}\right)+o_{N\to\oo}(1)
\end{equation}
uniformly for $n\in [N_0,N]$. Lastly, using equations (\ref{eq:remove_epsilon_weight}), (\ref{eq:f_difference_asymptotic}), and (\ref{eq:f'_asymptotics}) we have
\begin{align*}
    \E_{n\leq N}^W(e^{2\pi i f(n)}) &= \frac{1}{W(N)}\sum_{n=1}^{N}\Delta W(n)e^{2\pi i f(n)}\\ =& e^{2\pi i f(N)}\cdot \frac{1}{W(N)}\sum_{n=N_0}^{N}\Delta W(n)e^{2\pi i (f(n)-f(N))}+O(\epsilon)\\
    =&e^{2\pi i f(N)}\cdot \frac{1}{W(N)}\sum_{n=N_0}^{N}\Delta W(n)e^{2\pi i c\log(W(n)/W(N))}+o_{N\to\oo}(1)+O(\epsilon)\\
    =&e^{2\pi i f(N)}\cdot \frac{1}{W(N)}\sum_{n=1}^{N}\Delta W(n)e^{2\pi i c\log(W(n)/W(N))}+o_{N\to\oo}(1)+O(\epsilon)\\
    =&C\cdot e^{2\pi i f(N)}+o_{N\to\oo}(1)+O(\epsilon)
\end{align*}
for $C= \frac{1}{1+2\pi i c}$. Taking $\epsilon\to 0$ completes the proof.
\end{proof}
We now prove Theorem \ref{thm:W_ud} in the case $\deg^*(f)=1$.
\begin{theorem}\label{thm:d=1_thm}
    Let $\mathcal{H}$ be a maximal Hardy field and let $W,f\in \mathcal{H}$. Suppose that $\deg^*(f)=1$ and that $\log x\prec \log W(x)\prec x$. Then $(f(n))_{n\in \N}$ is u.d. mod 1 with respect to $W$-averages if and only if $W$ is compatible with $f$, meaning that 
\begin{equation}\label{eq:d=1_bosh_condition_again}
         \lim_{x\to\oo}\frac{|f'(x)-p(x)|}{(\log W)'(x)}=\oo
\end{equation}
for all $p(x)\in \Q[x]$. 
\end{theorem}
\begin{proof}

We can replace $f$ by $-f$ if necessary to assume that $f$ tends to $\oo$. For $q(x)\in \Q[x]$ the sequence $(q(n)\mod 1)_{n\in \N}$ is periodic and so replacing $f$ by $f-q$ does not affect uniform distribution (see lemma \ref{lem:variant_of_modulo_averages}) and so we assume that $1\prec f(x)\preceq x$. Lastly, we also assume that $f(x)\prec x$ since if 
\[
\lim_{x\to\oo}\frac{f(x)-p(x)}{x^n}\in \R\setminus \Q
\]
for some $p(x)\in \Q[x]$, $n\in \N$ then (\ref{eq:d=1_bosh_condition_again}) holds. We know that $\E_{\unif}(e^{2\pi i kf(n)})_{n\in \N}=0$ for all nonzero $k\in \Z$ by Theorem \ref{thm:Boshernitzan_wd} (this is the direction which is proven in \cite[Theorem 1.10]{Boshernitzan}), and so $\lim_{N\to\oo}\E^W_{n\leq N}e^{2\pi i kf(n)}=0$ for all nonzero $k\in \Z$ by Lemma \ref{lem:wd_means_W_ud} below.

For each $p(x)\in \Q[x]$, $\lim_{x\to\oo}\frac{f'(x)}{(\log W)'(x)}\leq \lim_{x\to\oo}\frac{|f'(x)-p(x)|}{(\log W)'(x)}$, so it suffices to consider the case $p(x)=0$ in (\ref{eq:d=1_bosh_condition_again}). Since $f$ and $W$ belong to the same Hardy field, the limit $\lim_{x\to\oo}\frac{f'(x)}{(\log W)'(x)}$ always exists in $[0,\oo)\cup\{\oo\}$. We have three cases to consider
    \begin{enumerate}[label = (\arabic*)]
        \item $\lim_{x\to\oo}\frac{f'(x)}{(\log W)'(x)}\in(0,\oo)$,
        \item $\lim_{x\to\oo}\frac{f'(x)}{(\log W)'(x)}=0$,
        \item $\lim_{x\to\oo}\frac{f'(x)}{(\log W)'(x)}=\oo$.
    \end{enumerate}

    In the first case, (\ref{eq:d=1_bosh_condition_again}) does not hold and we may apply Lemma \ref{lem:deg=1} to see that $\E_{n\leq N}^We^{2\pi i kf(n)} = Ce^{2\pi i kf(N)}+o_{N\to\oo}(1)$, so $\lim_{N\to\oo}\E^W_{n\leq N}e^{2\pi i f(n)}$ does not exist. Then the sequence $(f(n))_{n\in \N}$ is not u.d. mod 1 with respect to $W$-averages in this case.

    (\ref{eq:d=1_bosh_condition_again}) does not hold in the second case either. To see this, pick a Hardy function $V$ such that $\lim_{x\to\oo}\frac{f'(x)}{(\log V)'(x)}=1$, and note that $\lim_{x\to\oo}\frac{\log V(x)}{\log {W}(x)}=0$. One such function is given by 
    \begin{equation}\label{eq:maximal_hardy_field_function}
        V(x) = \exp\left(\int_{1}^xf'(t)~dt\right)
    \end{equation}
which is contained in $\mathcal{H}$ since maximal Hardy fields are closed under integration and exponentiation. By Lemma \ref{lem:deg=1}, $
    \E_{n\leq N}^{V}e^{2\pi i f(n)} = Ce^{2\pi i f(N)}+o_{N\to\oo}(1)$ and in particular, $\lim_{N\to\oo}\E^{V}_{n\leq N}e^{2\pi i f(n)}$ does not exist. By Theorem \ref{thm:mikey_thm} and Remark \ref{remark:1precW} it follows that $\lim_{N\to\oo}\E^{{W}}_{n\leq N}e^{2\pi i f(n)}$ also does not exist. So, $(f(n))_{n\in \N}$ is not u.d. mod 1 with respect to $W$-averages in this case.

Lastly, in the third case (\ref{eq:d=1_bosh_condition_again}) does hold. Again, consider a Hardy function $V$ such that $\lim_{x\to\oo}\frac{f'(x)}{(\log V)'(x)}=1$, and note that we now have $\lim_{x\to\oo}\frac{\log {W}(x)}{\log V(x)}=0$. Pick any nonzero $k\in \Z$, and observe that $\lim_{x\to\oo}\frac{kf'(x)}{(\log V)'(x)}\in(0,\oo)$.
By Lemma \ref{lem:deg=1} we have 
\begin{equation}\label{eq:C_eq}
    \E_{n\leq N}^{V}e^{2\pi i kf(n)} = Ce^{2\pi i kf(N)}+o_{N\to\oo}(1)
\end{equation}
    for some $C\in \C$ with $|C|<1$. Put $Y_{N,0} = e^{2\pi i kf(N)}$ and $Y_{N,k+1} = \E_{n\leq N}^{V}Y_{N,k}$ for $N\in \N$ and $k\geq 0$. From (\ref{eq:C_eq}) we have that $Y_{N,k} = C^ke^{2\pi i kf(N)}+o_{N\to\oo}(1)$ for all $k\geq 0$.  Hence $\lim_{N\to\oo}\E^W_{n\leq N}e^{2\pi i kf(n)}=0$  by Theorem \ref{thm:mikey_thm} and Remark \ref{remark:1precW}. So $(f(n))_{n\in \N}$ is u.d. mod 1 with respect to $W$-averages. This completes the proof.
\end{proof}

\subsection{Negative results when deg\texorpdfstring{$^*(f)=2$}{*(f)>1}}

We turn our attention to the case when $\deg^*(f)= 2$ and $W$ is not compatible with $f$. If $\log W(x)\preceq \log x$, this case is vacuous (see Remark \ref{remark:WP}), so we suppose that $\log x\prec \log W(x)$. Our goal is to show that $(f(n))_{n\in \N}$ is not u.d. mod 1 with respect to $W$-averages. To this end, we will find an infinite matrix $(\alpha_{N,n})_{N,n\in \N}$ such that if $\lim_{N\to\oo}\E^W_{n\leq N}e^{2\pi i f(n)}=0$ then $\lim_{N\to\oo}\sum_{n=1}^{\oo}\alpha_{N,n}e^{2\pi i f(n)}=0$ (Lemmas \ref{lem:silverman_toeplitz}, \ref{lem:Boos}, \ref{lem:gauss_stronger_than_weighted}). Next, we find a constant $C$ such that 
\begin{equation}\label{eq:find_C}
   \sum_{n=1}^{\oo}\alpha_{N,n}e^{2\pi i f(n)} = C\cdot e^{2\pi if(N)}+o_{N\to\oo}(1)
\end{equation}
for infinitely many values of $N\in \N$ (Theorem \ref{thm:gaussian_dne}). Then we show that this constant $C$ is nonzero, from which it follows that $\lim_{N\to\oo}\sum_{n=1}^{\oo}\alpha_{N,n}e^{2\pi i f(n)}$ does not exist and hence $\lim_{N\to\oo}\E^W_{n\leq N}e^{2\pi if(n)}$ does not exist and so $(f(n))_{n\in \N}$ is not u.d. mod 1 with respect to $W$-averages (Corollary \ref{cor:negative_result_d>2}).

First, we require a way to compare limits of the form  $\lim_{N\to\oo}\E^W_{n\leq N}x_n$ with limits of the form $\lim_{N\to\oo}\sum_{n=1}^{\oo}\alpha_{N,n}x_n$. For this, we need the following two results from \cite{Boos}.
\begin{lemma}{\cite[Theorem 2.3.7]{Boos}}]\label{lem:silverman_toeplitz}
Let $(\alpha_{N,n})_{N,n\in \N}$ be an infinite matrix with complex entries. Suppose that the following conditions hold.
\begin{enumerate}
    \item For each fixed $n\in \N$, $\lim_{N\to\oo}\alpha_{N,n} = 0$.   
    \item $\lim_{N\to\oo}\sum_{n\in \N}\alpha_{N,n}=1$.
    \item $\limsup_{N\to\oo}\sum_{n\in \N}|\alpha_{N,n}|<\oo$.
\end{enumerate}
Then $(\alpha_{N,n})_{N,n\in \N}$ defines a {regular matrix method}, meaning that if $(x_n)_{n\in \N}$ is a bounded sequence  of complex numbers such that $\lim_{N\to\oo}x_N$ exists then $\lim_{N\to\oo}\sum_{n\in \N}\alpha_{N,n}x_n=\lim_{N\to\oo}x_N$.
\end{lemma}
\begin{lemma}[{\cite[Theorem 3.2.8]{Boos}}]\label{lem:Boos}
    Let $W$ be a function which eventually increases to $\oo$ and let $(\alpha_{N,n})_{n\in \N}$ be an infinite matrix of complex numbers which defines a regular matrix method\footnote{The statement of \cite[Theorem 3.2.8]{Boos} assumes that $(\alpha_{N,n})_{N,n\in \N}$ is a \emph{conservative matrix method}, meaning that $\lim_{N\to\oo}\sum_{n=1}^{\oo}\alpha_{N,n}x_n$ exists whenever $\lim_{N\to\oo}x_N$ exists, but it is clear that regular matrix methods are conservative.} (see Lemma \ref{lem:silverman_toeplitz}). Let $c_{N,n} = W(n)\left(\frac{\alpha_{N,n}}{\Delta W(n)}-\frac{\alpha_{N,n+1}}{\Delta W(n+1)}\right)$ for $n,N\in \N$. The following are equivalent.
    \begin{itemize}
        \item $\lim_{N\to\oo}\sum_{n=1}^{\oo}\alpha_{N,n}x_n = L$ for any $L\in \C$ and any bounded sequence of complex numbers $(x_n)_{n\in \N}$ with $\lim_{N\to\oo}\E_{n\leq N}^Wx_n = L$.
        \item  $\limsup_{N\in \N} \sum_{n=1}^{\oo}|c_{N,n}|<\oo$, and for each $N\in \N$, $\lim_{n\to\oo}\frac{\alpha_{N,n}}{\Delta W(n)} = 0$.
    \end{itemize}
\end{lemma}

Next, we use Lemma \ref{lem:Boos} to approximate weighted averages with Gaussian shaped averages.
\begin{lemma}
\label{lem:gauss_stronger_than_weighted}
 Let $W$ be a Hardy function with $\log x\prec \log W(x)\prec x$. Define $\alpha_{N,n} = \frac{1}{\sigma_N\sqrt{2\pi}}e^{-\frac{(n-N)^2}{2\sigma_N^2}}$ for $N,n\in \N$, where $\sigma_N = ((\log W)'(N))^{-1}$ for all $N\in \N$. Let $(x_n)_{n\in \N}$ be a bounded sequence of complex numbers such that $\lim_{N\to\oo}\E_{n\leq N}^Wx_n$ exists. Then  $\lim_{N\to\oo}\sum_{n=1}^{\oo}\alpha_{N,n}x_n = \lim_{N\to\oo}\E_{n\leq N}^Wx_n$.
\end{lemma}

\begin{proof}
First, we will show that $(\alpha_{N,n})_{N,n\in \N}$ defines a regular matrix method.

By using the substitution $n\mapsto N-n$, we rewrite the sum $ \sum_{n=1}^{\oo}\frac{1}{\sigma_N\sqrt{2\pi}}e^{-\frac{(n-N)^2}{2\sigma_N^2}}$ as $\sum_{n=-\oo}^{N-1}\frac{1}{\sigma_N\sqrt{2\pi}}e^{-\frac{n^2}{2\sigma_N^2}}$. We know that $\sum_{n=-\oo}^{\oo}\frac{1}{\sigma_N\sqrt{2\pi}}e^{-\frac{n^2}{2\sigma_N^2}} = 1+o_{N\to\oo}(1)$ by comparison with the Gaussian integral $\int_{-\oo}^{\oo}\frac{1}{\sqrt{2\pi}}e^{-x^2/2}~dx=1$, and we additionally recall that for any $\epsilon>0$ there exists a constant $A\in (0,\oo)$ such that 
\begin{equation}
    \left|\sum_{n=-\oo}^{\oo}\frac{1}{\sigma_N\sqrt{2\pi}}e^{-\frac{n^2}{2\sigma_N^2}}  - \sum_{n=-\lfloor{A\sigma_N\rfloor}}^{\lfloor  A \sigma_N\rfloor}\frac{1}{\sigma_N\sqrt{2\pi}}e^{-\frac{n^2}{2\sigma_N^2}}\right| <\epsilon. 
\end{equation} 
Using the fact that $\log x\prec \log W(x)$, the inequality $A\sigma_N<N-1$ holds whenever $N$ is large enough and so we also have 
\begin{equation}\label{eq:central_limit_theorem}
     \left|\sum_{n=-\oo}^{N-1}\frac{1}{\sigma_N\sqrt{2\pi}}e^{-\frac{n^2}{2\sigma_N^2}}  - \sum_{n=-\lfloor{A\sigma_N\rfloor}}^{\lfloor  A \sigma_N\rfloor}\frac{1}{\sigma_N\sqrt{2\pi}}e^{-\frac{n^2}{2\sigma_N^2}}\right| <\epsilon,
\end{equation} 
and hence $\sum_{n=1}^{\oo}\alpha_{N,n}=\sum_{n=-\oo}^{N-1}\frac{1}{\sigma_N\sqrt{2\pi}}e^{-\frac{n^2}{2\sigma_N^2}}  = 1+o_{N\to\oo}(1)$. Since $\alpha_{N,n}\geq 0$ for all $N,n$, we also have $\limsup_{N\to\oo}\sum_{n\in \N}|\alpha_{N,n}|<\oo$. Additionally, for each fixed $N\in \N$, $\lim_{n\to\oo}\alpha_{N,n}=0$. Therefore, by Lemma \ref{lem:silverman_toeplitz} we know that $(\alpha_{N,n})_{N,n\in \N}$ defines a regular matrix method. It remains to show that we can apply Lemma \ref{lem:Boos}.

Observe that for fixed $N\in \N$, 
\begin{equation}\label{eq:boos_second_condition}
\frac{\alpha_{N,n}}{\Delta W(n)} = O(1)\cdot \frac{1}{\Delta W(n)}e^{-\frac{(n-N)^2}{2\sigma_N^2}}  \to 0 
\end{equation}
as $n\to\oo$ since $W(n)$, and hence $\Delta W(n)$, grows subexponentially.  

Lastly we show that
\begin{equation}\label{eq:messy_sum}
\limsup_{N\in \N} \sum_{n=1}^{\oo}W(n)\left|\frac{\alpha_{N,n}}{\Delta W(n)}-\frac{\alpha_{N,n+1}}{\Delta W(n+1)}\right|<\oo.
\end{equation}
Let $V(n) = \log W(n)$ so that $W(n) = e^{V(n)}$ and $\sigma_N = V'(N)$. Note that $\frac{ W(n)}{ W(n+1)} = e^{-\Delta V(n+1)} = 1+O_{N\to\oo}(\Delta V(n+1))$. Putting $\eta(n)= \frac{W(n)\alpha_{N,n}}{\Delta W(n)}$, we have
\begin{align}
  \notag & \sum_{n=1}^{\oo}\left|\frac{W(n)\alpha_{N,n}}{\Delta W(n)}-\frac{W(n)\alpha_{N,n+1}}{\Delta W(n+1)}\right| =\sum_{n=1}^{\oo}\left|\eta(n)-\eta(n+1)(1+O_{n\to\oo}(\Delta V(n+1)))\right|   \\
    \leq &\sum_{n=1}^{\oo}\left|\eta(n)-\eta(n+1)\right| +\sum_{n=1}^{\oo}\eta(n+1)\cdot O_{n\to\oo}(\Delta V(n+1)).\label{eq:bounded_sums}
    \end{align}
    The second sum is bounded since
    \begin{align*}
    &\sum_{n=1}^{\oo}\eta(n+1)\cdot O_{n\to\oo}(\Delta V(n+1)) \leq \sum_{n=1}^{\oo}\eta(n)\cdot O_{n\to\oo}(\Delta V(n))  
    \\ =& \sum_{n=1}^{\oo}\alpha_{N,n}\frac{W(n)}{\Delta W(n)}\cdot O_{n\to\oo} \left(\frac{\Delta W(n)}{W(n)}\right)=\sum_{n=1}^{\oo}\alpha_{N,n}\cdot O_{n\to\oo}(1)= O_{N\to\oo}(1).
    \end{align*}
    To bound the first sum in (\ref{eq:bounded_sums}), observe that the ratio $\frac{\eta(n+1)}{\eta(n)} \sim e^{-\frac{(n+1-N)^2-(n-N)^2}{2\sigma_N^2}} = e^{\frac{-2n-1+2N}{2\sigma_N^2}}$ is decreasing in $n$. This shows that $\eta(n)$ increases to its maximum and then decreases. So $\sum_{n=1}^{\oo}\left|\eta(n)-\eta(n+1)\right|
 \leq  2\cdot \sup_{n\in \N}\eta(n)$. We bound $\sup_{n\leq N}\eta(n)$ by noting that 
 \begin{equation*} 
 \frac{1}{\sigma_N\sqrt{2\pi}}e^{-\frac{(n-N)^2}{2\sigma_N^2}}\leq \frac{1}{\sigma_N\sqrt{2\pi}} = \frac{1}{ V'(N) \sqrt{2\pi}} = O_{N\to\oo}\left(\frac{W(N)}{ W'(N)}\right) = O_{N\to\oo}\left(\frac{W(N)}{ \Delta W(N)}\right).
 \end{equation*}
Then $\sup_{n\in \N}\eta(n) = \sup_{n\leq N}\frac{\Delta W(N)}{W(N)}\cdot O_{N\to\oo}\left(\frac{W(N)}{\Delta W(N)}\right) <\oo$. We have shown that  (\ref{eq:messy_sum}) holds and so we are done by Lemma \ref{lem:Boos}.
\end{proof}

Now we find a value of $C$ such that (\ref{eq:find_C}) holds.

\begin{theorem}\label{thm:gaussian_dne}
     Let $\mathcal{H}$ be a Hardy field. Let $W,f\in \mathcal{H}$ and suppose that $\log x\prec \log W(x)\prec x$ and $x\prec f(x)\prec x^2$. Additionally, suppose that $\lim_{x\to\oo}\frac{(f''(x))^{1/2}}{(\log W)'(x)}=\upsilon\in(0,\oo)$. Then there exist arbitrarily large values of $N\in \N$ such that
\begin{equation}\label{eq:gaussian_again}
    \sum_{n=1}^{\oo}\frac{1}{\sigma_N\sqrt{2\pi}}e^{-\frac{(n-N)^2}{2\sigma_N^2}}\cdot e^{2\pi i f(n)}= C_{\upsilon}\cdot  e^{2\pi i f(N)}+o_{N\to\oo}(1)
   \end{equation}
   where $\sigma_N = ((\log W)'(N))^{-1}$ and $C_{\upsilon} = \int_{-\oo}^{\oo} \frac{1}{\sqrt{2\pi}}e^{-u^2/2}e^{2\pi i (\upsilon/2)u^{2}}du$. 
\end{theorem}
\begin{proof}

Let $\epsilon>0$. Put $B = \upsilon^2/2 = \lim_{N\to\oo}\sigma_N^{2}\cdot \Delta^{2}f(N)/2$ and let 
\[
C_{\upsilon} =  \int_{-\oo}^{\oo} \frac{1}{\sqrt{2\pi}}e^{-u^2/2}e^{2\pi i Bu^{2}}~du.
\]
We will find an arbitrarily large value of $N\in \N$ such that 
\[
    \sum_{n=1}^{\oo}\frac{1}{\sigma_N\sqrt{2\pi}}e^{-\frac{(n-N)^2}{2\sigma_N^2}}\cdot e^{2\pi i f(n)} = C_{\upsilon}\cdot e^{2\pi i f(N)}+O(\epsilon)+o_{N\to\oo}(1).
\]

First, rewrite $ \sum_{n=1}^{\oo}\frac{1}{\sigma_N\sqrt{2\pi}}e^{-\frac{(n-N)^2}{2\sigma_N^2}} e^{2\pi i f(n)}$ as $\sum_{n=-\oo}^{N-1}\frac{1}{\sigma_N\sqrt{2\pi}}e^{-\frac{n^2}{2\sigma_N^2}} e^{2\pi i f(N-n)}$. Recall that there exists a constant $A\in (0,\oo)$ such that (\ref{eq:central_limit_theorem}) holds and so it follows from the triangle inequality that
\begin{equation}\label{eq:post_triangle_inequality}
      \sum_{n=1}^{\oo}\frac{1}{\sigma_N\sqrt{2\pi}}e^{-\frac{(n-N)^2}{2\sigma_N^2}} e^{2\pi i f(n)} =  \sum_{n=-\lfloor{A\sigma_N\rfloor}}^{\lfloor  A \sigma_N\rfloor}\frac{1}{\sigma_N\sqrt{2\pi}}e^{-\frac{n^2}{2\sigma_N^2}}e^{2\pi i f(N-n)}+O(\epsilon)
\end{equation}
since $(e^{2\pi i f(n)})_{n\in \N}$ is bounded. Next, recall Newton's backward difference formula, which says that
\begin{equation}\label{eq:newton_formula}
    f(N-n) = f(N)+
    \sum_{i=1}^{m} (-1)^i\binom{n}{i}\Delta^{i}f(N)+O_{N\to\oo}(n^{m+1}\cdot \Delta^{m+1}f(N))
\end{equation}
for any $m\in \N$. We will apply this with $m=2$, so that we have 
\begin{equation}\label{eq:newton_formula_m=2}
    f(N-n) = f(N)
    -\binom{n}{1} \Delta f(N)+\binom{n}{2} \Delta^2 f(N)+O_{N\to\oo}(n^{3}\cdot \Delta^{3}f(N)).
\end{equation}

Note that the sequence $(\Delta f(n)\mod 1)_{n\in \N}$ takes steps of size $\Delta^2f(n)$, which decreases to $0$ for sufficiently large $n$. Pick an arbitrarily large value of $N\in \N$ such that $\Delta f(N-1)\mod 1\in (1/2,1)$ and $\Delta f(N)\mod 1\in (0,1/2)$. It must be that 
\begin{equation}\label{eq:lemma_replacement}
\Delta f(N)\mod 1\in (0,\Delta^2 f(N))\subset (0,\frac{\epsilon}{A\upsilon}\cdot (\Delta^2 f(N))^{1/2}).
\end{equation}

 Recall that $\binom{n}{1}=n$, so 
 $$
 \max_{n\in [-\lfloor{A\sigma_N\rfloor},\lfloor{A\sigma_N\rfloor}]}\left|\binom{n}{1} \Delta f(N)\right|\mod 1 =\lfloor{A\sigma_N\rfloor}\cdot \Delta f(N)\mod 1\leq \epsilon
 $$
 by (\ref{eq:lemma_replacement}).

The $O_{N\to\oo}({n^{3}\Delta^{3}f(N)})$ term in (\ref{eq:newton_formula_m=2}) is $o_{N\to\oo}(1)$ uniformly for $n\in [-\lfloor{A\sigma_N\rfloor},\lfloor{A\sigma_N\rfloor}]$ by L'H\^opital's rule, and 
$$
\binom{n}{2}\Delta^2f(n) = (n^2/2)\Delta^2f(n)+\Delta^2 f(n)\cdot O_{n\to\oo}(n) = (n^2/2)\Delta^2f(n)+o_{n\to\oo(1)}(1).
$$
Altogether, we have shown that
\begin{align*}
    e^{2\pi i f(N-n)} =&e^{2\pi i f(N)}\cdot e^{2\pi i (n^2/2)\Delta^{2}f(N)}\cdot e^{O(\epsilon)}\cdot e^{o_{N\to\oo}(1)}\\
   =& e^{2\pi i f(N)}\cdot e^{2\pi i (n^2/2)\Delta^{2}f(N)}+O(\epsilon) +o_{N\to\oo}(1)\\
    =& e^{2\pi i f(N)}\cdot e^{2\pi i n^{2}\cdot B/\sigma_N^{2}\cdot (1+o_{N\to\oo}(1))}+O(\epsilon) +o_{N\to\oo}(1)\\
    =&e^{2\pi i f(N)}\cdot e^{2\pi i B(n/\sigma_N)^{2}}+O(\epsilon) +o_{N\to\oo}(1)
\end{align*}
uniformly for $n\in \{-\lfloor A\sigma_N\rfloor,\dots, \lfloor A\sigma_N\rfloor\}$, and so (\ref{eq:post_triangle_inequality}) becomes
\begin{equation}\label{eq:negligible_terms}
e^{2\pi i f(N)}\cdot\sum_{n=-\lfloor{A\sigma_N\rfloor}}^{\lfloor  A \sigma_N\rfloor}\frac{1}{\sigma_N\sqrt{2\pi}}e^{-\frac{(n/\sigma_N)^2}{2}}e^{2\pi i B(n/\sigma_N)^{2}}+O(\epsilon)+o_{N\to\oo}(1).
\end{equation}

Let $u = n/\sigma_N$. Then the sum in (\ref{eq:negligible_terms}) is a Riemann sum for the integral 
\[
\int_{-A}^A \frac{1}{\sqrt{2\pi}}e^{-u^2/2}e^{2\pi i Bu^{2}}du= \int_{-\oo}^{\oo} \frac{1}{\sqrt{2\pi}}e^{-u^2/2}e^{2\pi i Bu^{2}}du+O(\epsilon).
\]

This gives us 
\[
    \sum_{n=1}^{\oo}\frac{1}{\sigma_N\sqrt{2\pi}}e^{-\frac{(n-N)^2}{2\sigma_N^2}}\cdot e^{2\pi i f(n)} = C_{\upsilon}\cdot e^{2\pi i f(N)}+O(\epsilon)+o_{N\to\oo}(1).
\]
Taking $\epsilon\to 0$ shows (\ref{eq:gaussian_again}) as desired.
\end{proof}

We are now ready to prove this case of Theorem \ref{thm:W_ud}.

\begin{corollary}\label{cor:negative_result_d>2}
   Let $\mathcal{H}$ be a Hardy field and let $f,W\in \mathcal{H}$ satisfy $\deg^*(f)=2$ and  $\log x\prec \log W(x)\prec x$. Suppose that there is $p(x)\in \Q[x]$ such that 
\begin{equation}\label{eq:gen_bosh_condition_fails}
         \lim_{x\to\oo}\frac{|f^{''}(x)-p(x)|^{1/2}}{(\log W)'(x)}<\oo.
 \end{equation}
Then $(f(n))_{n\in \N}$ is not u.d. mod 1 with respect to $W$-averages. 
\end{corollary}
\begin{proof}
By the same reasoning as in Theorem \ref{thm:d=1_thm}, we assume without loss of generality that $x\prec f(x)\prec x^2$. We will show that $\lim_{N\to\oo}\E_{n\leq N}^We^{2\pi i f(n)}$ does not exist. Pick an arbitrarily small $\upsilon>0$ and pick a function $V$ which belongs to a maximal Hardy field containing $\mathcal{H}$ and satisfies $\lim_{x\to\oo}\frac{(f^{''}(x))^{1/2}}{(\log V)'(x)}=\upsilon$ (see (\ref{eq:maximal_hardy_field_function})). We will show that $\lim_{N\to\oo}\E_{n\leq N}^Ve^{2\pi i f(n)}$ does not exist and it will follow that $\lim_{N\to\oo}\E_{n\leq N}^We^{2\pi i f(n)}$ does not exist by applying Theorem \ref{thm:mikey_thm} if $\lim_{x\to\oo}\frac{\log V(x)}{\log W(x)}=0$. If instead we have $\lim_{x\to\oo}\frac{\log V(x)}{\log W(x)}\in(0,\oo)$, then we apply \cite[Lemma 7.1]{uniform_distribution} which says that $\lim_{N\to\oo}\E^W_{n\leq N}x_n$ converges if and only if $\lim_{N\to\oo}\E^V_{n\leq N}x_n$ converges under the condition that
$$
\lim_{x\to\oo}\frac{\log V(x)}{\log W(x)} = \lim_{x\to\oo}\frac{W(N)\cdot \Delta V(N)}{V(N)\cdot \Delta W(N)}\in(0,\oo).
$$

From Theorem \ref{thm:gaussian_dne} we have \begin{equation}\label{eq:gaussian_again_again}
    \sum_{n=1}^{\oo}\frac{1}{\sigma_N\sqrt{2\pi}}e^{-\frac{(n-N)^2}{2\sigma_N^2}}\cdot e^{2\pi i f(n)}= C_{\upsilon}\cdot  e^{2\pi i f(N)}+o_{N\to\oo}(1)
   \end{equation}
   holds for infinitely many values of $N$, where $\sigma_N = (\log V)'(N)$ and $C_{\upsilon}$ is given by the integral $ \int_{-\oo}^{\oo} \frac{1}{\sqrt{2\pi}}e^{-u^2/2}e^{2\pi i (\upsilon/2)u^{2}}du$. For this observation to be of any use to us, we need to know that $C_{\upsilon}$ is nonzero. Consider the function $\xi\mapsto C_{\xi} = \int_{-\infty}^{\infty}e^{-x^2/2}e^{2\pi i (\xi/2) x^{2}}dx$. This function is continuous and it is a classical fact that $C_0 = \int_{-\infty}^{\infty}e^{-x^2/2}dx = \sqrt{2\pi}>0$. So $|C_{\upsilon}|>0$ whenever $\upsilon>0$ is small enough. Then the limit as $N\to\oo$ of equation (\ref{eq:gaussian_again_again}) does not exist, and hence $\lim_{N\to\oo}\E_{n\leq N}^Ve^{2\pi i f(n)}$ does not exist. This concludes the proof.
\end{proof}

\subsection{Positive results when deg\texorpdfstring{$^*(f)>1$}{*(f)>1}}

We finish the proof of Theorem \ref{thm:W_ud} by showing that the sequence $(f(n))_{n\in \N}$ is u.d. mod 1 with respect to $W$-averages when $\deg^*(f)>1$ and $W$ is compatible with $f$. This is achieved by showing that 
\begin{equation}\label{eq:positive_results_goal}
    \lim_{N\to\oo}\frac{1}{s(N)}\sum_{n=N-s(N)}^Ne^{2\pi i kf(n)}=0
\end{equation}
 for all nonzero $k\in \Z$ and for all functions $s:\N\rightarrow \N$ satisfying the conditions in Theorem \ref{thm:mikey_thm}(3). More specifically, we show the following.

\begin{theorem}\label{thm:d>1_positive_thm}
    Let $f$ be a Hardy function with $x^{d-1}\prec f(x)\prec x^d$ for some $d\in \N$ and let $s:\N\rightarrow \N$ satisfy 
    $\lim_{N\to\oo}s(N)\cdot (f^{(d)}(N))^{1/d}=\oo$ and $s(N)\leq N-1$ for all $N\in\N$. Then (\ref{eq:positive_results_goal}) holds for all nonzero $k\in \Z$.
\end{theorem}
From Theorem \ref{thm:d>1_positive_thm}, the last case of Theorem \ref{thm:W_ud} follows.

\begin{corollary}\label{cor:positive_result_d>2}
   Let $\mathcal{H}$ be a Hardy field and let $f,W\in \mathcal{H}$ satisfy $\deg^*(f)=d>1$ and  $1\prec \log W(x)\prec x$. Suppose that
\begin{equation}\label{eq:gen_bosh_condition_holds}
         \lim_{x\to\oo}\frac{|f^{(d)}(x)-p(x)|^{1/d}}{(\log W)'(x)}=\oo
 \end{equation}
for all $p(x)\in \Q[x]$. Then $(f(n))_{n\in \N}$ is  u.d. mod 1 with respect to $W$-averages. 
\end{corollary}
\begin{proof}[Proof of Corollary \ref{cor:positive_result_d>2} given Theorem \ref{thm:d>1_positive_thm}]
    Pick a nonzero value of $k\in \Z$ and, as in Theorem \ref{thm:d=1_thm} and Corollary \ref{cor:negative_result_d>2}, assume without loss of generality that $x^{d-1}\prec f(x)\prec x^d$.

    Let $V$ belong to a maximal Hardy field containing $\mathcal{H}$ such that $\lim_{x\to\oo}\frac{(f^{(d)}(x))^{1/d}}{(\log V)'(x)}=1$ (see (\ref{eq:maximal_hardy_field_function})). By Theorem \ref{thm:d>1_positive_thm}, we know that (\ref{eq:positive_results_goal}) holds for each $s:\N\rightarrow \N$ with $\lim_{N\to\oo}s(N)\cdot (\log V)'(N) = \oo$ and $s(N)\leq N-1$ for all sufficiently large $N\in \N$. By (\ref{eq:gen_bosh_condition_holds}) and our choice of $V$, $\lim_{x\to\oo}\frac{\log W(x)}{\log V(x)}=0$ and so by the implication $(3)\implies(2)$ in Theorem \ref{thm:mikey_thm}, we have $\lim_{N\to\oo}\E_{n\leq N}^We^{2\pi i k f(n)}=0$. This holds for all nonzero $k\in \Z$, and so we have shown that $(f(n))_{n\in \N}$ is u.d. mod 1 with respect to $W$-averages.
\end{proof}

The proof of Theorem \ref{thm:d>1_positive_thm} goes by induction on $d$, where the base case $d=1$ follows from the results of Section \ref{section:d=1}. A helpful tool for proving the induction step is a variant of van der Corput's trick, which is a special case of \cite[Theorem~2.12]{BM16}.
\begin{theorem}[van der Corput's trick]\label{thm:vdc_trick}
Let $(x_n)_{n\in \N}$ be a sequence of real numbers and let $(I_N)_{N\in \N}$ be a sequence of intervals of natural numbers with $|I_N|\to\oo$ as $N\to\oo$. Suppose that for each $j\in \N$, 
$$
\frac{1}{|I_N|}\sum_{n\in I_N}e^{2\pi i x_{n+j}}\overline{e^{2\pi ix_n} }=\frac{1}{|I_N|}\sum_{n\in I_N}e^{2\pi i (x_{n+j}-x_n)}\to 0
$$
as $N\to\oo$. Then $\frac{1}{|I_N|}\sum_{n\in I_N} e^{2\pi ix_n} \to 0$ as $N\to\oo$.
\end{theorem}

We also need some technical lemmas on bounding exponential sums.

\begin{lemma}[{\cite[Theorem 2.2]{van_der_corput_inequality_book}}]\label{lem:second_derivative_vdc_inequality}
    Let $F$ be a real-valued function which is twice continuously differentiable on an interval $I$ and suppose that for some $\lambda>0$ and $\alpha\geq 1$, we have $\lambda \leq |F''(x)|\leq \alpha \lambda $ for $x\in I$. Then
    \begin{equation}\label{eq:second_derivative_vdc_inequality}
    \left|\sum_{x\in I}e^{iF(x)}\right|\ll \alpha \lambda^{1/2}|I|+\frac{1}{\lambda^{1/2}}.
    \end{equation}
    Here the notation $A\ll B$ means that there is an absolute constant $C$ such that $A\leq CB$.
\end{lemma}

The following is the classical iterated van der Corput inequality, whose proof can be found in section 2.4 of \cite{van_der_corput_inequality_book} or as Lemma 2.11 in \cite{BKS19}.
\begin{lemma}[{\cite[Lemma 2.11]{BKS19}}]\label{lem:better_vdc}
    Let $k$ be a positive integer and $K=2^k$. Assume that $I=(X_1,X_1+X]\subseteq (X_1,2X_1]$ and let $S= \sum_{x\in I} e^{if(x)}$. For any positive $H_1,...,H_k\leq C(k)\cdot X$, where $C(k)$ is a constant depending only on $k$, we have 
    \begin{align}\label{eq:vdc_inequality}
\left(\frac{S}{X}\right)^K \leq 8^{K-1} \left(\sum_{i=1}^k\frac{1}{H_i^{K/2^j}}+\frac{1}{XH_1\cdots H_k}\sum_{h_1=1}^{H_1}\cdots \sum_{h_k=1}^{H_k}\left|\sum_{x\in I(\mathbf{h})}e^{if_1(x)}\right|\right)
    \end{align}
    where $f_1(x) =f(\mathbf{h},x)=h_1\cdots h_k \int_0^1\cdots \int_0^1 \frac{\partial^k}{\partial x^k }f(x+\mathbf{h}\cdot\mathbf{t})d\mathbf{t}$, $\mathbf{h}=(h_1,\dots,h_k)$, $\mathbf{t}=(t_1,\dots ,t_k)$ and $I(\mathbf{h})=(X_1,X_1+X-h_1-\dots-h_k]$.
\end{lemma}

Combining the previous two lemmas, we obtain the following.

\begin{theorem}\label{thm:better_vdc}
    Let $j\geq 3$ and let $f$ be a real valued functions which is $j$-times continuously differentiable on the interval $I = (X_1,X_1+X]\subseteq (X_1,2X_1]$, for $X,X_1\in \N$. Suppose that $f^{(j)}$ is monotone on $I$ and that there are constants $\lambda>0$ and $\alpha\geq 1$ with $\lambda \leq |f^{(j)}(x)|\leq \alpha\lambda $ for all $x\in I$. Then 
    \begin{align}\label{eq:better_vdc}
        \left|\frac{1}{|I|}\sum_{x\in I}e^{if(x)}\right|\ll \left(\frac{1}{X}\right)^{\frac{1}{2^{j-2}}}+\alpha^{\frac{1}{2^{j-2}}} (\lambda X^{j-2})^{\frac{1}{2^{j-1}}}+\left(\frac{1}{\lambda X^{j}}\right)^{\frac{1}{2^{j-1}}}
    \end{align}
    and the absolute constant depends only on $j$.
\end{theorem}
\begin{proof}
    We will begin by applying Lemma \ref{lem:better_vdc}. Let $k=j-2$, $K=2^k$, and put $H_i = \min\{\frac{X}{C(k)},\frac{X}{2k}\}$ for $i\in \{1,\dots, k\}$ and $H_{\times} = [1,H_1]\times \dots \times [1,H_k]\subset \N^k$. It is clear that 
    \[
    \sum_{i=1}^k\frac{1}{H_i^{K/2^i}}\ll \frac{1}{X}
    \]
    and so it suffices to bound the second term in (\ref{eq:vdc_inequality}). Observe that $\lambda h_1\cdots h_k \leq |f_1{''}(x)|\leq \alpha\lambda h_1\cdots h_k$ for all $x\in I$ and $\mathbf{h}=(h_1,\dots, h_k)\in H_{\times}$. Therefore, we may apply Lemma \ref{lem:second_derivative_vdc_inequality} to see that for each $\mathbf{h} \in H_{\times}$, 
    \begin{align*}
        \left|\sum_{x\in I(\mathbf{h})}e^{if_1(x)}\right| \ll& \alpha \lambda^{1/2}(h_1\cdots h_k)^{1/2}(X-h_1-\cdots -h_k)+\frac{1}{\lambda^{1/2}(h_1\cdots h_k)^{1/2}}\\
        \leq & \alpha \lambda^{1/2}(h_1\cdots h_k)^{1/2}X+\frac{1}{\lambda^{1/2}(h_1\cdots h_k)^{1/2}}.
    \end{align*}

    So far, we have shown that 
    \begin{align}
         &\left|\frac{1}{X}\sum_{x\in I}e^{if(x)}\right|^{2^k}\\
         \ll&  \frac{1}{X} + \frac{1}{XH_1\cdots H_k}\sum_{\mathbf{h}\in H_{\times}}\left( \alpha \lambda^{1/2}(h_1\cdots h_k)^{1/2}X+\frac{1}{\lambda^{1/2}(h_1\cdots h_k)^{1/2}}\right)\label{eq:almost_there}.
    \end{align}
    Next, we bound the sum over $\mathbf{h}$ by noting that $\sum_{h_i=1}^{H_i}h_i^{1/2} \leq H_i^{3/2}\leq \frac{X^{3/2}}{(2k)^{3/2}}$ and $\sum_{h_i=1}^{H_i}h_i^{-1/2} \leq 2H_i^{1/2}\leq \frac{2X^{1/2}}{(2k)^{1/2}}$ for each $i$. So,
    \begin{align*}
        \sum_{\mathbf{h}\in H_{\times}} \alpha \lambda^{1/2}(h_1\cdots h_k)^{1/2}X+\frac{1}{\lambda^{1/2}(h_1\cdots h_k)^{1/2}} \ll \alpha\lambda^{1/2}X^{3k/2+1} +\frac{X^{k/2}}{\lambda^{1/2}}.
    \end{align*}
    Noting that $H_1\cdots H_k\gg X^k$, we combine this with (\ref{eq:almost_there}) to obtain
    \begin{align*}
        \left|\frac{1}{X}\sum_{x\in I}e^{if(x)}\right|^{2^k}
         \ll&  \frac{1}{X} +  \frac{1}{X^{k+1}}\left(\alpha\lambda^{1/2}X^{3k/2+1}+\frac{X^{k/2}}{\lambda^{1/2}}\right)\\
         =&\frac{1}{X} +\alpha(\lambda X^{k})^{1/2}+\frac{1}{(\lambda X^{k+2})^{1/2}}
         \\=&\frac{1}{X} +\alpha(\lambda X^{j-2})^{1/2}+\frac{1}{(\lambda X^{j})^{1/2}}.
    \end{align*}
    Taking $2^k$th roots of both sides gives (\ref{eq:better_vdc}), which completes the proof.
\end{proof}

Now we are ready to prove Theorem \ref{thm:d>1_positive_thm}. 

\begin{proof}[Proof of Theorem \ref{thm:d>1_positive_thm}]

Fix a nonzero $k\in \Z$. We will prove the statement of the theorem by induction on $d\in \N$.

The base case $d=1$ is true by Theorem \ref{thm:d=1_thm} and Theorem \ref{thm:mikey_thm}. More specifically, let $d=1$ and find a function $V$ which belongs a maximal Hardy field containing $f$ and satisfies $\lim_{x\to\oo}\frac{(f^{(d)}(x))^{1/d}}{(\log V)'(x)}=1$ (see (\ref{eq:maximal_hardy_field_function})). From Theorem \ref{thm:d=1_thm}, we have that $\lim_{N\to\oo}\E_{n\leq N}^We^{2\pi i kf(n)}=0$ for each $W$ which belongs to a Hardy field containing $f$ and $V$, which tends to $\oo$ and satisfies $\lim_{x\to\oo}\frac{\log W(x)}{\log V(x)} = 0$. Then the implication $(2)\implies (3)$ in Theorem \ref{thm:mikey_thm} shows that (\ref{eq:positive_results_goal}) holds for any $s:\N\rightarrow \N$ which satisfies $\lim_{N\to\oo}s(N)\cdot (\log V)'(N) = \lim_{N\to\oo}s(N)\cdot (f^{(d)}(x))^{1/d} = \oo$, and this shows that the base case holds.

Now for the induction step. 
Let $d\geq 2$ and suppose that for any Hardy function $F$ with $x^{d-2}\prec F(x)\prec x^{d-1}$ and for any $s:\N\rightarrow \N$ with 
    $\lim_{N\to\oo}s(N)^{d-1}\cdot F^{(d-1)}(N)=\oo$ and $s(N)\leq N-1$ for all sufficiently large $N\in\N$, we have 
    \begin{equation*}
    \lim_{N\to\oo}\frac{1}{s(N)}\sum_{n=N-s(N)}^Ne^{2\pi i kF(n)}=0
\end{equation*}
    for all nonzero $k\in \Z$.

Fix $s:\N\rightarrow \N$ with $s(N)\leq N-1$ for all sufficiently large $N\in\N$ and let $f$ be a Hardy function with $x^{d-1}\prec f(x)\prec x^d$ and $\lim_{N\to\oo}s(N)^d\cdot f^{(d)}(N)=\oo$. There are two cases to consider: the case where $\lim_{x\to\oo}s(x)^{d}\cdot f^{(d+1)}(x)=\oo$ and the case where $\lim_{x\to\oo}s(x)^{d}\cdot f^{(d+1)}(x)<\oo$.

First suppose that $\lim_{x\to\oo}s(x)^{d}\cdot f^{(d+1)}(x)=\oo$. Let $j\in \N$ and apply the induction hypothesis to $F(n)=f(n+j)-f(n)$, which satisfies $F(x) = j\cdot f^{'}(x)\cdot (1+o_{x\to\oo}(1))$ and hence $\lim_{x\to\oo}s(x)^{d}\cdot f^{(d)}(x)=\oo$, so that 
$$
\lim_{N\to\oo}\frac{1}{s(N)}\sum_{n=N-s(N)}^Ne^{2\pi i k (f(n+j)-f(n))}=0
$$
for all $k\in \N$ and hence (\ref{eq:positive_results_goal}) holds for all $k\in \N$ by applying Theorem \ref{thm:vdc_trick} with $I_{N} = [N-s(N),N]$ for $N\in \N$.

    On the other hand, suppose that  $\lim_{x\to\oo}f^{(d+1)}(x)\cdot s(x)^{d}<\oo$. Then it must be that $\lim_{x\to\oo}f^{(d+1)}(x)\cdot {s}(x)^{d-1}=0$. Take $j=d+1$, $X=s(N)$, $I = [N-s(N),N]$, $\lambda = f^{(d+1)}(N)$, $\alpha = \frac{f^{(d+1)}(N-s(N))}{f^{(d+1)}(N)}$ in order to apply Lemma \ref{lem:second_derivative_vdc_inequality} if $d=1$ or Theorem \ref{thm:better_vdc} if $d\geq 2$. We have $X\to \oo$, $X\lambda^{j-2}\to 0$, and $X\lambda^j \to \oo$ as $N\to\oo$ by assumption, and we can see that $\alpha\to 1$ as $N\to\oo$. In the case $d=1$ or $d>1$, we can use (\ref{eq:second_derivative_vdc_inequality}) or (\ref{eq:better_vdc}), respectively, to see that $$
    \lim_{N\to\oo}\frac{1}{s(N)}\sum_{n=N-s(N)}^Ne^{2\pi i f(n)}=0.
    $$
    This completes the proof of the induction step and so we are done.
\end{proof}

\section{Preliminaries}\label{section:preliminaries} 
In this section, we recall some notation and preliminary results about nilmanifolds and weighted averages from \cite{BMR24} and \cite{Richter23} that will be used in Sections \ref{section:BMR_results} and \ref{section:richter_results} to prove Theorem \ref{thm:main}. 

\subsection{Preliminaries on Nilmanifolds}

\begin{definition}\label{def:nilsystem}
  Let $G$ be a nilpotent Lie group. 
  We say that a closed subgroup $\Gamma\subset G$ is \emph{uniform} if $X = G/\Gamma$ is compact, and we say that $\Gamma$ is \emph{discrete} if there is an open cover of $\Gamma$ in which each element of $\Gamma$ belongs to a unique element of the cover. When $\Gamma$ is uniform and discrete, $X$ is called a \emph{nilmanifold}. We will use $\mu_X$ to denote the Haar measure on a nilmanifold $X$. A measure preserving system of the form $(X,\mathscr{B}_X,\mu_X,T)$, where $X$ is a nilmanifold and $T(x\Gamma) = ax\Gamma$ for some $a\in G$, is called a \emph{nilsystem}.
\end{definition}

    The simplest example of a nilmanifold is the torus $\R^n/\Z^n$ for $n\in \N$. Taking $n=1$ and $Tx=x+\alpha\mod 1$ for some $\alpha \in [0,1)$, the system $(\mathbb{R}/\mathbb{Z},\mathscr{B},\mu,T)$ is a nilsystem, where $\mathscr{B}$ is the Borel $\sigma$-algebra on $\R/\Z$ and $\mu$ is the Lebesgue measure. For a nonabelian example, consider the Heisenberg group
    \begin{equation*}
        G = \left\{
        \begin{pmatrix}
            1 & a & c\\
            0 & 1 & b\\
            0 & 0 & 1
        \end{pmatrix}:a,b,c\in \R\right\}
    \end{equation*}
    which is nilpotent because the commutator 
    $[G,G] = \{g_1^{-1}g_2^{-1}g_1g_2:g_1,g_2\in G\}$ is the abelian subgroup consisting of all elements of $G$ which have $a=b=0$. Let $\Gamma$ be the subgroup consisting of all elements of $G$ such that $a,b,c\in \Z$. Then the Heisenberg nilmanifold is $X = G/\Gamma$ with equivalence classes given by 
    \begin{equation*}
         \begin{pmatrix}
            1 & a & c\\
            0 & 1 & b\\
            0 & 0 & 1
        \end{pmatrix}\Gamma =  \begin{pmatrix}
            1 & a+n & c+k+am\\
            0 & 1 & b+m\\
            0 & 0 & 1
        \end{pmatrix}\Gamma
    \end{equation*}
    for any $a,b,c\in \R$ and any $n,m,k\in \Z$.
\begin{definition}\label{def:u.d.}
 Let $(x_n)_{n\in \N}$ be a sequence in a nilmanifold $X$ with measure $\mu$. We say that $(x_n)_{n\in \N}$ is \emph{uniformly distributed with respect to $W$-averages} in $(X,\mu)$ if 
 \begin{equation}
          \lim_{N\to\oo}\E^W_{n\leq N}F(x_n) = \int_X F~d\mu
 \end{equation}
 for all $F\in C(X)$ (we may say $X$ instead $(X,\mu)$ if the measure is the Haar measure $\mu_X$). When $W(N)=N$, we call $(x_n)_{n\in \N}$ \emph{uniformly distributed} in $(X,\mu)$. Let $(Y_n)_{n\in \N}$ be a sequence of subnilmanifolds of a nilmanifold $X$. We say that $(Y_n)_{n\in \N}$ is uniformly distributed in $X$ if 
  \begin{equation}
          \lim_{N\to\oo}\E_{n\leq N}F(Y_n) = \int_X F~d\mu_X
 \end{equation}
 for all $F\in C(X)$. Here and for the rest of the paper, we use the notation $F(Y) =\int_Y F~d\mu_Y$ for a subnilmanifold $Y$ and a function $F\in C(X)$.
\end{definition}

When $X= \R/\Z\cong \T=\{z\in \C:|z|=1\}$, uniform distribution in $X$ corresponds to the usual notion of uniform distribution mod 1 (see equations (\ref{eq:u.d._def}) and (\ref{eq:u.d._def_2})) using the map $x\mapsto e^{2\pi i x}$. For an example of a sequence which is uniformly distributed in the Heisenberg nilmanifold, let 
\begin{equation*}
    g = \begin{pmatrix}
            1 & \sqrt{2} & 0\\
            0 & 1 & \sqrt{3}\\
            0 & 0 & 1
        \end{pmatrix}
\end{equation*}
so that 
\begin{equation*}
    g^n = \begin{pmatrix}
            1 & n\sqrt{2} & \binom{n}{2}\sqrt{6}\\
            0 & 1 & n\sqrt{3}\\
            0 & 0 & 1
        \end{pmatrix}
\end{equation*}
for all $n\in \N$, which belongs to the equivalence class in the Heisenberg nilmanifold given by 
\begin{equation*}
    g^n\Gamma = \begin{pmatrix}
            1 & n\sqrt{2} \mod1 & \left(\binom{n}{2}\sqrt{6}-n\sqrt{2}\cdot \floor{n\sqrt{3}}\right)\mod 1\\
            0 & 1 & n\sqrt{3}\mod 1\\
            0 & 0 & 1
        \end{pmatrix}\Gamma.
\end{equation*}
By \cite[Theorem C]{Leibman04}, a sequence is uniformly distributed in a nilmanifold $G/\Gamma$ if and only if its projection is uniformly distributed in the \emph{maximal factor torus} $[G^{\circ},G^{\circ}]\backslash G$, where $G^{\circ}$ denotes the largest connected normal subgroup of $G$. When $G$ is the Heisenberg group, $G^{\circ} = G$ and $[G,G]\backslash G\cong \R^2/\Z^2$. The projection of $(g^n\Gamma)_{n\in \N}$ onto $\R^2/\Z^2$ is $((n\sqrt{2}\mod 1, n\sqrt{3}\mod 1))_{n\in \N}$ which is uniformly distributed by Weyl's criterion for uniform distribution. Hence $(g^n\Gamma)_{n\in \N}$ is uniformly distributed in the Heisenberg nilmanifold.

\begin{definition}\label{def:relatively_independent_product}
Let $G$ be a nilpotent Lie group, let $\Gamma\subset G$ be a uniform and discrete subgroup, and let $X=G/\Gamma$. Let $L\subset G$ be a normal subgroup which is also \emph{rational}, meaning that the set of rational elements $\{x\in L: x^n\in \Gamma\text{ for some } n\in \N\}$ is dense in $L$. Define
\begin{align*}
    &G\times_L G = \{(g_1,g_2)\in G\times G: g_1g_2^{-1}\in L\},\\
    &\Gamma\times_L \Gamma =\{(g_1,g_2)\in \Gamma\times \Gamma: g_1g_2^{-1}\in L\}= (\Gamma\times \Gamma)\cap (G\times_L G), \\
    &X\times_L X = (G\times_L G)/(\Gamma \times_L \Gamma).
\end{align*}
(Observe that $G \times_L G$ is a nilpotent Lie group and $X \times_L X$ is a nilmanifold.)
\end{definition}

Again, let $G$ be the Heisenberg group. Let $L = [G,G]$ and observe that $L$ is rational because the set of rational elements of $L$
\begin{equation*}
         \left\{
        \begin{pmatrix}
            1 & 0 & r\\
            0 & 1 & 0\\
            0 & 0 & 1
        \end{pmatrix}:r\in \Q\right\}
    \end{equation*}
    is dense in $L$. Then 
    \begin{align*}
         G\times_L G = &\left\{\left(
        \begin{pmatrix}
            1 & x & z_1\\
            0 & 1 & y\\
            0 & 0 & 1
            \end{pmatrix},\begin{pmatrix}
            1 & x & z_2\\
            0 & 1 & y\\
            0 & 0 & 1
            \end{pmatrix}\right):x,y,z_1,z_2\in \R\right\},\\
            \Gamma\times_L \Gamma = &\left\{\left(
        \begin{pmatrix}
            1 & n & k_1\\
            0 & 1 & m\\
            0 & 0 & 1
            \end{pmatrix},\begin{pmatrix}
            1 & n & k_2\\
            0 & 1 & m\\
            0 & 0 & 1
            \end{pmatrix}\right):n,m,k_1,k_2\in \Z\right\},
    \end{align*}
    and $X\times_L X$ is isomorphic to $X\times (\R/\Z)$ because 

\begin{align*}
&
\left(
        \begin{pmatrix}
            1 & x & z_1\\
            0 & 1 & y\\
            0 & 0 & 1
            \end{pmatrix},\begin{pmatrix}
            1 & x & z_2\\
            0 & 1 & y\\
            0 & 0 & 1
            \end{pmatrix}\right)(\Gamma\times_L \Gamma) \\
            =&
            \left(
        \begin{pmatrix}
            1 & x  & z_1\\
            0 & 1 & y \\
            0 & 0 & 1
            \end{pmatrix},\begin{pmatrix}
            1 & x& z_2\mod 1\\
            0 & 1 & y\\
            0 & 0 & 1
 \end{pmatrix}\right)(\Gamma\times_L \Gamma),
\end{align*}
and so we have the isomorphism given by
\begin{align*}
    \left(
        \begin{pmatrix}
            1 & x & z_1\\
            0 & 1 & y\\
            0 & 0 & 1
            \end{pmatrix},\begin{pmatrix}
            1 & x & z_2\\
            0 & 1 & y\\
            0 & 0 & 1
            \end{pmatrix}\right)(\Gamma\times_L \Gamma) \mapsto  \left(\begin{pmatrix}
            1 & x & z_1\\
            0 & 1 & y\\
            0 & 0 & 1
            \end{pmatrix},z_2 \mod 1\right)\in X\times (\R/\Z).
\end{align*}

    In Section \ref{section:BMR_results}, we will need the following lemma about uniform distribution in $X\times_L X$.
 \begin{lemma}[{\cite[Lemma 4.3]{Richter23}}]\label{lem:wd_in_diagonal}
Let $G$ be a simply connected nilpotent Lie group, $\Gamma$ a uniform and discrete subgroup of $G$ and consider the nilmanifold $X=G/\Gamma$. Let $b \in G$ be arbitrary and let $L$ denote the smallest connected, normal, rational, and closed subgroup of $G$ containing $b^{\R} = \{b^t : t \in \R\}$. Let $X^{\triangle}$ denote the diagonal $\{(x, x) : x \in X\}$. Then for all but countably many $\xi \in\R$, the sequence of subnilmanifolds $((b^{\xi n}, 1_G)X^{\triangle}))_{n\in \N}$ is uniformly distributed in the
relatively independent product $X\times_L X$, where we identify $(g\Gamma, g\Gamma)$ with $(g,g)\Gamma\times_L \Gamma$. 
\end{lemma}

There is a generalization of Weyl's criterion for uniform distribution (see equation (\ref{eq:u.d._def_2})) to sequences of points in nilmanifolds which we will use in Section \ref{section:richter_results}. This generalization is formulated in terms of central characters on nilmanifolds.

\begin{definition}
    Let $G$ be a nilpotent Lie group, let $\Gamma$ be a uniform and discrete subgroup of $G$, and let $X=G/\Gamma$. A pair $(\varphi,\chi)$ is called a \emph{central character} of $(G,\Gamma)$ if $\varphi:X\rightarrow  \C$ is continuous and $\chi:Z(G)\rightarrow \T\subset \C$ is a continuous group homomorphism defined on the center of $G$, such that $\varphi(yx) = \chi(y)\varphi(x)$ for all $y\in Z(G)$ and all $x\in X$.
\end{definition}
If $X = \R^n/\Z^n\cong \T^n$ then each central character of $X$ has the form
\begin{equation}\label{eq:character_example}
    \varphi(x_1,\dots, x_n) = \exp(2\pi i ({t_1x_1+\dots+t_nx_n}))
\end{equation}
for $t_1,\dots, t_n\in \R$. When $X$ is the Heisenberg nilmanifold, any central character has the form 
$$
\varphi \begin{pmatrix}
            1 & a & c\\
            0 & 1 & b\\
            0 & 0 & 1
            \end{pmatrix} = e^{2\pi i k c}F(a,b)
$$
where $k\in \Z$ and $F:\R^2\rightarrow \C$ is continuous and satisfies $F(a+n,b+m) = e^{-2\pi i k am}F(a,b)$. One such choice of $F$ for $k=1$ is given by $F(a,b) = \sum_{n\in \Z}e^{-\pi(n+b)^2}e^{2\pi i n a}$.

As above, for a subgroup $L\subset G$ we use the notation $L^{\circ}$ to denote the largest connected normal subgroup of $L$. 
In the course of proving \cite[Theorem 4.2]{Richter23}, Richter proves the following\footnote{Richter proves this statement for \Cesaro averages, but the proof for general weighted averages is identical.} generalization of the Weyl criterion.
\begin{lemma}\label{lem:richter_equivalence}
Let $G$ be a simply connected nilpotent
Lie group, $\Gamma$ a uniform and discrete subgroup of $G$, and $X=G/\Gamma$. 
Let $v:\N\rightarrow G$ be a sequence in $G$ and let $L$ be a rational, closed, normal subgroup of $G$ such that $L^{\circ}\cap Z(G)^{\circ}$ is nontrivial.
Then the following are equivalent:
    \begin{itemize}
   \item The sequence $(v(n)\Gamma)_{n\in \N}$ is uniformly distributed with respect to $W$-averages in $X$.
   \item For each central character $(\varphi,\chi)$ such that $\chi$ is nontrivial on $L^{\circ}\cap Z(G)^{\circ}$
   \begin{equation}\label{eq:richter_equivalence}
        \lim_{N\to\oo}\E_{n\leq N}^W\varphi(v(n)\Gamma)=0.
    \end{equation}
    \end{itemize}
\end{lemma}

In Section \ref{section:richter_results} we will consider yet another type of character for nilmanifolds.
\begin{definition}
    Let $G$ be a nilpotent Lie group, let $\Gamma$ be a discrete and uniform subgroup of $G$, and let $X=G/\Gamma$. A \emph{horizontal character} for $(G,\Gamma)$ is a continuous function $\eta:X\rightarrow \T\subset \C$ such that 
    $$
    \eta(ab\Gamma) = \eta(a\Gamma)\eta(b\Gamma)
    $$
    for all $a,b\in G$.
\end{definition}
Horizontal characters on $G/\Gamma$ correspond to homomorphisms from $G$ to $\R/\Z$ whose kernel contains $\Gamma$. For example, any horizontal character on the Heisenberg nilmanifold has the form $$
\eta \begin{pmatrix}
            1 & a & c\\
            0 & 1 & b\\
            0 & 0 & 1
            \end{pmatrix} = e^{2\pi i (na+mb)}
$$
for $n,m\in \Z$.
\subsection{Preliminaries on weighted averages}

The main ingredient in our proof of Theorem \ref{thm:main} is the  characterization of weighted uniform distribution along Hardy functions given in Theorem \ref{thm:W_ud}, which is a generalization of \cite[Theorem 5.1]{Richter23} and in turn allows us to generalize many other results from \cite{Richter23} in Section \ref{section:richter_results}. In this section we consider a few facts about weighted averages that we will be required for our proof of Theorem \ref{thm:main}.

In Section \ref{section:richter_results}, we need a variant of van der Corput's trick (compare this with Theorem \ref{thm:vdc_trick}, a different variant of var der Corput's trick given in Section \ref{section:uniform_distribution}).

\begin{lemma}[{\cite[Lemma 4.6]{BMR24}}]\label{lem:vdc}
Let $W$ be a function which tends to $\oo$ such that $\Delta W$ is eventually monotone. Additionally, assume that $\lim_{N\to\oo}\frac{\Delta W(N)}{W(N)}=0$. Let $(u_{n,N})_{n,N\in \N}$ be a sequence in a Hilbert space such that $|u_{n,N}|\leq 1$ for all $n,N\in \N$. Then
\begin{equation}\label{eq:vdc_eq}
    \limsup_{N\to\oo}\|\E_{n\leq N}^Wu_{n,N}\|^2 \leq \limsup_{H\to\oo}\limsup_{N\to\oo}\left|\E_{h\leq H}\E_{n\leq N}^W\langle u_{n+h,N},u_{n,N}\rangle\right|.
\end{equation}

\end{lemma}
\begin{remark}\label{remark:vdc}
    We can bound the right hand side of (\ref{eq:vdc_eq}) using the triangle inequality and linearity to see that
\begin{align*}\label{eq:vdc_eq_2}
\limsup_{H\to\oo}\limsup_{N\to\oo}&\left|\E_{h\leq H}\E_{n\leq N}^W\langle u_{n+h,N},u_{n,N}\rangle\right|\\\leq& \limsup_{H\to\oo}\E_{h\leq H}\left(\limsup_{N\to\oo}\left|\E_{n\leq N}^W\langle u_{n+h,N},u_{n,N}\rangle\right|\right).
\end{align*}
\end{remark}

Recall the Stolz-\Cesaro theorem, which is a discrete version of L'H\^optial's rule. This theorem was used Section \ref{section:applications} and will be used again in Section \ref{section:richter_results}. 
\begin{theorem}[{\cite[Problem 70]{Polya_Szego}}]\label{thm:stolz_cesaro}
Let $(a_n)_{n\in \N}$ and $(b_n)_{n\in \N}$ be real valued sequences such that $(b_n)_{n\in \N}$ increases to $\oo$. Then
\begin{equation}\label{eq:stolz_cesaro_derivative}
    \liminf_{n\to\oo}\frac{\Delta a_n}{\Delta b_n}\leq \liminf_{n\to\oo}\frac{a_n}{b_n}\leq \limsup_{n\to\oo}\frac{a_n}{b_n}\leq \limsup_{n\to\oo}\frac{\Delta a_n}{\Delta b_n}.
\end{equation}
In particular, if $\lim_{n\to\oo}\frac{\Delta a_n}{\Delta b_n}$ exists then $\lim_{n\to\oo}\frac{ a_n}{ b_n}=\lim_{n\to\oo}\frac{\Delta a_n}{\Delta b_n}$.
\end{theorem}
 
Finally, we will require the following lemma on uniform \Cesaro averages in Section \ref{section:richter_results}.
\begin{lemma}\label{lem:wd_means_W_ud}
    Let $Y$ be a Banach space and let $(x_n)_{n\in \N}\subset Y$ be a bounded sequence such that the limit $
        \lim_{N-M\to\oo}\frac{1}{N-M}\sum_{n=M}^Nx_n$ exists and is equal to $L\in Y$. Let $(\alpha_{N,n})_{N,n\in \N}$ define a regular matrix method (see Lemma \ref{lem:silverman_toeplitz}) with the additional property that for any $k\in \Z$ and any bounded sequence $(y_{n})_{n\in \N},$
        $$
        \sum_{n=1}^{\oo}\alpha_{N,n}y_{n+k} = \sum_{n=1}^{\oo}\alpha_{N,n}y_n +o_{N\to\oo}(1).
        $$
        Then $\lim_{N\to\oo} \sum_{n=1}^{\oo}\alpha_{N,n}x_{n}=L$. In particular, $\lim_{N\to\oo}\E_{n\leq N}^Wx_n = L$ for any $W$ which eventually increases to $\oo$ and satisfies $\lim_{N\to\oo}\frac{\Delta W(N)}{W(N)}=0$.
\end{lemma}
\begin{proof}
    Let $\epsilon>0$. Pick $K\in \N$ large enough so that if $N-M\geq K$ then 
    $$
    \left|L-\frac{1}{N-M}\sum_{k=M}^Nx_k\right|<\epsilon.
    $$
    Let $y_n = L-\frac{1}{K}\sum_{k=1}^Kx_{n+k}$ so that $|y_n|<\epsilon$ for each $n\in \N$. Then $\left|\sum_{n=1}^N\alpha_{N,n}y_n\right|<\epsilon+o_{N\to\oo}(1)$. By linearity and shift invariance, we also have 
    $$
    \sum_{n=1}^N\alpha_{N,n}y_n= L - \sum_{n=1}^N\alpha_{N,n}x_n  +o_{N\to\oo}(1).
    $$
    Hence $\left| L - \sum_{n=1}^N\alpha_{N,n}x_n\right|<\epsilon+o_{N\to\oo}(1)$ and since $\epsilon$ is arbitrary, we are done. 
\end{proof}

\begin{remark} The proof of Lemma \ref{lem:wd_means_W_ud} given here is similar to the argument given in \cite[Theorem 3.6]{tempered_functions}, which shows the implication $\lim_{N-M\to\oo}\sum_{n=M}^Nx_n = L \implies \lim_{N\to\oo}\E_{n\leq N}^Wx_n = L$ for functions satisfying $W$ such that $\Delta W$ eventually increases to $\oo$ and $\lim_{N\to\oo}\frac{\Delta W(N)}{W(N)}=0$.
\end{remark}

\begin{remark}\label{remark:wd_means_W_ud}
    The same argument used to prove Lemma \ref{lem:wd_means_W_ud} shows that if $d\in \N$ and $(x_{\mathbf{n}})_{\mathbf{n}\in \N^{d}}$ is a bounded sequence such that 
    \begin{equation}
    \frac{1}{(N_1-M_1)\cdots (N_d-M_d)}\sum_{\mathbf{n}\in [M_1,N_1]\times\cdots \times [M_d,N_d]}x_{\mathbf{n}}
    \end{equation}
    converges to $L$ as $N_i-M_i\to\oo$ for all $i$, then 
    $$
    \lim_{N_1,\dots, N_d\to\oo} \sum_{n_1,\dots, n_d\in \N}\alpha_{N_1,\dots, N_d,n_1,\dots, n_d}x_{\mathbf{n}}=L,
    $$
    where $(\alpha_{N_1,\dots, N_d,n_1,\dots, n_d})_{N_1,\dots, N_d,n_1,\dots, n_d\in \N}$ is a matrix of real numbers satisfying 
    \begin{itemize}
        \item For each $i$, if the values of $n_1,\dots, n_d$ and $N_j$ for $j\neq i$ are fixed, then 
        $$
        \lim_{N_i \to \oo}\alpha_{N_1,\dots, N_d,n_1,\dots, n_d} = 0,
        $$
        \item $\lim_{N_1,\dots, N_d\to\oo} \sum_{n_1,\dots, n_d\in \N}\alpha_{N_1,\dots, N_d,n_1,\dots, n_d}=1$,
        \item $\lim_{N_1,\dots, N_d\to\oo} \sum_{n_1,\dots, n_d\in \N}|\alpha_{N_1,\dots, N_d,n_1,\dots, n_d}|<\oo$,
        \item for any $\mathbf{k}\in \N^d$ and any bounded sequence $(y_{\mathbf{n}})_{\mathbf{n}\in \N^d}$
        $$
        \lim_{N_1,\dots, N_d\to\oo}\left|\sum_{n_1,\dots, n_d\in \N}\alpha_{N_1,\dots, N_d,n_1,\dots, n_d}y_{\mathbf{n}+\mathbf{k}} -\sum_{n_1,\dots, n_d\in \N}\alpha_{N_1,\dots, N_d,n_1,\dots, n_d}y_{\mathbf{n}} \right|=0.
        $$
    \end{itemize}
\end{remark}

\section{Generalizations of results from \texorpdfstring{\cite{BMR24}}{[BMR24]}}\label{section:BMR_results}

 In this section we borrow some ideas developed in \cite{BMR24} to reduce Theorem \ref{thm:main} to statements about uniform distribution on nilmanifolds. More specifically, we will use a result about uniformity seminorms (Definition \ref{def:uniformity_seminorms}, Theorem \ref{thm:generalized_BMR_4.2}) and a result about nilsystems (Lemma \ref{lem:bmr_sec_5}) to show that it suffices to prove Theorem A in the case when $(X,\mathscr{B},\mu,T)$ is an ergodic nilsystem (Theorem \ref{thm:true_for_nilsystem} below). Then we will use a fact about summation methods (Lemma \ref{lem:variant_of_modulo_averages}) to show that it is enough to prove Theorem \ref{thm:true_for_nilsystem} with an additional assumption (Theorem \ref{thm:true_for_nilsystems_and_also}). Lastly, we use a lemma which was proven in \cite{BMR24} (Lemma \ref{lem:floor_removal}) to reduce Theorem \ref{thm:true_for_nilsystems_and_also} to a statement about uniform distribution on nilmanifolds (Theorem \ref{thm:BMR_uniform_distribution_result}), that we then prove using the main result of Section \ref{section:richter_results} (Theorem \ref{thm_generalized_richter_D})

 We begin by considering Theorem \ref{thm:generalized_BMR_4.2}, a version of Theorem \ref{thm:main} in which we have control over certain seminorms.

\begin{definition}\label{def:uniformity_seminorms}
    For a measure preserving system $(X,\mathscr{B},\mu,T)$, the \emph{uniformity seminorms} on $L^{\oo}(X)$ are defined inductively by 
    \begin{equation*}
        \|h\|_0 = \int_X h~d\mu_X\quad \text{ and }\quad \|h\|_{s}^{2^s} = \lim_{N\to\oo}\E_{n\leq N}\|\overline{h}\cdot T^nh\|_{s-1}^{2^{s-1}}
    \end{equation*}
    for $h\in L^{\oo}(X)$ and $s\in \N$.
\end{definition}
\begin{theorem}\label{thm:generalized_BMR_4.2}
Let $\mathcal{H}$ be a Hardy field and $W \in \mathcal{H}$ with
$1 \prec \log W(x) \prec x$.
Let $\mathscr{F} = \{f_1,\ldots,f_{\ell}\}$ be a collection of functions such that $ f_1,\ldots,f_{\ell}\in \mathcal{H}$, $\mathscr{F}$ satisfies property (\ref{eq:W_compatible}) (Definition \ref{def:W_compatible}), $|f_1(x)| \preceq \cdots \preceq |f_{\ell}(x)|$, and $
\lim_{x\to\infty} |f_{\ell}(x)|=\lim_{x\to\infty} |f_{\ell}(x)-f_i(x)|
=
\infty
\quad \text{for every } i<\ell$. Then there exists $s\in\mathbb{N}$ such that for any invertible
measure preserving system $(X,\mathcal{B},\mu,T)$ and any $h_{\ell}\in L^\infty(X)$ with $\|h_{\ell}\|_s=0$ we have
\begin{equation}\label{eq:4.2_goal}
\sup_{h_1,\ldots,h_{\ell-1}\in L^\infty}
\ \sup_{a\in \ell^\infty}
\left\|
\E_{n\leq N}^Wa(n)
\prod_{i=1}^{\ell} T^{[f_i(n)]}h_i
\right\|_{L^2}
=
o_{N\to\infty}(1),
\end{equation}
where the suprema are taken over all functions
$h_1,\ldots,h_{\ell-1}\in L^\infty(X)$ with
$\|h_i\|_{L^\infty}\le 1$
and all $a\in \ell^\infty(\mathbb{N})$ with
$\|a\|_{\ell^\infty}\le 1$.
\end{theorem}

A variant of the preceding theorem appears in \cite[Theorem 4.2]{BMR24} with the more restrictive assumptions that $W(x)\preceq x$ and that $f_1,\dots, f_{\ell}$ satisfy property (P), defined as follows.

\begin{definition}[\cite{BMR24}]\label{def:property_P}
    Let $\mathcal{H}$ be a Hardy field and let $W,f_1,\dots, f_{\ell}\in \mathcal{H}$. Suppose that $1\prec \log W(x)\prec x$, and that $f_1,\dots, f_{\ell}$ are each subpolynomial. We say that $\{f_1,\dots, f_{\ell}\}$ satisfies property (P) if $W$ is compatible (Definition \ref{def:W_compatible}) with each $f\in \text{Span}_{\mathbb{R}}\{f_i^{(m)}:1\leq i\leq \ell, m\geq 0\}$.
\end{definition}
\begin{remark}\label{remark:P_WP_star}
    If $f_1,\dots, f_{\ell}$ each have degree $1$ then property (P) and property (\ref{eq:W_compatible}) are equivalent.
\end{remark}

\begin{example} 
Recall from Definition \ref{def:W_compatible} that $\{f_1,\dots, f_{\ell}\}$ is said to satisfy property (\ref{eq:W_compatible}) if $W$ is compatible with each $f\in \text{Span}_{\mathbb{R}}\{f_1,\dots, f_{\ell}\}$. Property (P) is more restrictive than property (\ref{eq:W_compatible}) but this allows for sets satisfying Property (P) to be invariant under differentiation, meaning that if $\{f_1,\dots, f_{\ell}\}$ satisfies property (P) then $\{f_1,\dots, f_{\ell},f_1',\dots, f_{\ell}'\}$ also satisfies property (P). The same is not true for property (\ref{eq:W_compatible}). For example, let $W(x)=x$ and let $f_1(x) = x\log (x)$ so that $f_1'(x) = \log(x)+1$. Then $\{f_1\}$ satisfies property (\ref{eq:W_compatible}) but $\{f_1'\}$ does not.
\end{example}

Theorem \ref{thm:generalized_BMR_4.2} is proven using a PET induction argument (see \cite{Bergelson_PET}). However, unlike the proof of \cite[Theorem 4.2]{BMR24}, we must take extra steps to account for the fact that $f$ being compatible with $W$ does not guarantee that $f'$ is compatible with $W$. Our proof of Theorem \ref{thm:generalized_BMR_4.2} is inspired by arguments in \cite{donoso25} (note that when $W(x)=x$, Theorem \ref{thm:generalized_BMR_4.2} follows from \cite[Theorem 10.1]{donoso25}).

For the base case of the induction, we assume that each $f_i$ is sublinear. For a proof of this case, we direct the reader to \cite[Theorem 4.3]{BMR24}. More precisely, the formulation of \cite[Theorem 4.3]{BMR24} uses property (P), which is equivalent to property (\ref{eq:W_compatible}) when each $f_i$ grows sublinearly (see Remark \ref{remark:P_WP_star}).

\begin{theorem}[{\cite[Theorem 4.3]{BMR24}}]\label{thm:BMR24_sublinear}
Let $\ell \in \mathbb{N}$.
Let $\mathcal{H}$ be a Hardy field and $W \in \mathcal{H}$ with
$1 \prec \log W(x)\prec x$.
Assume $f_1,\ldots,f_{\ell} \in \mathcal{H}$ satisfy property (\ref{eq:W_compatible}),
$|f_1(x)| \preceq \cdots \preceq |f_{\ell}(x)| \preceq x$, and $\lim_{x\to\infty} |f_{\ell}(x)|
=
\lim_{x\to\infty} |f_{\ell}(x)-f_i(x)|
=
\infty
\quad \text{for every } i<\ell$
Then there exists a constant $C_{\ell}>0$, depending only on
$\ell$ and $f_1,\ldots,f_{\ell}$, such that for any invertible
measure preserving system $(X,\mathcal{B},\mu,T)$ and any $h_{\ell}\in L^\infty(X)$ we have
\begin{equation}
\sup_{h_1,\ldots,h_{\ell-1}\in L^\infty}
\ \sup_{a\in \ell^\infty}
\left\|
\E_{n\leq N}^Wa(n)
\prod_{i=1}^{\ell} T^{[f_i(n)]}h_i
\right\|_{L^2}
\le
C_{\ell} \|h_{\ell}\|_{\ell+1}
+
o_{N\to\infty}(1),
\end{equation}
where the suprema are taken over all functions
$h_1,\ldots,h_{\ell-1}\in L^\infty(X)$ with
$\|h_i\|_{L^\infty}\le 1$
and all $a\in \ell^\infty(\mathbb{N})$ with
$\|a\|_{\ell^\infty}\le 1$.
\end{theorem}

Before we address the induction step, we consider a few helpful lemmas. To begin, the following lemma is a straightforward generalization of  \cite[Lemma 5.4]{donoso25} (see also \cite[Lemma 3.3]{TSINAS2023a}), from which it follows that it suffices to prove seminorm bounds for averages along intervals.

\begin{definition}
    For an interval $I\subset \mathbb{N}$ and a sequence $(x_n)_{n\in \N}$, define 
    $$
    \E_{n\in I}x_n = \frac{1}{|I|}\sum_{n\in I}x_n.
    $$
\end{definition}
\begin{lemma}[{cf. \cite[Lemma 5.4]{donoso25}}]\label{lem:donoso}
    Let $\mathcal{H}$ be a Hardy field. Let $W,L\in \mathcal{H}$ satisfy $1\prec \log W(x)\prec x$ and $1\prec L(x)\prec ((\log W)'(x))^{-1}$. Let $(X,\mathscr{B},\mu)$ be a probability space and let $\ell,d\in \N$. For $n,N\in \N$, let $A_{n,N}:(L^{\oo}(X))^{\ell}\rightarrow L^{\oo}(X)$ be a function that satisfies $\|A_{n,N}(h_1,\dots, h_{\ell})\|_{{\oo}}\leq 1$ whenever $\|h_i\|_{{\oo}}\leq 1$ for $i\in \{1,\dots, \ell\}$. Then
    \begin{align*}
        &\left(\limsup_{N\to\oo}\sup_{\|h_1\|_{\oo},\dots, \|h_{\ell}\|_{\oo}\leq 1}\sup_{\|a\|_{\ell^{\oo}}\leq 1}\left\|\E_{n\leq N}^W a(n)\cdot A_{n,N}(h_1,\dots, h_{\ell})\right\|_{L^2}\right)^d\\
        \leq& \limsup_{R\to\oo}\sup_{\|h_1\|_{\oo},\dots, \|h_{\ell}\|_{\oo}\leq 1}\sup_{\|a\|_{\ell^{\oo}}\leq 1}\E_{N\leq R}^W\left\|\E_{n\in [N,N+\floor{L(N)}]} a(n)\cdot A_{n,R}(h_1,\dots, h_{\ell})\right\|_{L^2}^d
    \end{align*}
\end{lemma}
\begin{proof}
    Pick sequences $(a_N(n))_{n\in \N}$ and $(h_{i,N})_{N\in \N}$ for $i\in \{1,\dots, \ell\}$, each of norm at most 1, such that 
    \begin{align*}
  &\limsup_{N\to\oo}\left\|\E_{n\leq N}^W a_N(n)\cdot A_{n,N}(h_{1,N},\dots, h_{\ell,N})\right\|_{L^2}\\
  =& \limsup_{N\to\oo}\sup_{\|h_1\|_{\oo},\dots, \|h_{\ell}\|_{\oo}\leq 1}\sup_{\|a\|_{\ell^{\oo}}\leq 1}\left\|\E_{n\leq N}^W a(n)\cdot A_{n,N}(h_{1},\dots, h_{\ell})\right\|_{L^2}.
    \end{align*}
Put $B_N(n) = a_N(n)\cdot A_{n,N}(h_{1,N},\dots, h_{\ell,N})$. We will show that 
\begin{equation}\label{eq:donoso_lem_goal}
\E_{N\leq R}^W\E_{n\in [N,N+\floor{L(N)}]}B_R(n)=\E_{n\leq R}^W B_R(n)+o_{R\to\oo}(1)
\end{equation}
so that the statement of the lemma follows from H\"older's inequality.
Let $ \widetilde{L}(n) = \min\{N\in \N: N+\floor{L(N)}\geq n\}$  so that $\floor{L(\widetilde{L}(N))} = (N-\widetilde{L}(N))\cdot (1+o_{N\to\oo}(1))$. Then we have
\begin{align}
    &\E_{N\leq R}^W\E_{n\in [N,N+\floor{L(N)}]}B_R(n) = \frac{1}{W(R)}\sum_{N=1}^R\Delta W(N)\cdot \frac{1}{\floor{L(N)}}\sum_{n=N}^{N+\floor{L(N)}]}B_R(n)\\
    =&\frac{1}{W(R)}\sum_{n=1}^{R+\floor{L(R)}}B_R(n)\cdot \sum_{N=1}^R\frac{\Delta W(N)}{\floor{L(N)}}\cdot 1_{n\in [N,N+\floor{L(N)}]}\\
    =&\frac{1}{W(R)}\sum_{n=1}^{R}B_R(n)\cdot \sum_{N=\widetilde{L}(n)}^n\frac{\Delta W(N)}{\floor{L(N)}}+\frac{1}{W(R)}\sum_{n=R+1}^{R+\floor{L(R)}}B_R(n)\cdot \sum_{N=\widetilde{L}(n)}^n\frac{\Delta W(N)}{\floor{L(N)}}.\label{eq:donoso_goal_2}
\end{align}
We will show that $\sum_{N=\widetilde{L}(n)}^n\frac{\Delta W(N)}{\floor{L(N)}} = \Delta W(n)\cdot (1+o_{n\to\oo}(1))$, from which (\ref{eq:donoso_lem_goal}) follows because (\ref{eq:donoso_goal_2}) becomes
$$
\frac{1}{W(R)}\sum_{n=1}^{R}\Delta W(n)B_R(n)+o_{R\to\oo}(1) =\E_{n\leq R}^WB_R(n)+o_{R\to\oo}(1).
$$
To this end, we may approximate $\sum_{N=\widetilde{L}(n)}^n\frac{ \Delta W(N)}{\floor{L(N)}}$ by $\int_{\widetilde{L}(n)}^n\frac{W'(x)}{L(x)}~dx$ since $\Delta W(N)=W'(N)\cdot (1+o_{N\to\oo}(1))$ and the function $x\mapsto \frac{W'(x)}{L(x)}$ is eventually monotone.

We need to consider two cases, but first we must make note of the following claims:

\textbf{Claim 1:} Let $F$ be a function which belongs to a Hardy field and satisfies $1\prec \log F(x)\prec x$. Suppose that $I_N = [N,N+|I_N|]$ is a sequence of intervals which satisfies  $\limsup_{N\to\oo}\frac{|I_N|\cdot F'(N)}{F(N)}=0$. Then $F(x)=F(N)\cdot (1+o_{N\to\oo}(1))=F(N+|I_N|)\cdot (1+o_{N\to\oo}(1))$ uniformly for $x\in I_N$.

To see why this claim is true, note that $(\log F)'$ is eventually decreasing, meaning that $\log F$ is concave, and so 
$$
\log F(N+|I_N|)-\log F(N)\leq |I_N|\cdot (\log F)'(N) = \frac{|I_N|F'(N)}{F(N)}\to 0
$$
as $N\to\oo$. Then it follows that $\lim_{N\to\oo}\frac{F(N+|I_N|)}{F(N)} =1$
and hence $F(x)=F(N)\cdot (1+o_{N\to\oo}(1))=F(N+|I_N|)\cdot (1+o_{N\to\oo}(1))$ uniformly for $x\in I_N$. This proves the claim.

\textbf{Claim 2:} Let $F$ be a function which belongs to a Hardy field and satisfies $1\prec F(x)\prec x$. Suppose that $I_N = [N,N+|I_N|]$ is a sequence of intervals with 
$$
\limsup_{N\to\oo}\frac{NF'(N)}{F(N)}\log\left(1+\frac{|I_N|}{N}\right)=0.
$$
Then $F(x)=F(N)\cdot (1+o_{N\to\oo}(1))=F(N+|I_N|)\cdot (1+o_{N\to\oo}(1))$ uniformly for $x\in I_N$.

To see why this claim is true, note that 
\begin{align*}
    \log(F(N+|I_N|))-\log(F(N)) = \int_{N}^{N+|I_N|}\frac{F'(x)}{F(x)}~dx=\int_{N}^{N+|I_N|}\frac{xF'(x)}{F(x)}\cdot \frac{dx}{x}.
\end{align*}
The function $x\mapsto \frac{xF'(x)}{F(x)}$ is eventually monotone because it is a Hardy function. Additionally, we have
$$
\lim_{x\to\oo} \frac{xF'(x)}{F(x)} =  \lim_{x\to\oo}\frac{\frac{F'(x)}{F(x)}}{1/x} =  \lim_{x\to\oo}\frac{\log F(x)}{\log x}\leq 1
$$
and so $\lim_{x\to\oo} \frac{xF'(x)}{F(x)} $ monotonically converges to a nonnegative constant $C\in [0,1]$. If $C=0$ then it must be that $\frac{xF'(x)}{F(x)}$ is eventually decreasing. So, whether or not $C$ is zero, we have for all large enough $N$ and all $x\in I_N$ that $ \frac{xF'(x)}{F(x)}\leq 2\cdot \frac{NF'(N)}{F(N)}$. It follows that 
\begin{align*}
\int_{N}^{N+|I_N|}\frac{xF'(x)}{F(x)}\cdot \frac{dx}{x}\leq 2\frac{NF'(N)}{F(N)}\cdot \int_{N}^{N+|I_N|}\frac{dx}{x} = 2\frac{NF'(N)}{F(N)}\cdot \log\left(1+\frac{|I_N|}{N}\right)\to 0
\end{align*}
as $N\to\oo$. This shows that the claim is true.

\textbf{Claim 3:} Let $G$ be a Hardy function with $x\preceq G(x)$. Suppose that $\int_1^{\oo}\frac{1}{G(x)}~dx=\oo$. Then $\lim_{x\to\oo}\frac{x\cdot \left(G(x)/x\right)'}{G(x)/x}\cdot \log (G(x)/x)=0$.
This claim is true because 
$$
\frac{x\cdot \left(G(x)/x\right)'}{G(x)/x}\cdot \log (G(x)/x) = \frac{d}{du}\hat{G}(u)\bigg\vert_{u=\log x}
$$
where $\hat{G}(u) = \frac{1}{2}\left(\log\left(G(e^u)e^{-u}\right)\right)^2$. If the derivative of $\hat{G}$ does not tend to $0$ then there exists $\epsilon>0$ such that $\hat{G}(u)\geq \epsilon u$ for all sufficiently large $u$, from which it follows that $\log\left(\frac{G(x)}{x}\right)\geq \sqrt{2\epsilon \log x}$ and hence 
\begin{align*}
&\int_1^{\oo}\frac{1}{G(x)}~dx =  \int_1^{\oo}\frac{1}{x}e^{-\log\left(\frac{G(x)}{x}\right)}~dx
\\
\leq & \int_1^{\oo}\frac{1}{x}e^{-\sqrt{2\epsilon \log x}}~dx+O(1)=\int_e^{\oo}e^{-\sqrt{2\epsilon u}}~du+O(1)<\oo.
\end{align*}
This shows that the claim is true.

In both of the following cases, we will take $I_N = [N,N+\floor{L(N)}]$ for $N\in \N$. Additionally, we will use the fact that $W(x) = W(N)\cdot (1+o_{N\to\oo}(1))$ uniformly for $x\in I_N$ by Claim 1.

Consider the case when $\log x\prec \log W(x)$, so that $L(x)\prec x$. Taking $F(x)=L(x)$ and $I_N = [N,N+\floor{L(N)}]$, Claim 1 gives that $L(x) =L(\widetilde{L}(n))\cdot (1+o_{N\to\oo}(1))$ uniformly for $x\in I_{\widetilde{L}(n)}$. 
Additionally, we have that $G(x)=\frac{W(x)}{W'(x)}$ satisfies $G(x)=G(n)\cdot (1+o_{N\to\oo}(1))$ uniformly for $x\in I_{\widetilde{L}(n)}$ by Claim 2, since
$$
\limsup_{N\to\oo}\frac{NF'(N)}{F(N)}\log\left(1+\frac{|I_N|}{N}\right)=0.
$$ Altogether, this gives us that 
\begin{align}\label{eq:W'/L_initial}
&\int_{\widetilde{L}(n)}^n\frac{W'(x)}{L(x)}~dx = \int_{\widetilde{L}(n)}^nL(x)^{-1}\cdot \left(\frac{W(x)}{W'(x)}\right)^{-1}\cdot W(x)~dx \\
=& (1+o_{n\to\oo}(1))\cdot \int_{\widetilde{L}(n)}^nL(\widetilde{L}(n))^{-1}\cdot \left(\frac{W(n)}{W'(n)}\right)^{-1}\cdot W(n)~dx \\
=&  (1+o_{n\to\oo}(1))\cdot  (n-\widetilde{L}(n))\cdot L(\widetilde{L}(n))^{-1}\cdot W'(n)=(1+o_{n\to\oo}(1))\cdot W'(n),\label{eq:W'/L_final}
\end{align}
as desired.

Lastly, we consider the case $\log W(x)\preceq \log x$. Applying Claim 2 and Claim 3 with $G(x) = W(x)/W'(x)$ and $F(x)=G(x)/x$ shows that $F(x) = F(n)\cdot (1+o_{n\to\oo}(1))$ uniformly for $x\in I_{\widetilde{L}(n)}$, from which it follows that $W'(x) = nW'(n)/x\cdot (1+o_{N\to\oo}(1))$ uniformly for $x\in I_{\widetilde{L}(n)}$.

A similar statement is true for $L$. Namely, $L(x) = x\cdot L(\widetilde{L}(n))/\widetilde{L}(n)\cdot  (1+o_{N\to\oo}(1))$ uniformly for $x\in I_{\widetilde{L}(n)}$. This is clear for $L(x)\prec x$, and for $x\preceq L(x)$ we may apply Claim 3 to $L(x)/x$ and Claim 2 to $L(x)$ as we did above. This shows that 
\begin{align}
&\int_{\widetilde{L}(n)}^n\frac{W'(x)}{L(x)}~dx = (1+o_{n\to\oo}(1))\cdot \int_{\widetilde{L}(n)}^n\left(\frac{xL(\widetilde{L}(n))}{\widetilde{L}(n)}\right)^{-1}\cdot \left(\frac{nW'(n)}{x}\right)~dx \\
=& (1+o_{n\to\oo}(1))\cdot \frac{n\widetilde{L}(n)W'(n)}{L(\widetilde{L}(n))}\int_{\widetilde{L}(n)}^n\frac{1}{x^2}~dx = (1+o_{n\to\oo}(1))\cdot \frac{n\widetilde{L}(n)W'(n)}{L(\widetilde{L}(n))}\left(\frac{1}{\widetilde{L}(n)}-\frac{1}{n}\right)\\
=&  (1+o_{n\to\oo}(1))\cdot  \frac{n-\widetilde{L}(n)}{L(\widetilde{L}(n))}\cdot W'(n)=(1+o_{n\to\oo}(1))\cdot W'(n).
\end{align}
This completes the proof.
\end{proof}
We also need another variant of van der Corput's trick.
\begin{lemma}\label{lem:almost_final_vdc}
 Let $(u_{N}(n))_{n,N\in \N}$ be a sequence in a Hilbert space such that $\|u_{N}(n)\|\leq 1$ for all $n,N\in \N$. Let $\mathcal{H}$ be a Hardy field and suppose that $q,L\in \mathcal{H}$ satisfy $1\prec q(x)\prec L(x)$. Then for any $M\in \N$ we have 
\begin{align*}
&\|\E_{n\in [N,N+\floor{L(N)}]}  u_N(n)\|^2 \\
\leq &\frac{2}{M}\sum_{m=1}^{M-1}(1-\frac{m}{M})\left|\E_{n\in [N,N+\floor{L(N)}]}\langle u_{N}(n+\floor{q(N)}m), u_{N}(n)\rangle\right|+o_{N\to\oo}(1)+O\left(\frac{1}{M}\right).
\end{align*}
\end{lemma}
\begin{proof}
     We prove the following more general statement: Let $I$ be an interval and $(v_n)_{n\in \N}$ be a collection of vectors in a Hilbert space, each with norm at most $1$. Then for any $M,h,K\in \N$ with $Mh\leq K$ we have 
    \begin{align}
        \|\E_{n\in I}v_n\|^2\leq \frac{2}{M}\sum_{m=1}^{M-1}(1-\frac{m}{M})\left|\E_{n\in I}\langle v_{n+mh},v_n\rangle\right|+O\left(\frac{Mh}{K}\right)+O\left(\frac{1}{M}\right),
    \end{align}
    where $I\subset \N$ is an interval of length $K$. 
    The argument proceeds along similar lines as the usual proof of van der Corput's trick. For each $m\in \{0,\dots, M-1\}$, 
    \begin{equation*}
        \left\|\frac{1}{K}\sum_{n\in I}v_n- \frac{1}{K}\sum_{n\in I}v_{n+mh}\right\|\leq \frac{2hm}{K}
    \end{equation*}
    and so 
    \begin{equation*}
        \left\|\frac{1}{K}\sum_{n\in I}v_n- \frac{1}{KM}\sum_{m=0}^{M-1}\sum_{n\in I}v_{n+mh}\right\|\leq \frac{1}{M}\sum_{m=0}^{M-1}\frac{2hm}{K} = O\left(\frac{Mh}{K}\right).
    \end{equation*}
    Next, use the fact that $v\mapsto \|v\|^2$ is a convex function to see that
    \begin{align*}
&\left\| \frac{1}{KM}\sum_{m=0}^{M-1}\sum_{n\in I}v_{n+mh}\right\|^2\leq \frac{1}{K}\sum_{n\in I}\left\| \frac{1}{M}\sum_{m=0}^{M-1}v_{n+mh}\right\|^2\\
= &\frac{1}{K}\sum_{n\in I}\langle\frac{1}{M}\sum_{m=0}^{M-1}v_{n+mh},\frac{1}{M}\sum_{m=0}^{M-1}v_{n+mh}\rangle
= \frac{1}{K}\sum_{n\in I}\frac{1}{M^2}\sum_{i,j=0}^{M-1}\langle v_{n+ih},v_{n+jh}\rangle\\
=&\E_{n\in I}\frac{1}{M^2}\sum_{i,j=0}^{M-1}\langle v_{n+ih},v_{n+jh}\rangle = \frac{1}{M^2}\sum_{i,j=0}^{M-1}\E_{n\in I}\langle v_{n+ih},v_{n+jh}\rangle .
    \end{align*}
   There are $M$ terms with $i=j$, each of norm at most 1. For $i\neq j$, we can note that 
   $$
   \E_{n\in I} \langle v_{n+ih},v_{n+jh}\rangle +\E_{n\in I}\langle v_{n+jh},v_{n+ih}\rangle = 2\text{Re}\E_{n\in I}\langle v_{n+ih},v_{n+jh}\rangle\leq 2|\E_{n\in I}\langle v_{n+ih},v_{n+jh}\rangle|.
   $$
   Hence
    \begin{align*}
\frac{1}{M^2}\sum_{i,j=0}^{M-1}\E_{n\in I}\langle v_{n+ih},v_{n+jh}\rangle =& \frac{1}{M^2}\left(\sum_{i<j\in \{0,\dots, M-1\}}2|\E_{n\in I}\langle v_{n+ih},v_{n+jh}\rangle|+O(M)\right)\\
    \leq &\frac{1}{M^2}\sum_{m=0}^M\sum_{k=1}^{M-1-m}2|\E_{n\in I}\langle v_{n+mh},v_{n+(m+k)h}\rangle|+O\left(\frac{1}{M}\right).
    \end{align*}
    For each $m$, rewrite the sum $\sum_{k=1}^{M-1-m}2|\E_{n\in I}\langle v_{n+mh},v_{n+(m+k)h}\rangle|$ as 
    $$
    \sum_{k=1}^{M-1-m}2|\E_{n\in I}\langle v_{n},v_{n+kh}\rangle|+O\left(\frac{mh}{K}\right)
    $$
    by making the substitution $n\mapsto n+mh$ and noting that at most $2mh$ terms are added or deleted. Then
    \begin{align*}
        &\frac{1}{M^2}\sum_{m=0}^M\sum_{k=1}^{M-1-m}2|\E_{n\in I}\langle v_{n+mh},v_{n+(m+k)h}\rangle|\\
        =&\frac{1}{M^2}\sum_{m=0}^M\left(\sum_{k=1}^{M-1-m}2|\E_{n\in I}\langle v_{n},v_{n+mh}\rangle|+O\left(\frac{mh}{K}\right)\right)\\
        =&\frac{1}{M^2}\sum_{m=0}^M\sum_{k=1}^{M-1-m}2|\E_{n\in I}\langle v_{n},v_{n+mh}\rangle|+O\left(\frac{Mh}{K}\right)\\
        =&\frac{2}{M^2}\sum_{m=0}^M(M-m)|\E_{n\in I}\langle v_{n},v_{n+mh}\rangle|+O\left(\frac{Mh}{K}\right)\\
        =&\frac{2}{M}\sum_{m=0}^M\left(1-\frac{m}{M}\right)|\E_{n\in I}\langle v_{n},v_{n+mh}\rangle|+O\left(\frac{Mh}{K}\right).
    \end{align*}
    Altogether, we have 
    \begin{align*}
         \left\|\E_{n\in I}v_n\right\|^2\leq & 
        \left\| \frac{1}{KM}\sum_{m=0}^{M-1}\sum_{n\in I}v_{n+mh}\right\|^2+ O\left(\frac{Mh}{K}\right)\\
        \leq &\frac{2}{M}\sum_{m=0}^M\left(1-\frac{m}{M}\right)|\E_{n\in I}\langle v_{n},v_{n+mh}\rangle|+O\left(\frac{1}{M}\right)+O\left(\frac{Mh}{K}\right).
    \end{align*}
    as desired. The statement of the lemma follows by taking $I = [N,N+\floor{L(N)}]$, $h = \floor{q(N)}$ and $v_n = u_N(n)$.
\end{proof}

We will also need a technical statement about swapping out exponents in multiple ergodic averages.
\begin{lemma}\label{lem:partitions}
Let $L:\N\rightarrow \R$ be a function that tends to $\oo$. Let $\ell\in \N$ and let $F_{i,N},G_{i,N}:\Z\rightarrow \Z$ be functions for $i\in \{1,\dots, \ell\}$ and $N\in \N$. Suppose that the set $\{F_{i,N}(n)-G_{i,N}(n): N\in \N, n\in [N,N+\floor{L(N)}, i\in \{1,\dots, \ell\}\}$ is finite. Let $(X,\mathcal{B},\mu,T)$ be a measure preserving system, let $s\in \N$, and let $h_{\ell}\in L^\infty(X)$ with $\|h_{\ell}\|_s=0$. Then
\begin{equation}\label{eq:partition_1}
\sup_{h_1,\ldots,h_{\ell-1}\in L^\infty}
\ \sup_{a\in \ell^\infty}
\left\|
\E_{n\in [N,N+\floor{L(N)}]}a(n)
\prod_{i=1}^{\ell} T^{F_{i,N}(n)}h_i
\right\|_{L^2}
=
o_{N\to\infty}(1)
\end{equation}
if and only if 
\begin{equation}\label{eq:partition_2}
\sup_{h_1,\ldots,h_{\ell-1}\in L^\infty}
\ \sup_{a\in \ell^\infty}
\left\|
\E_{n\in [N,N+\floor{L(N)}]}a(n)
\prod_{i=1}^{\ell} T^{G_{i,N}(n)}h_i
\right\|_{L^2}
=
o_{N\to\infty}(1).
\end{equation}
\end{lemma}
\begin{proof}
    Suppose that (\ref{eq:partition_1}) holds. Let $(a_N(n))_{n,N\in \N}$ and $(h_{i,N})_{N\in \N}$ for $i\in \{1,\dots, \ell-1\}$ be sequences which attain the suprema on the left-hand side of (\ref{eq:partition_2}) (also let $h_{\ell,N}=h_{\ell}$ for all $N\in \N$). In order to prove that (\ref{eq:partition_2}) holds, it suffices to show that 
    \begin{equation}\label{eq:partition_3}
\left\|
\E_{n\in [N,N+\floor{L(N)}]}a_N(n)
\prod_{i=1}^{\ell} T^{G_{i,N}(n)}h_{i,N}
\right\|_{L^2}
=
o_{N\to\infty}(1).
\end{equation}
    
 Put $S_N(\mathbf{v}) = \{n\in \N: (F_{1,N}(n)-G_{1,N}(n),\dots, F_{\ell,N}(n)-G_{\ell,N}(n)) = \mathbf{v}\}$ for $N\in \N$ and $\mathbf{v} \in \Z^{\ell}$. Then for each $N\in \N$, $(S_N(\mathbf{v}))_{\mathbf{v}\in \Z^{\ell}}$ is a partition of $\N$ with only finitely many nonempty cells. In order to prove (\ref{eq:partition_3}), it is enough to show that 
\begin{equation}
\left\|
\E_{n\in [N,N+\floor{L(N)}]}1_{S_N(\mathbf{v}_N)}(n)\cdot a_N(n)
\prod_{i=1}^{\ell} T^{G_{i,N}(n)}h_{i,N}
\right\|_{L^2}
=
o_{N\to\infty}(1),
\end{equation}
where $(\mathbf{v}_N)_{N\in \N}\subseteq \Z^{\ell}$ is an arbitrary sequence with $\mathbf{v}_N = (v_{1,N},\dots, v_{\ell,N})$ for $N\in \N$. Put $\widetilde{a}_N(n)=1_{S_N(\mathbf{v}_N)}(n)\cdot a_N(n)$ and $\widetilde{h}_{i,N}= T^{-v_{i,N}}h_{i,N}$ for $i\in \{1,\dots, \ell\}$ and $N\in \N$. Note that $\|\widetilde{h}_{\ell,N}\|_s=0$ for each $N\in \N$ and that the sequence $(\widetilde{h}_{\ell,N})_{N\in \N}$ contains only finitely many functions and so (after further refining the partition $(S_N(\mathbf{v}))_{\mathbf{v}\in \Z^{\ell}}$) we can assume that $(\widetilde{h}_{\ell,N})_{N\in \N}$ is a constant sequence. Then
\begin{align*}
&\left\|
\E_{n\in [N,N+\floor{L(N)}]}1_{S_N(\mathbf{v}_N)}(n)\cdot a_N(n)
\prod_{i=1}^{\ell} T^{G_{i,N}(n)}h_{i,N}
\right\|_{L^2} \\
= &\left\|
\E_{n\in [N,N+\floor{L(N)}]}1_{S_N(\mathbf{v}_N)}(n)\cdot a_N(n)
\prod_{i=1}^{\ell} T^{G_{i,N}(n)+v_{i,N}}\widetilde{h}_{i,N}
\right\|_{L^2}\\
=& \left\|
\E_{n\in [N,N+\floor{L(N)}]}\widetilde{a}_N(n)
\prod_{i=1}^{\ell} T^{F_{i,N}(n)}\widetilde{h}_{i,N}
\right\|_{L^2}=o_{N\to\oo}(1),
\end{align*}
where the last equality is true by (\ref{eq:partition_1}). This proves (\ref{eq:partition_2}).

By symmetry, we also know that if (\ref{eq:partition_2}) holds then (\ref{eq:partition_1}) holds and so we are done.
\end{proof}

Another helpful fact that we will use to prove Theorem \ref{thm:generalized_BMR_4.2} comes from \cite[Proposition 10.3]{donoso25}. Define polynomials $p_{i,N}(n) = \sum_{j=0}^{\deg(f_i)}(n-N)^j\cdot f_i^{(j)}(N)/j!$. By \cite[Proposition 10.3]{donoso25} we know that for any $L\in \mathcal{H}$ which satisfies $L(x)\prec (f_i^{(\deg^*(f_i))}(x))^{-1/(\deg^*(f_i))}$ for all $i$ (a sufficient condition for this is $L(x)\prec ((\log W)'(x))^{-1}$) we have
\begin{equation} \label{eq:poly_approx}
\lim_{N\to\oo}\max_{n\in \{N,\dots, N+\floor{L(N)}\}}|f_i(n)-p_{i,N}(n)|=0.
\end{equation}
Lastly, we will apply Lemma \ref{lem:partitions} to obtain a statement about polynomial approximations.
\begin{lemma}\label{lem:polynomial_on_intervals}
    Let $\mathcal{H}$ be a Hardy field and let $L \in \mathcal{H}$ be a function which tends to $\oo$. Let $\ell\in \N$, let $f_1,\ldots,f_{\ell}\in \mathcal{H}$ be functions, and let $(p_{i,N}(x))_{i,N}\subset \R[x]$ be sequences of polynomials for $i\in \{1,\dots, \ell\}$. Suppose that 
\begin{equation} \label{eq:poly_approx_2}
\lim_{N\to\oo}\max_{n\in \{N,\dots, N+\floor{L(N)}\}}|f_i(n)-p_{i,N}(n)|=0
\end{equation}
for each $i\in \{1,\dots, \ell\}$. Let $(X,\mathcal{B},\mu,T)$ be a measure preserving system, let $s\in \N$, and let $h_{\ell}\in L^\infty(X)$ with $\|h_{\ell}\|_s=0$. Then 
\begin{equation}\label{eq:partition_4}
\sup_{h_1,\ldots,h_{\ell-1}\in L^\infty}
\ \sup_{a\in \ell^\infty}
\left\|
\E_{n\in [N,N+\floor{L(N)}]}a(n)
\prod_{i=1}^{\ell} T^{\round{f_i(n)}}h_i
\right\|_{L^2}
=
o_{N\to\infty}(1)
\end{equation}
if and only if 
\begin{equation}\label{eq:partition_5}
\sup_{h_1,\ldots,h_{\ell-1}\in L^\infty}
\ \sup_{a\in \ell^\infty}
\left\|
\E_{n\in [N,N+\floor{L(N)}]}a(n)
\prod_{i=1}^{\ell} T^{\round{p_{i,N}(n)}}h_i
\right\|_{L^2}
=
o_{N\to\infty}(1).
\end{equation}
\end{lemma}
\begin{proof}
     Recall that for any $x,y\in \R$, $\round{x}-\round{y}-\round{x-y} \in \{-1,0,1\}$. Hence 
     $$
     \round{f_i(n)}-\round{p_{i,N}(n)}-\round{f_i(n)-p_{i,N}(n)}
     $$
     takes finitely many values for $N\in \N$, $n\in \N$ and $i\in \{1,\dots, \ell\}$. Additionally, by (\ref{eq:poly_approx_2}) ${f_i(n)-p_{i,N}(n)}= o_{N\to\oo}(1)$ uniformly for $n\in [N,N+\floor{L(N)}]$ for each $i\in \{1,\dots, \ell\}$, and so $\round{f_i(n)-p_{i,N}(n)}$ takes finitely many values for $N\in \N$, $n\in [N,N+\floor{L(N)}$, and $i\in \{1,\dots, \ell\}$. It follows that 
     $$
     \round{f_i(n)}-\round{p_{i,N}(n)}-\round{f_i(n)-p_{i,N}(n)}
     $$
     takes finitely many values for $N\in \N$, $n\in [N,N+\floor{L(N)}]$ and $i\in \{1,\dots, \ell\}$. Then the statement of the lemma follows from Lemma \ref{lem:partitions}.
\end{proof}

In order to perform PET induction, we define the equivalence relation $\psim$  given by $f\psim g$ if $\deg(f) = \deg (g) > \deg(f-g)$. Let $\mathcal{H}$ be a Hardy field and let $\mathscr{F} = \{f_1,\dots, f_{\ell}\}\subset \mathcal{H}$ be a finite collection of subpolynomial functions. Define $m_d(\mathscr{F})$ to be the number of equivalence classes in $\mathscr{F}$ which have degree $d$. As is typical in this context, we define $d_{\max}(\mathscr{F}) = \max\{d\in \N: m_{d}(\mathscr{F})\neq 0\}$ and we call $(m_1(\mathscr{F}),\dots, m_{d_{\max}(\mathscr{F})}({\mathscr{F}}))$ the \emph{characteristic vector} of $\mathscr{F}$.

For two such collections $\mathscr{F},\mathscr{F}'$, we say that 
$$
(m_1(\mathscr{F}'),\dots, m_{d_{\max}(\mathscr{F}')}({\mathscr{F}}'))<(m_1(\mathscr{F}),\dots, m_{d_{\max}(\mathscr{F})}({\mathscr{F}}))
$$
to mean that if $i$ is the largest natural number with $m_i(\mathscr{F}')\neq m_i(\mathscr{F})$ then $m_i(\mathscr{F}')< m_i(\mathscr{F})$. 
It will be useful to note that the set of finite collections of Hardy functions is well ordered under this ordering.

The following theorem combines with Lemma \ref{lem:donoso} to complete the proof of Theorem \ref{thm:generalized_BMR_4.2}.

\begin{theorem}\label{thm:polynomail_PET_intervals}
Let $\mathcal{H}$ be a Hardy field and $W \in \mathcal{H}$ with
$1 \prec \log W(x) \prec x$.
Let $\mathscr{F} = \{f_1,\ldots,f_{\ell}\}$ be a collection of functions such that $ f_1,\ldots,f_{\ell}\in \mathcal{H}$, $\mathscr{F}$ satisfies property (\ref{eq:W_compatible}), $|f_1(x)| \preceq \cdots \preceq |f_{\ell}(x)|$, and $
\lim_{x\to\infty} |f_{\ell}(x)|=\lim_{x\to\infty} |f_{\ell}(x)-f_i(x)|
=
\infty$ for every $i<\ell$. Then there exists $L\in \mathcal{H}$ with $1\prec L(x)\prec ((\log W)'(x))^{-1}$ and a value of $s\in\mathbb{N}$ such that for any invertible
measure preserving system $(X,\mathcal{B},\mu,T)$ and any $h_{\ell}\in L^\infty(X)$ with $\|h_{\ell}\|_s=0$ we have
\begin{equation}\label{eq:final_goal_eq}
\sup_{h_1,\ldots,h_{\ell-1}\in L^\infty}
\ \sup_{a\in \ell^\infty}
\left\|
\E_{n\in [N,N+\floor{L(N)}]}a(n)
\prod_{i=1}^{\ell} T^{\round{f_i(n)}}h_i
\right\|_{L^2}
=
o_{N\to\infty}(1).
\end{equation}
\end{theorem}
\begin{proof}
    Without loss of generality, assume that $\mathcal{H}$ is a maximal Hardy field and that $\deg(f_i)>0$ for each $i$, since if this is not true for some $i<\ell$ then we can factor $T^{\lim_{x\to\oo}[f_{i}(n)]}h_i+o_{N\to\oo}(1)$ out of the left-hand side of (\ref{eq:final_goal_eq}). Pick a (to be determined) function $L\in \mathcal{H}$ such that $L(x)\prec ((\log W)'(x))^{-1}$.

We proceed by induction on the characteristic vector $(m_1(\mathscr{F}),\dots, m_{d_{\max}(\mathscr{F})}({\mathscr{F}}))$.

For the base case of the induction, assume that $d_{\max}(\mathscr{F})=1$. Pick a function $\widetilde{W}\in \mathcal{H}$ such that $ ((\log \widetilde{W})'(x))^{-1}\prec  ((\log W)'(x))^{-1}$ and $\{f_1,\dots, f_{\ell}\}$ satisfies property (\ref{eq:W_compatible}) for $\widetilde{W}$. For example, we can pick $\widetilde{W}$ such that $((\log \widetilde{W})'(x))^{-1}$ grows like $((\log {W})'(x))^{-1}$ divided by a function which goes to $\oo$ slower than each of
$$
\frac{|f^{(\deg^*(f))}(x)-p(x)|^{1/\deg^*(f)}}{(\log W)'(x)}
$$
for nonzero $f\in \text{Span}_{\R}\{f_1,\dots, f_{\ell}\}$ and $p(x)\in \Q[x]$. We can apply Theorem \ref{thm:BMR24_sublinear} to $\mathscr{F}$ and $\widetilde{W}$ to see that 
\begin{equation}
\sup_{h_1,\ldots,h_{\ell-1}\in L^\infty}
\ \sup_{a\in \ell^\infty}
\left\|
\E_{n\leq N}^{\widetilde{W}}a(n)
\prod_{i=1}^{\ell} T^{[f_i(n)]}h_i
\right\|_{L^2}
=
o_{N\to\infty}(1).
\end{equation}
Pick $L\in \mathcal{H}$ such that $ ((\log \widetilde{W})'(x))^{-1}\prec  L(x)\prec ((\log W)'(x))^{-1}$ (such a selection is possible by \cite[Theorem A]{ADH24}) and let $s=\ell+1$. Then Theorem \ref{thm:mikey_thm} gives the statement of the base case.

Now for the induction step. Suppose that $d_{\max}(\mathscr{F})\geq 2$ and that the statement of Theorem \ref{thm:polynomail_PET_intervals} holds for all finite collections $\mathscr{F}'$ with 
\begin{equation}\label{eq:PET_induction}
(m_1(\mathscr{F}'),\dots, m_{d_{\max}(\mathscr{F}')}({\mathscr{F}}'))<(m_1(\mathscr{F}),\dots, m_{d_{\max}(\mathscr{F})}({\mathscr{F}})).
\end{equation}

We will find a collection $\mathscr{F}' = \{G_1,\dots, G_{\ell'}\}$ such that (\ref{eq:PET_induction}) holds for some (to be determined) $L,s$ and reduce showing (\ref{eq:final_goal_eq}) to showing that 
\begin{equation}\label{eq:final_goal_eq_again}
\sup_{h_1,\ldots,h_{\ell'-1}\in L^\infty}
\ \sup_{a\in \ell^\infty}
\left\|
\E_{n\in [N,N+\floor{L(N)}]}a(n)
\prod_{i=1}^{\ell'} T^{\round{G_i(n)}}h_i
\right\|_{L^2}
=
o_{N\to\infty}(1).
\end{equation}

 Define polynomials $p_{i,N}(n) = \sum_{j=0}^{\deg(f_i)}(n-N)^j\cdot f_i^{(j)}(N)/j!$ for $N\in \N$ and $n\in [N,N+\floor{L(N)}]$, and recall that (\ref{eq:poly_approx}) holds.

 By Lemma \ref{lem:polynomial_on_intervals}, it suffices to show that 
\begin{equation}\label{equation_label}
\sup_{h_1,\ldots,h_{\ell-1}\in L^\infty}
\ \sup_{a\in \ell^\infty}
\left\|
\E_{n\in [N,N+\floor{L(N)}]}a(n)
\prod_{i=1}^{\ell} T^{[p_{i,N}(n)]}h_i
\right\|_{L^2}
=
o_{N\to\infty}(1).
\end{equation}
Let $(a_N(n))_{N,n\in \N}$ and $(h_{i,N})_{N\in \N}$ for $i\in \{1,\dots, \ell-1\}$ be sequences which achieve the suprema on the left-hand side of (\ref{equation_label}) (also put $h_{\ell,N}=h_{\ell}$ for all $N\in\N$).

Let $u_{N}(n)  = a_N(n)\prod_{i=1}^{\ell}T^{\round{p_{i,N}(n)}}h_{i,N}$ for $n,N\in \N$.


From the fact that $L(x)\prec ((\log W)'(x))^{-1}$, we have
\begin{equation}\label{eq:L_growth}
      \lim_{N\to\oo} |f^{(\deg^*(f))}(N)-p(N)|\cdot  L(N)^{\deg^*(f)}=\oo
\end{equation}
for all $f\in \text{Span}\{ f_1,\dots, f_{\ell}\}$ with $\deg^*(f)>0$ and all $p(x)\in \Q[x]$.
Pick $q\in\mathcal{H}$ with $1\prec q(x)\prec L(x)$ such that 
\begin{equation}\label{eq:q_growth}
      \lim_{N\to\oo} |f^{(\deg^*(f))}(N)-p(N)|\cdot  L(N)^{\deg^*(f)-1}\cdot q(N)=\oo
\end{equation}
for all $f\in \text{Span}\{ f_1,\dots, f_{\ell}\}$ with $\deg^*(f)>0$ and all $p(x)\in \Q[x]$ (compare with (\ref{eq:L_growth})). Such a selection of $q$ is possible by \cite[Theorem A]{ADH24}. In light of Lemma \ref{lem:almost_final_vdc}, in order to show (\ref{equation_label}), it suffices to show  
\begin{equation}\label{eq:vdc_goal}
 \lim_{N\to\oo} \E_{n\in [N,N+\floor{L(N)}]} \left(\int u_{N}(n+m\floor{q(N)})\cdot \overline{u_{N}(n)}~d\mu\right)=0
\end{equation}
for each $m\in \N$. So fix $m\in \N$ and observe that
\begin{align*}
   &\int u_{N}(n+m\floor{q(N)})\cdot \overline{u_{N}(n)} ~d\mu\\
   =&\int a_N(n+m\floor{q(N)})\prod_{i=1}^{\ell}T^{\round{p_{i,N}(n+m\floor{q(N)})}}h_{i,N}\cdot \overline{a_N(n)}\prod_{i=1}^{\ell}T^{\round{p_{i,N}(n)}}\overline{h_{i,N}} ~d\mu\\
    =&a_N(n+m\floor{q(N)})\overline{a_N(n)}\int \prod_{i=1}^{\ell}T^{\round{p_{i,N}(n+m\floor{q(N)})}}h_{i,N}\cdot T^{\round{p_{i,N}(n)}}\overline{h_{i,N}} ~d\mu\\
    =&a_N(n+m\floor{q(N)})\overline{a_N(n)}\int \prod_{i=1}^{\ell}T^{\round{p_{i,N}(n+m\floor{q(N)})}-\round{p_{1,N}(n)}}h_{i,N}\cdot T^{\round{p_{i,N}(n)}-\round{p_{1,N}(n)}}\overline{h_{i,N}}~d\mu
\end{align*}
by the fact that $T$ is measure preserving. Using Lemma \ref{lem:partitions}, we can replace $\round{p_{i,N}(n+m\floor{q(N)})}-\round{p_{1,N}(n)}$ and $\round{p_{i,N}(n)}-\round{p_{1,N}(n)}$, with $\round{p_{i,N}(n)-p_{1,N}(n)}$ and $\round{p_{i,N}(n)-p_{1,N}(n)}$, respectively. 

Now put $P_{i,N}(n)  = {p_{i,N}(n)-p_{1,N}(n)}$, $P_{\ell+i,N}(n) = p_{i,N}(n+m\floor{q(N)})-p_{1,N}(n)$ for $i=1,\dots, \ell$ and let $\widetilde{a}_N(n)=a_N(n+m\floor{q(N)})\overline{a_N(n)}$. So far, we have reduced the problem to showing that 
\begin{equation}\label{equation_label_again}
\left\|
\E_{n\in [N,N+\floor{L(N)}]}\widetilde{a}_N(n)
\prod_{i=1}^{2\ell} T^{[P_{i,N}(n)]}\widetilde{h}_{i,N}
\right\|_{L^2}
=
o_{N\to\infty}(1).
\end{equation}
for arbitrary sequences $(\widetilde{a}_N(n))_{N,n\in \N}$, $(h_{i,N})_{N\in \N}$ for $i\in \{1,\dots, 2\ell\}$, which each have norm at most $1$, with $(\widetilde{h}_{2\ell,N})_{N\in \N}$ constant and $\|\widetilde{h}_{2\ell,N}\|_s=0$ for all $N\in \N$.

Define $g_i(N) = P_{i,N}(N)$ for $i\in \{1,\dots, 2\ell\}$ and note that each $g_i$ belongs to $\mathcal{H}$ after subtracting a bounded function (which is inconsequential by Lemma \ref{lem:partitions})).

Next, we claim that $\{g_1,\dots, g_{2\ell}\}$ satisfies property (\ref{eq:W_compatible}). To this end, note that $\{g_1,\dots, g_{\ell}\}$ satisfies property (\ref{eq:W_compatible}) because $\text{Span}_{\R}\{g_1,\dots, g_{\ell}\} \subset \text{Span}_{\R}\{f_1,\dots, f_{\ell}\}$. We also have that $\{g_{\ell+1},\dots, g_{2\ell}\}$ satisfies property (\ref{eq:W_compatible}) because for any $g\in \text{Span}_{\R}\{g_{\ell+1},\dots, g_{2\ell}\}$ we have $g(N) = \widetilde{g}(N+m\floor{q(N)})-\widetilde{g}(N)\approx m q(N)\widetilde{g}'(N)$, and hence $g^{(\deg(g))}(x) = mq(x)(\tilde{g}')^{(\deg(g))}(x)\cdot (1+o_{x\to\oo}(1))$, for some $\widetilde{g}\in \text{Span}_{\R}\{f_1,\dots, f_{\ell}\}$. By slightly changing the growth rate of $q$ if necessary, we can assume that $g$ does not grow at the same rate as any $p(x)\in \Q[x]$, so we have 
\begin{align*}
\lim_{x\to\oo}|g^{(\deg^*(g))}(x)|\cdot L^{\deg^*(g)}(x)=&\lim_{x\to\oo}|mq(x)(\widetilde{g}')^{(\deg^*({g}))}(x)|\cdot L^{\deg^*(g)}\\
=&\lim_{x\to\oo}|m(\widetilde{g})^{(\deg^*(\widetilde{g}))}(x)|\cdot L^{\deg^*(\widetilde{g})-1}\cdot q(x) = \oo
\end{align*}
by (\ref{eq:q_growth}). Again, by slightly changing the growth rate of $q$ if necessary, we can guarantee that each unbounded element of $\text{Span}_{\R}\{g_{\ell+1},\dots, g_{2\ell}\} $ has a different growth rate than each element of $\text{Span}_{\R}\{g_1,\dots, g_{\ell}\}$. It follows that $\{g_1,\dots, g_{2\ell}\}$ satisfies property (\ref{eq:W_compatible}).

As a consequence of this
we have that 
\begin{equation} \label{eq:poly_approx_again}
\lim_{N\to\oo}\max_{n\in \{N,\dots, N+\floor{L(N)}\}}|g_i(n)-P_{i,N}(n)|=0.
\end{equation}
for $i\in \{1,\dots, 2\ell\}$. By Lemma \ref{lem:polynomial_on_intervals}, it suffices to show that 
\begin{equation}\label{equation_label_again_again}
\sup_{\widetilde{h}_1,\ldots,\widetilde{h}_{2\ell-1}\in L^\infty}
\ \sup_{\widetilde{a}\in \ell^\infty}
\left\|
\E_{n\in [N,N+\floor{L(N)}]}\widetilde{a}(n)
\prod_{i=1}^{2\ell} T^{[g_{i}(n)]}\widetilde{h}_i
\right\|_{L^2}
=
o_{N\to\infty}(1),
\end{equation}
 so that (\ref{equation_label_again}) holds.

 Note that  $\lim_{x\to\infty}|g_{2\ell}(x)|=\lim_{x\to\infty}|g_{2\ell}(x)-g_j(x)| = \oo$ for all $j<2\ell$ because
$$
\lim_{x\to\infty}|g_{2\ell}(x)-g_j(x)|  = \begin{cases} 
\lim_{x\to\infty}|f_{\ell}(x+m\floor{q(x)})-f_j(x)| &\text{ if }  j\in \{1,\dots, \ell-1\}\\
\lim_{x\to\infty}|f_{\ell}(x+m\floor{q(x)})-f_{j-\ell}(x+m\floor{q(x)})| &\text{ if }  j\in \{\ell+1,\dots, 2\ell-1\}
\end{cases}
$$
which tends to $\oo$ by assumption in all cases. If $j=\ell$ then $\lim_{x\to\infty}|g_{2\ell}(x)-g_{\ell}(x)|  = \lim_{x\to\infty}|f_{\ell}(x+m\floor{q(x)})-f_{\ell}(x)|= \lim_{x\to\infty}|f_{\ell}'(x)\cdot m\floor{q(x)}|=\oo$ because $\deg(f_{\ell}) = \max\{d\in \N: m_d(\mathscr{F})\neq 0\}\geq 2$.

Let $|G_{1}|\preceq \dots \preceq  |G_{\ell'}|$ be a relabeling of the unbounded elements of $\{g_{1},\dots, g_{2\ell}\}$ and let $\mathscr{F}'  = \{G_1,\dots, G_{\ell'}\}$.

The last thing to note is that 
\begin{equation}\label{eq:char_vectors}
(m_1(\mathscr{F}'),\dots, m_{d_{\max}(\mathscr{F}')}({\mathscr{F}}'))<(m_1(\mathscr{F}),\dots, m_{d_{\max}(\mathscr{F})}({\mathscr{F}})).
\end{equation}

This is true because if $f_1\not\psim f_i$, then $g_{\ell+i}\psim g_i $. If $f_1\not\psim f_i$ and $f_1\not\psim f_j$ then $f_i\psim f_j$ holds if and only if $g_{i}\psim g_j$. So it follows that $m_d(\mathscr{F}) = m_{d}(\mathscr{F}')$ for $d>\deg(f_1)\geq 1$. But $m_{\deg(f_1)}(\mathscr{F})>m_{\deg(f_1)}(\mathscr{F}')$ because $\deg(f_i-f_1)<\deg(f_i)$ whenever $f_i\psim f_1$. This shows that (\ref{eq:char_vectors}) holds. By the induction hypothesis, we have 
\begin{equation}
\sup_{h_1,\ldots,h_{\ell-1}\in L^\infty}
\ \sup_{a\in \ell^\infty}
\left\|
\E_{n\in [N,N+\floor{L(N)}]}a(n)
\prod_{i=1}^{\ell'} T^{\round{G_i(n)}}h_i
\right\|_{L^2}
=
o_{N\to\infty}(1)
\end{equation}
which implies that (\ref{equation_label_again_again}) holds. This concludes the proof.
\end{proof}

Next, we will reduce Theorem \ref{thm:main} to the following statement about pointwise convergence in nilsystems.
\begin{theorem}\label{thm:true_for_nilsystem}
 Let $\mathcal{H}$ be a Hardy field. Let $f_1,\dots, f_{\ell}\in \mathcal{H}$ be  subpolynomial and let $W\in \mathcal{H}$ satisfy $1\prec \log W(x)\prec x$. Suppose that $\{f_1,\dots, f_{\ell}\}$ satisfies property (\ref{eq:W_compatible}). Let $(X,\mathscr{B}_X,\mu_X,T)$ be a nilsystem where $X= G/\Gamma$ for $G$ simply connected. 
     \begin{enumerate}[label = (\roman*)]
         \item For each $g_1,\dots, g_{\ell}\in C(X)$ and each $x\in X$,
         \begin{equation}\label{eq:new_eq_1}
             \lim_{N\to\oo}\E_{n\leq N}^W(T^{\round{f_1(n)}}g_1(x)\cdots T^{\round{f_{\ell}(n)}}g_{\ell}(x))
         \end{equation}
         exists.
        \item Suppose that $\Poly\{f_1,\dots, f_{\ell}\}\cap \Z[x]=\{0\}$. Then for each $g_1,\dots, g_{\ell}\in C(X)$ and each $x\in X$,
        \begin{equation}\label{eq:new_eq_2}
             \lim_{N\to\oo}\E_{n\leq N}^W(T^{\round{f_1(n)}}g_1(x)\cdots T^{\round{f_{\ell}(n)}}g_{\ell}(x)) = \prod_{i=1}^{\ell}\left(\lim_{N\to\oo}\E_{n\leq N}g_i(T^nx)\right).
         \end{equation}
         \item 
         Suppose that $\Poly\{f_1,\dots, f_{\ell}\}$ is jointly intersective.
         Then for any $A\in \mathscr{B}_X$ with $\mu_X(A)>0$,
        \begin{equation}\label{eq:new_eq_3}
             \lim_{N\to\oo}\E_{n\leq N}^W\mu_X(A\cap T^{-\round{f_1(n)}}A\cdots T^{-\round{f_{\ell}(n)}}A) >0.
         \end{equation}
     \end{enumerate}
\end{theorem}

The proof of Theorem \ref{thm:main} using Theorem \ref{thm:true_for_nilsystem} relies on the following lemma, which summarizes the facts shown in the proofs of \cite[Theorem B]{BMR24} and \cite[Theorem 5.3]{BMR24} in \cite[Section 5]{BMR24}. More specifically, the proof of \cite[Theorem 5.3]{BMR24} establishes the implications $(i)\implies (i)'$, $(ii)\implies (ii)'$, $(iii)\implies (iii)'$ and the proof of \cite[Theorem B]{BMR24} establishes the implications $(i)'\implies (i)''$, $(ii)'\implies (ii)''$, $(iii)'\implies (iii)''$.

\begin{lemma}[{\cite[Section 5]{BMR24}}]\label{lem:bmr_sec_5}
   Let $\mathcal{H}$ be a Hardy field, let $\{f_1,\dots, f_{\ell}\}\subset\mathcal{H}$ be an arbitrary collection of functions, and let $W\in \mathcal{H}$ such that $1\prec \log W(x)\prec x$. 
   Consider the following statements:
   \begin{enumerate}
    \item[$(i)$] For any nilsystem $(X,\mathscr{B}_X,\mu_X,T)$with $X =G/\Gamma$ for $G$ simply connected, for any $h_1,\dots, h_{\ell}\in C(X)$ and any $x\in X$,
\begin{equation}\label{eq:BMR_sec_5_lemma_1}
    \text{ the limit }\lim_{N\to\oo}\E_{n\leq N}^W\left(\prod_{i=1}^{\ell}T^{\round{f_i(n)}}h_i(x)\right)\text{ exists.}
\end{equation}
       \item[$(i)'$] (\ref{eq:BMR_sec_5_lemma_1}) holds in $L^2(X)$ for any ergodic nilsystem $(X,\mathscr{B}_X,\mu_X,T)$ and any $h_1,\dots, h_{\ell}\in L^{\oo}(X)$.
       \item[$(i)''$] (\ref{eq:BMR_sec_5_lemma_1}) holds in $L^2(X)$ for any invertible measure preserving system $(X,\mathscr{B},\mu,T)$ and any $h_1,\dots, h_{\ell}\in L^{\oo}(X)$.
       \item[$(ii)$] For each nilsystem $(X,\mathscr{B}_X,\mu_X,T)$ with $X =G/\Gamma$ for $G$ simply connected, for each $h_1,\dots, h_{\ell}\in C(X)$, and for each $x\in X$,
\begin{equation}\label{eq:BMR_sec_5_lemma_2}
   \lim_{N\to\oo}\E_{n\leq N}^W\left(\prod_{i=1}^{\ell}T^{\round{f_i(n)}}h_i(x)\right) = \prod_{i=1}^{\ell}h_i^*(x)
\end{equation}
 where $h_i^*$ is the projection in $L^2(X)$ of $h_i$ onto the subspace of $T$ invariant functions (in this case we have $h_i^*(x) = \lim_{N\to\oo}\E_{n\leq N}h_i(T^nx)$ since nilsystems are uniquely ergodic).
       \item[$(ii)'$] (\ref{eq:BMR_sec_5_lemma_2}) holds in $L^2(X)$ for any ergodic nilsystem $(X,\mathscr{B}_X,\mu_X,T)$ and any $h_1,\dots, h_{\ell}\in L^{\oo}(X)$.
       \item[$(ii)''$] (\ref{eq:BMR_sec_5_lemma_2}) holds in $L^2(X)$ for any invertible measure preserving system $(X,\mathscr{B},\mu,T)$ and any $h_1,\dots, h_{\ell}\in L^{\oo}(X)$.
       \item[$(iii)$] For any nilsystem $(X,\mathscr{B}_X,\mu_X,T)$ with $X =G/\Gamma$ for $G$ simply connected and any nonzero $h_1,\dots, h_{\ell}\in C(X)$ taking values in $[0,\oo)$,        \begin{equation}\label{eq:BMR_sec_5_lemma_3}
    \lim_{N\to\oo}\E_{n\leq N}^W\left(\prod_{i=1}^{\ell}T^{\round{f_i(n)}}h_i(x)\right)>0.
       \end{equation}
       \item[$(iii)'$] (\ref{eq:BMR_sec_5_lemma_3}) holds for any ergodic nilsystem $(X,\mathscr{B}_X,\mu_X,T)$ and any nonzero $h_1,\dots, h_{\ell}\in L^{\oo}(X)$ taking values in $[0,\oo)$.
       \item[$(iii)''$] (\ref{eq:BMR_sec_5_lemma_3}) holds for any invertible measure preserving system $(X,\mathscr{B},\mu,T)$ and any nonzero $h_1,\dots, h_{\ell}\in L^{\oo}(X)$ taking values in $[0,\oo)$.
   \end{enumerate}
Then the implications $(i)\implies (i)'$, $(ii)\implies (ii)'$, $(iii)\implies (iii)'$ hold. With the additional assumption that there exists an $s\in \N$ such that such that (\ref{eq:4.2_goal}) holds for any invertible measure preserving system $(X,\mathcal{B},\mu,T)$ and any $h_{\ell}\in L^\infty(X)$ with $\|h_{\ell}\|_s=0$, the implications $(i)'\implies (i)''$, $(ii)'\implies (ii)''$, $(iii)'\implies (iii)''$ also hold.

\end{lemma}

Theorem \ref{thm:main} follows readily from Theorem \ref{thm:generalized_BMR_4.2}, Theorem \ref{thm:true_for_nilsystem}, and Lemma \ref{lem:bmr_sec_5}. Indeed, Theorem \ref{thm:true_for_nilsystem} shows that (\ref{eq:BMR_sec_5_lemma_1}), (\ref{eq:BMR_sec_5_lemma_2}), and (\ref{eq:BMR_sec_5_lemma_3}) hold when $h_1,\dots, h_{\ell}$ are continuous functions defined on $X= G/\Gamma$ for $G$ a simply connected Lie group. Then Theorem \ref{thm:generalized_BMR_4.2} and Lemma \ref{lem:bmr_sec_5} transform this result into the conclusion of Theorem \ref{thm:main}.

In order to prove Theorem \ref{thm:true_for_nilsystem}, we require the following lemma.

\begin{lemma}\label{lem:variant_of_modulo_averages}
    Let $\mathcal{H}$ be a maximal 
    Hardy field and let $W\in \mathcal{H}$ with $1\prec \log W(x)\prec x$. Let $(a(n))_{n\in \N}$ be a bounded sequence and let $R\in \N$. Then
    \begin{itemize}
        \item $\lim_{N\to\oo}\E_{n\leq N}^Wa(n)$ exists if and only if $\lim_{N\to\oo}\frac{1}{R}\sum_{r=0}^{R-1}\E_{n\leq N}^Wa(nR+r)$ exists. Moreover, these two limits are equal when they both exist.
        \item Suppose that $a(n)\geq 0$ for all $n\in \N$, $\lim_{N\to\oo}\E_{n\leq N}^Wa(n)$ exists, and for some $r_0\in \Z$ with $0\leq r_0\leq R-1$, the limit $\lim_{N\to\oo}\E_{n\leq N}^Wa(nR+r_0)$ exists and is positive. Then $\lim_{N\to\oo}\E_{n\leq N}^Wa(n)>0$.
    \end{itemize}
\end{lemma}
\begin{proof}
    First note that 
\begin{align*}
&\frac{1}{R}\sum_{r=0}^{R-1}\E_{n\leq N}^Wa(nR+r) = \frac{1}{R}\sum_{r=0}^{R-1}\frac{1}{W(N)}\sum_{n=1}^N\Delta W(n)a(nR+r)\\ = &\frac{1}{W(NR/R)}\sum_{n=1}^{NR}\frac{\Delta W(\lfloor n/R\rfloor)}{R}a(n)+o_{N\to\oo}(1) = \E_{n\leq NR}^Va(n)+o_{N\to\oo}(1)
\end{align*}
for $V(N) = W(N/R)$. Then
\begin{equation}\label{eq:silverman_toeplitz_application_1}
\E_{n\leq N}^Wa(n) = \sum_{n=1}^Nc_{N,n}\E_{k\leq n}^Va(k)+o_{N\to\oo}(1)
\end{equation}
where $c_{N,N} =\frac{V(N)\Delta W(N)}{W(N)\Delta V(N)}$, $c_{N,n} = \left(\frac{\Delta W(n)}{\Delta V(n)}-\frac{\Delta W(n+1)}{\Delta V(n+1)}\right)\frac{\Delta V(n)}{W(N)}$ for $n<N$, and $c_{N,n} = 0$ for $n>N$. Observe that 
\begin{align}
    \lim_{N\to\oo}c_{N,N} =& \lim_{N\to\oo}\frac{V(N)\Delta W(N)}{W(N)\Delta V(N)} = \lim_{N\to\oo}\frac{V(N)\cdot  W'(N)}{W(N) \cdot V'(N)}\\
    =&\lim_{N\to\oo}\frac{(\log W)'(N)}{(\log V)'(N)} =  \lim_{N\to\oo}\frac{\log (W(N))}{\log (V(N))} \label{eq:limit_in_avaerging_lemma}
\end{align}
by L'H\^opital's rule. We know that the limit in (\ref{eq:limit_in_avaerging_lemma}) exists since $\log W$ and $\log V$ belong to the same Hardy field, and the limit must be larger than $1$ because $W$ is eventually increasing and so $\log W$ is eventually larger than $\log V$. Additionally, this limit is finite since $\log (W(x))\prec x$ and so $\log (W(x/R))\geq \log (W(x))/R$.

It is clear from (\ref{eq:silverman_toeplitz_application_1}) that $\sum_{n=1}^Nc_{N,n} = 1+o_{N\to\oo}(1)$ by taking the sequence $a(k) = 1$ for all $k$. Additionally, for $n<N$, $c_{N,n}$ is built out of functions which belong to the same Hardy field and so it must be eventually positive or eventually negative. Given that $c_{N,N}$ is larger than $1$, we conclude that $c_{N,n}< 0$ for all $N$ and all sufficiently large $n<N$. Hence 
\begin{equation}
\limsup_{N\to\oo}\sum_{n=1}^N|c_{N,n}| \leq \limsup_{N\to\oo}\left(c_{N,N}-\sum_{n=1}^{N-1}c_{N,n}\right) \leq \lim_{N\to\oo}c_{N,N}+\left(\lim_{N\to\oo}c_{N,N}-1\right)<\oo.
\end{equation}
We have shown the conditions of Lemma \ref{lem:silverman_toeplitz} hold and so $(c_{N,n})_{N,n\in \N}$ defines a regular method of summation. This shows that if the limit $\lim_{N\to\oo}\E_{n\leq N}^Va(n)$ exists then $\lim_{N\to\oo}\E_{n\leq N}^Wa(n)$ exists and they are equal. The other direction follows from a similar argument, since
\[
\E_{n\leq N}^Va(n) = \sum_{n=1}^Nd_{N,n}\E_{k\leq n}^Wa(k)+o_{N\to\oo}(1)
\]
where $d_{N,N} =\frac{W(N)\Delta V(N)}{V(N)\Delta W(N)}$, $d_{N,n} = \left(\frac{\Delta V(n)}{\Delta W(n)}-\frac{\Delta V(n+1)}{\Delta W(n+1)}\right)\frac{\Delta W(n)}{V(N)}$, and $d_{N,n}=0$ for $n>N$. By similar reasoning as above, $\lim_{N\to\oo}d_{N,N}$ exists and is less than $1$, and $\sum_{n=1}^Nd_{N,n}= 1+o_{N\to\oo}(1)$. It follows that $d_{N,n}$ is positive for all $n\leq N$ and all large enough $N$ and so $\limsup_{N\to\oo}\sum_{n=1}^N|d_{n,N}|=1<\oo$. Again by Lemma \ref{lem:silverman_toeplitz}, $(d_{N,n})_{N,n\in \N}$ defines a regular method of summation and this proves the first item in the statement of the lemma.

The second item follows from the first. Indeed,
\[
\lim_{N\to\oo}\E_{n\leq N}^Wa(n)=\lim_{N\to\oo}\frac{1}{R}\sum_{r=0}^{R-1}\E_{n\leq N}^Wa(nR+r)\geq \frac{1}{R}\lim_{N\to\oo}\E_{n\leq N}^Wa(nR+r_0)>0.
\]
\end{proof}
Lemma \ref{lem:variant_of_modulo_averages} allows us to simplify the proof of Theorem \ref{thm:true_for_nilsystem} by adding an additional assumption.
\begin{theorem}\label{thm:true_for_nilsystems_and_also}
    Theorem \ref{thm:true_for_nilsystem} is true when each function $f_i$ satisfies the additional assumption that
\begin{equation}\label{eq:additional_assumption}
    \text{if }|f_i(t)-p(t)|\to 0\text{ for some }p(t)\in \R[t] \text{ with }p(t)-p(0)\in \Q[t], \text{ then }p(t)-p(0)\in \Z[t].
\end{equation}
\end{theorem}
\begin{proof}[Proof of Theorem \ref{thm:true_for_nilsystem} given Theorem \ref{thm:true_for_nilsystems_and_also}]
    Let $I = \{i\in \{1,\dots, \ell\}: f_i$ does not satisfy (\ref{eq:additional_assumption})$\}$.
    For $i\in I$, pick $p_i(t)\in \R[t]$ such that $p_i(t)-p_i(0)\in \Q[t]$ and $f_i(t) = p_i(t)+o_{t\to\oo}(1)$. Pick $R\in \N$ such that $R\cdot (p_i(t)-p_i(0))\in \Z[t]$. 

    Let $r\in \N$ be arbitrary. For $i=1,\dots, \ell$, define $f_{i,r}(x) = f_i(Rx+r)$ and note that $\{f_{1,r},\dots, f_{\ell,r}\}$ satisfies (\ref{eq:additional_assumption}) by our choice of $R$. Indeed, if $i\not\in I$ then $f_{i,r}$ satisfies (\ref{eq:additional_assumption}) because $f_i$ satisfies (\ref{eq:additional_assumption}), and if $i\in I$ then $f_{i,r}(t) = p_{i}(Rt+r)$ and $p_{i}(Rt+r)-p_i(r)\in \Z[t]$. Additionally, $\{f_{1,r},\dots, f_{\ell,r}\}$ satisfies property (\ref{eq:W_compatible}) since $\{f_{1},\dots, f_{\ell}\}$ satisfies property (\ref{eq:W_compatible}).

Let $(X,\mathscr{B},\mu,T)$ be a nilsystem with $X=G/\Gamma$ for $G$ simply connected. By Theorem \ref{thm:true_for_nilsystems_and_also}, the limit 
    \begin{equation}\label{eq:reduction_eq_1}
\lim_{N\to\oo}\E_{n\leq N}^W(T^{\round{f_{1,r}(n)}}h_1\cdots T^{\round{f_{\ell,r}(n)}}h_{\ell})
    \end{equation}
    exists in $L^2(X)$ for all $h_1,\dots, h_{\ell}\in L^{\oo}(X)$ and for all $r\in \N$. Hence the limit 
        \begin{equation}\label{eq:reduction_eq_2}
\lim_{N\to\oo}\frac{1}{R}\sum_{r=0}^{R-1}\E_{n\leq N}^W(T^{\round{f_{1,r}(n)}}h_1\cdots T^{\round{f_{\ell,r}(n)}}h_{\ell})
    \end{equation}
    exists and equals (\ref{eq:reduction_eq_1}) for all $h_1,\dots, h_{\ell}\in L^{\oo}(X)$. By the first item in Lemma \ref{lem:variant_of_modulo_averages}, the limit 
\begin{equation}\label{eq:reduction_eq_3}
\lim_{N\to\oo}\E_{n\leq N}^W(T^{\round{f_{1}(n)}}h_1\cdots T^{\round{f_{\ell}(n)}}h_{\ell})
\end{equation}
exists and equals (\ref{eq:reduction_eq_2}) for all $h_1,\dots, h_{\ell}\in L^{\oo}(X)$. This shows that items (i) and (ii) of Theorem \ref{thm:true_for_nilsystem} hold in this case.

Lastly, suppose that $\Poly\{f_1,\dots, f_{\ell}\}$ is jointly intersective. Let $q_1(t),\dots, q_{m}(t)\in \Z[t]$ be such that $\Poly\{f_1,\dots, f_{\ell}\}\subset \text{Span}_{\mathbb{R}}\{q_1,\dots, q_{m}\}$ and for each $r\in \N$ there exists $n\in \N$ with $q_i(n)$ divisible by $r$ for all $i$. Now pick $r_0\in \{0,\dots, R-1\}$ such that the functions $Q_{i}(t) = q_i(Rt+r_0)$ are jointly intersective for $i=1,\dots, m$. Such an $r_0$ exists because for any $M\in \N$ there is an $n\in \N$ such that $q_i(n)$ is divisible by $MR$ for all $i$, from which it follows that $q_i(n)$ is divisible by both $M$ and $R$ for all $i$. We also know that $q_i(n_1)\equiv q_i(n_2)\mod R$ if $n_1\equiv n_2\mod R$ and so it follows that there exists $r_0$ such that $q_i(Rk+r_0)\equiv 0\mod R$ for all $i$ and all $k\in \N$, so that for each $M\in \N$ there is a $k\in \N$ with $q_i(Rk+r_0)\equiv 0\mod M$.

This shows that $\Poly\{f_{1,r_0},\dots, f_{\ell,r_0}\}$ is jointly intersective and so by Theorem \ref{thm:main} we have
\begin{equation*}             \lim_{N\to\oo}\E_{n\leq N}^W\mu(A\cap T^{-\round{f_{1,r_0}(n)}}A\cdots T^{-\round{f_{\ell,r_0}(n)}}A) >0.
\end{equation*}
for any $A\in \mathscr{B}$ with $\mu(A)>0$. By the second item in Lemma \ref{lem:variant_of_modulo_averages}, it follows that
\begin{equation*}             \lim_{N\to\oo}\E_{n\leq N}^W\mu(A\cap T^{-\round{f_{1}(n)}}A\cdots T^{-\round{f_{\ell}(n)}}A) >0.
\end{equation*}
and so we are done.
\end{proof}

Next, we further reduce Theorem \ref{thm:true_for_nilsystems_and_also} a statement about uniform distribution.

\begin{theorem}[cf. {\cite[Theorem 5.7]{BMR24}}]\label{thm:BMR_uniform_distribution_result}
Let $G$ be a simply connected nilpotent Lie group, let
$\Gamma \subset G$ be a uniform and discrete subgroup, let
$X = G/\Gamma$, and let $a \in G$. Let $\mathcal{H}$ be a Hardy field and let $f_1,\ldots,f_{\ell} \in \mathcal{H}$ be subpolynomial functions which satisfy (\ref{eq:additional_assumption}). Let $W \in \mathcal{H}$ with $1 \prec \log W(x) \prec x$. Suppose that property (\ref{eq:W_compatible}) holds. Then there exists a Borel probability measure $\nu_a$ on
$X^{\ell}$ such that the sequence
\begin{equation}\label{eq:BMR_5.2}
n \mapsto \bigl(a^{ \round{f_1(n)}},
a^{ \round{f_2(n)}},\ldots,
a^{ \round{f_{\ell}(n)}}\bigr)\Gamma^{\ell}
\end{equation}
is uniformly distributed with respect to $W$-averages in $(X^{\ell},\nu)$.

Moreover, if $\Poly\{f_1,\dots, f_{\ell}\}$ is jointly intersective then the point
$1_{X^{\ell}} = 1_G\Gamma^{\ell}$ belongs to the support of $\nu_a$, and if $\Poly\{f_1,\dots, f_{\ell}\}\cap \Z[x]=\{0\}$ then $\nu_a$ is the Haar measure on the
subnilmanifold $Y_a^{\ell} \subset X^{\ell}$, where $Y_a = \overline{\{a^n\Gamma : n \in \mathbb{Z}\}}$.
\end{theorem}
\begin{proof}[Proof of Theorem \ref{thm:true_for_nilsystems_and_also} Given Theorem \ref{thm:BMR_uniform_distribution_result}]

Let $G$ be a simply connected nilpotent Lie group, let
$\Gamma \subset G$ be a uniform and discrete subgroup, let
$X = G/\Gamma$, and let $T:X\rightarrow X$ be given by $T(x\Gamma) =  ax\Gamma$ for some $a \in G$.

Let $x\in X$ and let $\gamma_x\in G$ denote a coset representative for $x$, so that $x=\gamma_x\Gamma$. Let $a_x \gamma_x^{-1}a\gamma_x$ so that $a^{\round{f_i(n)}}x = \gamma_x(\gamma_x^{-1}a\gamma_x)^{\round{f_i(n)}}\Gamma = \gamma_x a_x^{\round{f_i(n)}}\Gamma$ for all $n\in \N$ and all $i$. Let $g_1,\dots, g_{\ell}\in C(X)$ and define $G_{x}(x_1,\dots, x_k) = g_1(\gamma_xx_1)\cdots g_{\ell}(\gamma_xx_{\ell})$. Then by (\ref{eq:BMR_5.2}) we have that
 \begin{align*}
             &\lim_{N\to\oo}\E_{n\leq N}^W(T^{\round{f_1(n)}}g_1(x)\cdots T^{\round{f_{\ell}(n)}}g_{\ell}(x)) = \lim_{N\to\oo}\E_{n\leq N}^Wg_1(a^{\round{f_1(n)}}x)\cdots g_{\ell}(a^{\round{f_{\ell}(n)}}x)\\
             =&\lim_{N\to\oo}\E_{n\leq N}^Wg_1(\gamma_x a_x^{\round{f_1(n)}}\Gamma)\cdots g_{\ell}(\gamma_x a_x^{\round{f_{\ell}(n)}}\Gamma)\\
             =& \lim_{N\to\oo}\E_{n\leq N}^WG_x(a_x^{\round{f_1(n)}},\dots, a_x^{\round{f_{\ell}(n)}}) = \int_{X^{\ell}} G_x ~d\nu_{a_x}.
\end{align*}
This shows that $\lim_{N\to\oo}\E_{n\leq N}^W(T^{\round{f_1(n)}}g_1(x)\cdots T^{\round{f_{\ell}(n)}}g_{\ell}(x))$ exists. Next, suppose that $\Poly\{f_1,\dots, f_{\ell}\}\cap \Z[x]=\{0\}$. Then $\nu_{a_x}$ is the Haar measure on $Y_x = \overline{\{a_x^n\Gamma:n\in \Z\}}$ and so  
\begin{align*}
    \int_{X^{\ell}} G_x ~d\nu_{a_x} = \int_{Y_{a_x}^{\ell}} G_x ~d\nu_{a_x}=&\int_{Y_{a_x}^{\ell}} \prod_{i=1}^{\ell}g_i(\gamma_x x_i)~d\nu_x(x_1,\dots, x_{\ell}) \\=&\prod_{i=1}^{\ell} \int_{Y_{a_x}} g_i(\gamma_x x_i)~d\mu_{Y_{a_x}}(x_i)
\end{align*}
    but we know that 
    \[
    \int_{Y_{a_x}} g_i(\gamma_x x_i)~d\mu_{Y_{a_x}}(x_i)=\lim_{N\to\oo}\E_{n\leq N}g_i(\gamma_xa_x^n\Gamma) =  \lim_{N\to\oo}\E_{n\leq N}g_i(a^nx)
    \]
    for each $i$, and so we have that
    \begin{align*}
    \int_{X^{\ell}} G_x ~d\nu_{a_x} = \prod_{i=1}^{\ell}\lim_{N\to\oo}\E_{n\leq N}g_i(a^nx)
\end{align*}
as desired.

Lastly, suppose that $\Poly\{f_1,\dots, f_{\ell}\}$ is jointly intersective. Let $g\in C(x)$ satisfy $g(x)\geq 0$ for all $x\in X$ and $\int_X g~d\mu_X>0$. We will show that 
\begin{equation}\label{eq:u.d._reduction_goal}
    \lim_{N\to\oo}\E_{n\leq N}^W\int_X g(x)\cdot \prod_{i=1}^{\ell}g(a^{\round{f_i(n)}}x)~d\mu_X(x) >0
\end{equation}
so that (\ref{eq:new_eq_3}) follows by approximating the step function $1_A$ by continuous functions. By the dominated convergence theorem, we have that
\[
\lim_{N\to\oo}\E_{n\leq N}^W\int_X g(x) \prod_{i=1}^{\ell}g(a^{\round{f_i(n)}}x)~d\mu_X(x)=\int_X g(x) \lim_{N\to\oo}\E_{n\leq N}^W\left(\prod_{i=1}^{\ell}g(a^{\round{f_i(n)}}x)\right)d\mu_X(x) 
\]
and we recall from the above arguments that $\lim_{N\to\oo}\E_{n\leq N}^W\prod_{i=1}^{\ell}g(a^{\round{f_i(n)}}x)$ converges to $\int_{X^{\ell}}\prod_{i=1}^{\ell}g(\gamma_x x_i)~d\nu_{x}(x_1,\dots, x_{\ell})$. From Theorem \ref{thm:BMR_uniform_distribution_result}, $1_{G^{\ell}}\Gamma$ is in the support of $\nu_{a_x}$ and so $\prod_{i=1}^{\ell}g(\gamma_x 1_{G}\Gamma) = \prod_{i=1}^{\ell}g(x) =g(x)^{\ell}>0$ whenever $g(x)>0$ and so the fact that $\int_{X^{\ell}}\prod_{i=1}^{\ell}g(\gamma_x x_i)~d\nu_{x}(x_1,\dots, x_{\ell})$ is positive follows from the fact that $\int g ~d\mu_X$ is positive. This concludes the proof. 
\end{proof}

The following lemma allows us to remove the rounding functions $\round{\cdot}$ in (\ref{eq:BMR_5.2})  

\begin{lemma}[{\cite[Section 5.4]{BMR24}}]\label{lem:floor_removal}
Let $G$ be a simply connected nilpotent Lie group, let
$\Gamma \subset G$ be a uniform and discrete subgroup, let
$X = G/\Gamma$, and let $a \in G$. Let $\mathcal{H}$ be a Hardy field and let $f_1,\ldots,f_{\ell} \in \mathcal{H}$ be subpolynomial functions which satisfy (\ref{eq:additional_assumption}). Let $W \in \mathcal{H}$ with $1 \prec \log W(x) \prec x$. Suppose that $\nu$ is a Borel probability measure on $X^{\ell}$ such that the sequence
\begin{equation}\label{eq:BMR_no_floor}
n \mapsto \bigl(a^{ {f_1(n)}},
a^{ {f_2(n)}},\ldots,
a^{ {f_{\ell}(n)}}\bigr)\Gamma^{\ell}.
\end{equation}
is uniformly distributed with respect to $W$-averages in $(X^{\ell},\nu)$. Then the sequence
\begin{equation}
n \mapsto \bigl(a^{ \round{f_1(n)}},
a^{ \round{f_2(n)}},\ldots,
a^{ \round{f_{\ell}(n)}}\bigr)\Gamma^{\ell}
\end{equation}
is also uniformly distributed with respect to $W$-averages in $(X^{\ell},\nu)$.
\end{lemma}

Now for the proof of Theorem \ref{thm:BMR_uniform_distribution_result}.

\begin{proof}[Proof of Theorem \ref{thm:BMR_uniform_distribution_result}]
    This proof relies on Theorem \ref{thm_generalized_richter_D}, which will be proven in Section \ref{section:richter_results}.

    Without loss of generality, assume that $\overline{\{a^n\Gamma:n\in \Z\}} = X$. Let $a_1 = (a\Gamma,\Gamma,\dots,\Gamma)\in X^{\ell}$, $a_2 = (\Gamma,a\Gamma,\Gamma,\dots,\Gamma)\in X^{\ell}, \dots, a_{\ell} = (\Gamma,\dots,\Gamma,a\Gamma)\in X^{\ell}$, so that 
    \[
    v(n) = a_1^{f_1(n)}\cdots a_{\ell}^{f_{\ell}(n)} \Gamma^{\ell}= (a^{f_1(n)},\dots, a^{f_{\ell}}(n))\Gamma^{\ell}
    \]
    is a sequence in $X^{\ell}$. By Theorem \ref{thm_generalized_richter_D}, there exists $q\in \N$ and closed and connected submanifolds $Y_0,\dots,Y_{q-1}$ such that $(v(qn+r))_{n\in \N}$ is uniformly distributed with respect to $W$-averages in $Y_r$. Let $\nu_a = \frac{1}{q}\sum_{r=0}^{q-1}\mu_{Y_r}$. Then $(v(n))_{n\in \N}$ is uniformly distributed with respect to $W$-averages in $(X,
    \nu_a)$ by Lemma \ref{lem:variant_of_modulo_averages}. From Lemma \ref{lem:floor_removal}, it follows that the sequence
\begin{equation}\label{eq:BMR_5.2_again}
n \mapsto \bigl(a^{ \round{f_1(n)}},
a^{ \round{f_2(n)}},\ldots,
a^{ \round{f_{\ell}(n)}}\bigr)\Gamma^{\ell}
\end{equation}
is uniformly distributed with respect to $W$-averages in $(X^{\ell},\nu_a)$.

The remainder of Theorem \ref{thm:BMR_uniform_distribution_result} follows from \cite[Theorem 5.7]{BMR24}. More specifically, the statement of \cite[Theorem 5.7]{BMR24} is the same as the statement of Theorem \ref{thm:BMR_uniform_distribution_result} except it is assumed that $1\prec W(x)\preceq x$ and that the functions $f_1,\dots, f_{\ell}$ satisfy property (P) (see Definition \ref{def:property_P}) instead of property (\ref{eq:W_compatible}). By \cite[Corollary A.5]{Richter23} there exists a function $\widetilde{W}\in \mathcal{H}$ such that $1\prec \widetilde{W}(x)\preceq x$ and the functions $f_1,\dots, f_{\ell}$ satisfy property (P) for $\widetilde{W}$ (note that property (P) is called property $(P_W)$ in \cite{Richter23}). We assume that $\lim_{x\to\oo}\frac{\log \widetilde{W}(x)}{\log W(x)} = 0$, since otherwise we can replace $\mathcal{H}$ with a maximal Hardy field containing $\mathcal{H}$ and we can replace $\widetilde{W}$ with $\log {W}$.

Then by \cite[Theorem 5.7]{BMR24} there exists a probability measure $\nu^*$ such that the sequence (\ref{eq:BMR_5.2_again}) is uniformly distributed in $(X^{\ell},\nu^*)$ with respect to $\E^{\widetilde{W}}$. Additionally, if $\Poly\{f_1,\dots, f_{\ell}\}$ is jointly intersective then the point
$1_{X^{\ell}} = 1_G\Gamma^{\ell}$ belongs to the support of $\nu^*$, and if $\Poly\{f_1,\dots, f_{\ell}\}\cap \Z[x]=\{0\}$ then $\nu^*$ is the Haar measure on the
subnilmanifold $Y_a^{\ell} \subset X^{\ell}$, where $Y_a = \overline{\{a^n\Gamma : n \in \mathbb{Z}\}}$.

Since (\ref{eq:BMR_5.2_again}) is uniformly distributed in $(X^{\ell},\nu_a)$ with respect to $W$-averages, it is uniformly distributed in $(X^{\ell},\nu_a)$ with respect to $\E^{\widetilde{W}}$ by Theorem \ref{thm:mikey_thm} and our assumption that $\lim_{x\to\oo}\frac{\log \widetilde{W}(x)}{\log W(x)} = 0$. But  (\ref{eq:BMR_5.2_again}) is also uniformly distributed in $(X^{\ell},\nu^*)$ with respect to $\E^{\widetilde{W}}$ and hence it must be that $\nu^* = \nu_a$. This completes the proof.
\end{proof}
\section{Generalizations of results from \texorpdfstring{\cite{Richter23}}{Ric23}}\label{section:richter_results}

In this section, we prove Theorem \ref{thm_generalized_richter_D}, thereby completing the proof of Theorem \ref{thm:BMR_uniform_distribution_result} and, by extension, Theorem \ref{thm:main}.
 \begin{theorem}[cf. {\cite[Theorem D]{Richter23}}]\label{thm_generalized_richter_D}
Let $\mathcal{H}$ be a Hardy field and let $W\in \mathcal{H}$ satisfy $1\prec \log W(x)\prec x$. Let $G$ be a simply connected nilpotent Lie group, let $\Gamma$ a uniform and discrete subgroup of $G$, and assume $X = G/\Gamma$ is connected.
Suppose
$$
v(n) = a_1^{f_1(n)}\cdots a_{\ell}^{f_{\ell}(n)}\quad \text{ for all }n\in \N
$$
where $a_1,\dots , a_{\ell} \in G$ are commuting, and $f_1,\dots , f_{\ell} \in \mathcal{H}$ are subpolynomial functions which satisfy property (\ref{eq:W_compatible}). Then there exists a closed and connected subgroup $H$ of $G$, $q\in \N$, and points $x_0, x_1,\dots , x_{q-1} \in X$ such that $Y_r = H{x_r}$ is a closed sub-nilmanifold of $X$ and $(v(qn + r)\Gamma)_{n\in\N}$ is uniformly
distributed with respect to $W$-averages in $Y_r$ for all $r = 0, 1, \dots , q - 1$.
 \end{theorem}

In order to prove Theorem \ref{thm_generalized_richter_D}, we first introduce and prove a more general statement about uniform distribution on nilmanifolds (Theorem \ref{thm:generalized_G} below) and then at the end of the section we show how Theorem \ref{thm:generalized_G} implies Theorem \ref{thm_generalized_richter_D}.

 To begin, consider the following somewhat technical lemma.

 \begin{lemma}[cf. {\cite[Lemma 6.4]{Richter23}}]\label{lem:replace_richter_6.4}
     Let $W,f$ be functions which belong to the same Hardy field and suppose that $1\prec \log W(x)\prec f(x)\prec x$. Pick $\xi \in (0,1]$ and define $K_n = \N\cap f^{-1}(\xi n, \xi (n+1))$. Also, put $g(n) = f(n)/\xi$, $p_n = \sum_{i\in K_n}\Delta W(i)$, and $P_N = \sum_{n=1}^Np_n$. Then
     \begin{enumerate}[label = (\roman*)]
         \item $\frac{\Delta (W\circ g^{-1})(n)}{p_n}= 1 +o_{n\to\oo}(1)$,  
         \item $\frac{(W\circ g^{-1})(N)}{P_N}= 1 +o_{N\to\oo}(1)$, 
         \item $\lim_{N\to\oo}P_N=\oo$, 
         \item $\frac{p_N}{P_N}=o_{N\to\oo}(1)$.
     \end{enumerate}
 \end{lemma}
 \begin{proof}

Define $G(k) = \max\{n\in \N: {g(n)}\leq k\}
= \max\{n\in \N: f(n)\leq \xi k\}$ for $k\in \N$. For large enough $n$, $K_n = \{G(n)+1,\dots, G(n+1)\}$ and so $p_n = W(G(n+1))-W(G(n)) = \Delta (W\circ G)(n+1)$. Additionally, note that 
\begin{equation}\label{eq:6.4_1}
    g^{-1}(k-1)\leq G(k) \leq g^{-1}(k)
\end{equation} 
for all sufficiently large $k$. By assumption, we have $1\prec \log(W\circ g^{-1})(x)\prec x$ from which it follows that 
\begin{equation}\label{eq:6.4_2}
    \lim_{k\to\oo}\frac{\Delta \log(W\circ g^{-1})(k-1)}{\Delta \log(W\circ g^{-1})(k)}=1.
\end{equation}
Combining (\ref{eq:6.4_1}) and (\ref{eq:6.4_2}) with the fact that $W$ and $g^{-1}$ are eventually monotone, we have $\lim_{n\to\oo}\frac{\Delta (W\circ g^{-1})(n-1)}{p_{n-1}}=\lim_{n\to\oo}\frac{\Delta (W\circ g^{-1})(n-1)}{\Delta (W\circ G)(n)}=1$. This shows that statement (i) holds.

By Theorem \ref{thm:stolz_cesaro} we also have $\lim_{n\to\oo}\frac{ (W\circ g^{-1})(n)}{ (W\circ \hat{s})(n)}=1$, which is statement (ii). Statement (iii) follows from (ii) since we know that $(W\circ g^{-1})$ tends to $\oo$, and (iv) follows from (i) and (ii) since we know that 
\begin{equation}
    \lim_{x\to\oo}\frac{\Delta (W\circ g^{-1})(x)}{(W\circ g^{-1})(x)} =  \lim_{x\to\oo}\frac{\Delta \log (W\circ g^{-1})(x)}{\Delta x}
 = \lim_{x\to\oo}\frac{ \log (W\circ g^{-1})(x)}{ x} = 0.
 \end{equation}

 \end{proof}
 Using Lemma \ref{lem:replace_richter_6.4}, we obtain yet another version of van der Corput's trick, which we will need in the sequel.
\begin{theorem}[cf. {\cite[Proposition 6.1]{Richter23}}]\label{thm:funny_vdc}
Let $\mathcal{H}$ be a Hardy field and suppose that $W,g_1,\dots, g_{m}\in \mathcal{H}$ with 
\begin{equation}
    1\prec \log W(x)\prec g_1(x)\prec g_2(x)\prec\cdots \prec g_{m}(x)\prec x.
\end{equation}
Let $\Psi:\R^{\ell}\rightarrow \C$ be bounded and uniformly continuous. Suppose that for each $\epsilon\in \R$, the limit 
\begin{equation*}
    A({\epsilon}) = \lim_{N\to\infty}\E_{n\leq N}^W\left(\Psi(g_1(n),\dots, g_{m-1}(n),g_{m}(n)+\epsilon)\cdot \overline{\Psi(g_1(n),\dots, g_{m-1}(n),g_{m}(n))}\right)
\end{equation*}
exists. Suppose also that for each $\delta>0$, there is an $\xi\in (0,\delta)$ with $\lim_{H\to\infty}\E_{h\leq H}A(\xi h) = 0$. Then 
\begin{equation}
    \lim_{N\to\infty}\E_{n\leq N}^W\Psi(g_1(n),\dots, g_{m}(n))=0.
\end{equation}
\end{theorem}

   The proof of Theorem \ref{thm:funny_vdc} is exactly the same as the proof of \cite[Proposition 6.1]{Richter23}, except that Lemma \ref{lem:replace_richter_6.4} in needed instead of \cite[Lemma 6.4]{Richter23}. This brings us to the following theorem, which is a generalization \cite[Theorem G]{Richter23} that will be used to prove Theorem \ref{thm_generalized_richter_D}.
\begin{theorem}\label{thm:generalized_G}
    Let $\mathcal{H}$ be a Hardy field. Let $g_1,\dots, g_{m}\in \mathcal{H}$ be subpolynomial and let $W\in \mathcal{H}$ with $1\prec \log W(x)\prec x$.
    Let $G$ be a simply connected nilpotent Lie group, $\Gamma$ a uniform and discrete subgroup of $G$. Define $v:\N\rightarrow G$ by 
    \begin{equation}
        v(n) = z_1^{g_1(n)}\cdots z_{m}^{g_{m}(n)}s_1^{p_1(n)}\cdots s_M^{p_M(n)},
    \end{equation}
    where $z_1,\dots, z_{m}, s_1\dots, s_M\in G$, are pairwise commuting, for each $i$ the set $\overline{s_i^{\mathbb{Z}}\Gamma}$ is a connected subnilmanifold of $X = G/\Gamma$, and $p_1,\dots, p_m\in \Q[x]$. Additionally, assume the following:
    \begin{enumerate}[label =(\arabic*)]
        \item $p_j(\Z)\subset \Z$ for all $j=1,\dots, M$,
        \item $\deg(p_j)=j$ for all $j=1,\dots, M$,
        \item $g_1(x)\prec g_2(x)\prec \cdots \prec g_{m}(x)$,
        \item \label{condition_(4)} for each horizontal character $\eta:X\rightarrow \T$ which is nontrivial on $\overline{z_1^{\R}\cdots z_{m}^{\R}\Gamma}$ we have $\lim_{N\to\oo}\E^W_{n\leq N}\eta(v(n)\Gamma)=0$,
        \item if $g\in \{g_1,\dots, g_{\ell}\}$ has $\deg(g)\geq 2$ then $g'\in \{g_1,\dots, g_{\ell}\}$.
    \end{enumerate}

    Then $(v(n)\Gamma)_{n\in \N}$ is uniformly distributed with respect to $W$-averages in the sub-nilmanifold $\overline{z_1^{\R}\cdots z_{m}^{\R}\cdot s_1^{\Z}\cdots s_M^{\Z}\Gamma}$. 
\end{theorem}

\begin{remark}
    Condition \ref{condition_(4)} in Theorem \ref{thm:generalized_G} is the only part which is not an immediate generalization of \cite[Theorem G]{Richter23}. It replaces the assumption made in \cite[Theorem G]{Richter23} that $W$ is compatible with each $f_i$, which cannot hold in our context since compatibility with $W$ is not preserved under taking derivatives whenever $\log(x)\preceq \log W(x)$. However, condition \ref{condition_(4)} could replace the $W$ compatibility condition in \cite[Theorem G]{Richter23} and the proofs of \cite{Richter23} would remain unchanged.
\end{remark}
We prove Theorem \ref{thm:generalized_G} following the same strategy used the proof of \cite[Theorem G]{Richter23} by considering three cases. In case 1, we assume $G$ is abelian. In case 2, we assume that $s_1=\cdots = s_m = 1_G$ and that each $g_i$ is sublinear. Finally, in case 3, we prove the theorem in general.

\begin{proposition}[cf. {\cite[Theorem  4.1]{Richter23}}]\label{prop:abelian_case}
    Theorem \ref{thm:generalized_G} is true when $G$ is abelian.
\end{proposition}
\begin{proof}
It is known that any connected abelian Lie group is isomorphic to $\R^{d}$ for some $d\geq 0$. Then $\Gamma$ must be isomorphic to $\mathbb{Z}^{d}$, so without loss of generality assume $G = \R^d$ and $\Gamma = \mathbb{Z}^{d}$.

The horizontal characters $\eta:\R^d/\Z^d\rightarrow \T$ are precisely the nonzero central characters $\varphi:\R^d/\Z^d\rightarrow \C\setminus \{0\}$. So, in order to show that
\begin{equation*}
    v(n) \mod \Z^d= z_1g_1(n)+\cdots +z_{m}g_{m}(n)+s_1p_1(n)+\dots+s_Mp_M(n)\mod \Z^d
\end{equation*}
is uniformly distributed in $T:=\overline{\R z_1+\cdots \R z_{m}+\Z s_1+\cdots \Z s_{M}\mod \Z^{d}}$ with respect to $W$-averages it suffices to show that $ \lim_{N\to\oo}\E_{n\leq N}^W\eta( v(n)\mod \Z^d) = 0$ for each nontrival horizontal character $\eta:T\rightarrow \T$ by the Weyl criterion (see Lemma \ref{lem:richter_equivalence}). However, this is guaranteed by condition \ref{condition_(4)} of Theorem \ref{thm:generalized_G}. Indeed, when $\eta$ is is nontrivial on $\overline{\R z_1+\cdots \R z_{m}\mod \Z^{d}}$ we have $\lim_{N\to\oo}\E_{n\leq N}^W\eta( v(n)\mod \Z^d)=0$. So we may suppose that $\eta$ is identically equal to $1$ on $\overline{\R z_1+\cdots \R z_{m}\mod \Z^{d}}$. Then 
\begin{equation}\label{eq:weyl_criterion_in_proof_of_abelian_case}
   \E_{n\leq N}^W\eta( v(n)\mod \Z^d)=\E_{n\leq N}^W\exp( \zeta_1p_1(n)+\dots+\zeta_Mp_M(n))
\end{equation}
for $\zeta_{1},\dots, \zeta_M\in \R$ which are irrational or equal to $0$ (since $\overline{s_i\Gamma}$ is connected for all $i$). At least one $\zeta_i$ must be nonzero and so (\ref{eq:weyl_criterion_in_proof_of_abelian_case}) is equal to $0$ by Theorem \ref{thm:Boshernitzan_wd}. This concludes the proof.

\end{proof}

Using Proposition \ref{prop:abelian_case}, we will prove the next case of Theorem \ref{thm:generalized_G}.
\begin{theorem}[cf. {\cite[Section 4.2]{Richter23}}]\label{thm:sub_linear}
    Theorem \ref{thm:generalized_G} is true under the added assumptions that $s_1=\cdots=s_M = 1_G$ and $g_{m}(x)\prec x$.
\end{theorem}
\begin{proof}
    As in the proof of \cite[Theorem 4.2]{Richter23} we assume that $X  = \overline{z_1^{\R}\cdots z_{m}^{\R}\Gamma}$, that $z_{m}\neq 1_G$, and that $G$ is a $d$-step nilpotent group and we proceed by induction on $d$.

    In the base case $d=1$, $G$ is abelian and so we know that the theorem statement holds in this case by Proposition \ref{prop:abelian_case}. So we suppose that $d\geq 2$ and that the statement of the theorem is true for any $d-1$-step nilpotent group. Let $L$ be the smallest connected, closed, rational, normal subgroup of $G$ which contains $z_{m}^{\R}$. By Lemma \ref{lem:richter_equivalence}, it suffices to show that 
    \begin{equation}\label{eq:richter_equivalence_again}
        \lim_{N\to\oo}\E_{n\leq N}^W\varphi(v(n)\Gamma)=0
    \end{equation}
    holds for each central character $(\varphi,\chi)$ such that $\chi$ is nontrivial on $L^{\circ}\cap Z(G)^{\circ}$. To this end, define
    \begin{equation}
        A(\epsilon) = \lim_{N\to\infty}\E_{n\leq N}^W\left(\varphi(v(n)z_{m}^{\epsilon}\Gamma)\cdot \overline{\varphi(v(n)\Gamma)}\right)
    \end{equation}
as in the statement of Theorem \ref{thm:funny_vdc}. To show that (\ref{eq:richter_equivalence_again}) holds we will apply Theorem \ref{thm:funny_vdc} with $\Psi(x_1,\dots, x_m) = \varphi(z_1^{g_1(x_1)}\cdots z_{m}^{g_{m}(x_m)})$ by first showing that the limit in the definition of $A(\epsilon)$ exists for each $\epsilon\in \R$ and that for each $\delta>0$ there is a $\xi\in (0,\delta)$ such that $\lim_{H\to\infty}\E_{h\leq H}A(\xi h) = 0$. Once we have shown these two facts, we will have shown that $\lim_{N\to\oo}\E_{n\leq N}^W\varphi(v(n)\Gamma)=0$ as desired.

 Fix $\epsilon\in \R$, put $b=z_{m}$, define $v^{\triangle}(n) = (v(n),v(n))$ for $n\in \N$, and define $\Phi_{\epsilon}(g\Gamma)=\varphi(g\cdot b^{\epsilon}\Gamma)\cdot \overline{\varphi(g\Gamma)}$ for all $g\in X$. Then 
\begin{equation}
    A(\epsilon)  = \lim_{N\to\oo}\E_{n\leq N}^W\Phi_{\epsilon}(v(n)\Gamma).
\end{equation}
Additionally, let $\widehat{G} = G/Z(G)$, let $\sigma:G\rightarrow \widehat{G}$ be the quotient map, and put $\widehat{\Gamma} = \sigma(\Gamma)$, $\widehat{v}(n) = \sigma(v(n))$, and $\widehat{X} = \widehat{G}/\widehat{\Gamma} = \overline{\sigma(z_1)^{\R}\cdots \sigma(z_{m})^{\R}\Gamma}$. Note that $\Phi_{\epsilon}(yx\Gamma) = \Phi(x\Gamma)$ for all $y\in Z(G)$, and so the function $\widehat{\Phi}_{\epsilon}:\widehat{X}\rightarrow \C$ defined by  $\widehat{\Phi}_{\epsilon}(g\widehat{\Gamma}) = {\Phi}_{\epsilon}(g\Gamma)$ is well defined.

Note that $\widehat{G}$ is a $d-1$-step nilpotent group. By the induction hypothesis, we have that 
\begin{align*}
    A(\epsilon)  = \lim_{N\to\oo}\E_{n\leq N}^W\Phi_{\epsilon}(v(n)\Gamma) = \lim_{N\to\oo}\E_{n\leq N}^W\widehat{\Phi}_{\epsilon}(v(n)\widehat{\Gamma})= \int_{\widehat{X}}\widehat{\Phi}_{\epsilon}d\mu_{\widehat{X}} = \int_X{\Phi}_{\epsilon}d\mu_{{X}}.
\end{align*}
Define $\Phi:X\times_L X\rightarrow \C$ by $\Phi((g_1,g_2)\Gamma\times_L\Gamma) =\varphi(g_1\Gamma)\overline{\varphi(g_2\Gamma)}$, so that $\Phi$ is a continuous with $\Phi((b^{\epsilon}g,g)\Gamma\times_L\Gamma) = \Phi_{\epsilon}(g\Gamma )$ for all $g\in G$. Then
\begin{align*}
\int \Phi_{\epsilon}d\mu_{{X}} = \int_{(x,x)\in X^{\triangle}} \Phi((b^{\epsilon}x,x)\Gamma\times_L\Gamma)d\mu_{{X^{\triangle}}}.
\end{align*}
Lastly, from Lemma \ref{lem:wd_in_diagonal}, for each $\delta>0$ there exists a value of $\xi\in (0,\delta)$ such that the sequence of subnilmanifolds $(b^{\xi n}, 1_G)X^{\triangle}))_{n\in \N}$ is uniformly distributed in $X\times_L X$, which shows that
\begin{equation}
    \lim_{H\to\oo}\E_{h\leq H}A(\xi h) = \lim_{H\to\oo}\E_{h\leq H} \int_{(x,x)\in X^{\triangle}} \Phi((b^{\xi h}x,x)\Gamma\times_L\Gamma)d\mu_{{X^{\triangle}}} = \int_{X\times_L X}\Phi d\mu_{X\times_L X}.
\end{equation}
Since $\int\Phi d\mu_{X\times_L X}=0$ by \cite[Claim 1 p. 449]{Richter23}, we have completed the proof of the induction step and so we are done.
 \end{proof}
We are now ready to prove the general case of Theorem \ref{thm:generalized_G} (cf. {\cite[Section 4.3]{Richter23}}).
 \begin{proof}[Proof of Theorem \ref{thm:generalized_G}]
 As in \cite[p. 453]{Richter23}, we assume without loss of generality that $X = \overline{z_1^{\R}\cdots z_{m}^{\R}\cdot s_1^{\Z}\cdots s_M^{\Z}\Gamma}$ and that $z_i\neq 1_G$ for all $i$. Indeed, one can replace $G$ with the smallest closed, rational subgroup of $G$ containing $z_1^{\R}\cdots z_{m}^{\R}\cdot s_1^{\Z}\cdots s_M^{\Z}$ if necessary. 
 
 Let $\mathscr{I}$ be the set of all $i$ such that $\deg(g_i)\geq 2$ and $\mathscr{J}$ the set of all $j$ such that $s_j\neq 1_{G}$. If $\mathscr{I}\cup \mathscr{J}=\varnothing$ then we can apply Theorem \ref{thm:sub_linear}, so suppose that $\mathscr{I}\cup \mathscr{J}\neq \varnothing$. Let $L$ be the smallest closed, connected, rational, normal subgroup of $G$ which contains $z_i^{\R}$ for all $i\in \mathscr{I}$ and $s_{j}^{\Z}$ for all $j=1,\dots, M$. Richter shows that $L^{\circ}\cap Z(G)^{\circ}$ is nontrivial \cite[p. 454]{Richter23} and so by Lemma \ref{lem:richter_equivalence} it suffices to show that 
 \begin{equation}\label{eq:generalized_G_phi}
     \lim_{N\to\oo}\E_{n\leq N}^W\varphi(v(n)\Gamma)=0
 \end{equation}
 for each central character $(\varphi,\chi)$ such that $\chi$ is nontrival on $L^{\circ}\cap Z(G)^{\circ}$. Define 
 $$
 A(h) = \lim_{N\to\oo}\E_{n\leq N}^W(\varphi(v(n+h)\Gamma)\cdot \overline{\varphi(v(n)\Gamma)}).
 $$
  We will prove that
 \begin{equation}
\lim_{H\to\oo}\E_{h\leq H}A(h)=0
 \end{equation}
 so that (\ref{eq:generalized_G_phi}) follows by Lemma \ref{lem:vdc} (or by {\cite[Theorem A.8]{Richter23}}).

Now we appeal to following facts proven by Richter \cite[pp. 453-466]{Richter23}.
 \begin{itemize}
    \item There is a continuous function $\Phi:X\times_L X\rightarrow \C$ defined by 
    \begin{equation}
        \Phi((g_1,g_2)\Gamma\times_L \Gamma) = \varphi(g_1\gamma)\cdot \overline{\varphi(g_2\Gamma)} \quad \text{ for all} (g_1,g_2)\in G\times_L G.
    \end{equation}
     \item Let $Z(G)^{\triangle} = \{(g,g):g\in Z(G)\}$ and let $\sigma: G\times_L G\rightarrow G\times_L G/Z(G)^{\triangle}$ be the quotient map. Define $\widehat{G} = \sigma(G)$, $\widehat{\Gamma}= \sigma(\Gamma)$, $\widehat{X}=\widehat{G}/\widehat{\Gamma}$. Then $\Phi$ induces a continuous function $\widehat{\Phi}:\widehat{X}\rightarrow \C$ such that 
         \begin{equation}
        \Phi((g_1,g_2)\Gamma\times_L \Gamma) = \widehat{\Phi}(\sigma(g_1,g_2)\widehat{\Gamma}) \quad \text{ for all }~ (g_1,g_2)\in G\times_L G.
        \end{equation}
    \item $\int \widehat{\Phi}~d\mu_{\widehat{X}}=0$.
    \item There exists a sequence of subnilmanifolds of $\widehat{X}$, $(\widehat{Y}_h)_{h\in \N}$ such that $A(h) = \int \widehat{\Phi}~d\mu_{\widehat{Y}_h}$ for all $h\in \N$.
    \item There exists a sequence $(w(h,\mathbf{n}, \mathbf{l}))_{h\in \N, \mathbf{n}\in \N^{m},\mathbf{l}\in \N^{M}}$ in $\widehat{X}$ which is well distributed in $\widehat{X}$ and for each fixed $h\in \N$, the sequence $(w(h,\mathbf{n}, \mathbf{l}))_{\mathbf{n}\in \N^{m},\mathbf{l}\in \N^{M}}$ is uniformly distributed in $\widehat{Y}_h$.
 \end{itemize}
 
Now we conclude the proof by noting that 
\begin{align*}
    &\lim_{H\to\oo}\E_{h\leq H}A(h)= \lim_{H\to\oo}\E_{h\leq H}\int \widehat{\Phi}~d\mu_{\widehat{Y}_h}\\ =& \lim_{H\to\oo}\lim_{N\to\oo}\E_{h\leq H}\left(\frac{1}{N^{m+M}}\sum_{\mathbf{n}\in\{1,\dots, N\}^{m}}\sum_{\mathbf{l}\in\{1,\dots, N\}^M}\widehat{\Phi}(w(h,\mathbf{n}, \mathbf{l}))\right).
\end{align*}
 
 Since $(w(h,\mathbf{n}, \mathbf{l}))_{h\in \N,\mathbf{n}\in \N^{m},\mathbf{l}\in \N^{M}}$ is well distributed in $\widehat{X}$, observe that Remark \ref{remark:wd_means_W_ud} holds with 
 $$
 \alpha_{H,N_1,\dots, N_{m+M},h,n_1,\dots, n_{m+M}} = \frac{1}{HN_1\cdots  N_{m+M}}\cdot 1_{h\leq H}\cdot 1_{\{n_1\leq N_1\}}\cdots 1_{\{n_{m+M}\leq N_{m+M}\}}
 $$
 to see that that
 \begin{equation}
     \lim_{H\to\oo}\lim_{N\to\oo}\E_{h\leq H}\left(\frac{1}{N^{m+M}}\sum_{\mathbf{n}\in\{1,\dots, N\}^{m}}\sum_{\mathbf{l}\in\{1,\dots, N\}^M}\widehat{\Phi}(w(h,\mathbf{n}, \mathbf{l}))\right) = \int \widehat{\Phi}~d\mu_{\widehat{X}}= 0,
 \end{equation}
 as desired.
 \end{proof}

 We require the following auxiliary lemma, along with Theorem \ref{thm:generalized_G}, to obtain Theorem \ref{thm_generalized_richter_D}.

  \begin{lemma}\label{lem:richter_A.6}
Let $\mathcal{H}$ be a Hardy field. Suppose that $W,f_1,\dots, f_{\ell}\in \mathcal{H}$ such that $\{f_1,\dots, f_{\ell}\}$ satisfies property (\ref{eq:W_compatible}) and $1 \prec \log W(x)\prec x$. Then there exists $m \in \N$, $g_1,\dots, g_m \in \mathcal{H}$, $p_1(t),\dots , p_{\ell}(t) \in \R[t]$, and $(\lambda_{i,j})_{i,j=1}^{\ell,m}\subset \R$
with the following properties:
\begin{enumerate}[label = (\arabic*)]
    \item $g_1(t)\prec \cdots \prec g_m(t)$,
    \item for all nonzero $g\in \{g_1,\dots, g_m\}$, there exists $d\in \N$ such that $x^{d-1}\prec g(x)\prec x^d$,
    \item for all $g\in \{g_1,\dots, g_m\}$, either $W$ is compatible with $g$ or there is an $i\in \{1,\dots, m\}$ such that $g=g_i'$,
    \item for all $g\in \{g_1,\dots, g_m\}$ with $\deg(g)\geq 2$, $g'\in \{g_1,\dots, g_m\}$,
    \item for all $i\in \{1,\dots, \ell\}$,
    $$
    \lim_{t\to\oo}\left|f_i(t)-\sum_{j=1}^m\lambda_{i,j}g_j(t)-p_i(t)\right|=0.
    $$
\end{enumerate}

 \end{lemma}
 \begin{proof}
     \cite[Lemma A.4]{Richter23} says that there exist $(\lambda_{i,j})_{i,j=1}^{\ell,m}\subset \R$, $p_1(t),\dots, p_{\ell}(t)\in \R[t]$, and $g_1,\dots, g_m \in \text{Span}_{\mathbb{R}}\{f^{(k)}_i:i\in \{1,\dots, \ell\}, k\in \{0,1,2,\dots\}\}$ such that the following hold
     \begin{itemize}
         \item $g_1\prec \dots \prec g_m$,
         \item all $g\in \{g_1,\dots, g_m\}$, either $g=0$ or $\deg(g)\geq 1$ and $x^{\deg(g)-1}\prec g(x)\prec x^{\deg(g)}$,
         \item for all $g\in \{g_1,\dots, g_m\}$, if $\deg(g)\geq 2$ then $g' = g_j$ for some $j\in \{1,\dots, m\}$, and if $1\leq \deg(g)\leq \deg(g_m)-1$ then $g_j'=g$ for some $j\in\{1,\dots, m\}$,
         \item 
         for each $i\in \{1,\dots, \ell\}$,
    \begin{equation}\label{eq:A.6}
    \lim_{t\to\oo}\left|f_i(t)-\sum_{j=1}^m\lambda_{i,j}g_j(t)-p_i(t)\right|=0.
    \end{equation}
     \end{itemize}
     The only thing to prove is that this collection $\{g_1,\dots, g_m\}$ satisfies condition (3) in the statement of Lemma \ref{lem:richter_A.6}. To this end, pick $j\in \{1,\dots, m\}$ such that there is no $i\in \{1,\dots, m\}$ with $g_i'=g_j$. We will show that $W$ is compatible with $g_j$. First note that $\deg(g_j) = \deg(g_m) = \max\{\deg(f_i):i\in \{1,\dots, \ell\}\}$, since otherwise $\{g_1,\dots, g_m\}$ would contain an anti-derivative of $g_j$. From the fact that $\deg(g_j) =\max\{\deg(f_i):i\in \{1,\dots, \ell\}\}$ and $g_j\in \text{Span}_{\mathbb{R}}\{f_i^{(k)}:i\in \{1,\dots, \ell\},k\in \{0,1,2,\dots\}\}$, there is an $f \in \text{Span}_{\mathbb{R}}\{f_i:i\in \{1,\dots, \ell\}\} $ with $g_j(x) = f(x)\cdot (1+o_{x\to\oo}(1))$. Then $W$ is compatible with $g_j$ because $\{f_1,\dots, f_{\ell}\}$ satisfies property (\ref{eq:W_compatible}).
 \end{proof}

Now we prove Theorem \ref{thm_generalized_richter_D}.

\begin{proof}[Proof of Theorem \ref{thm_generalized_richter_D}]
    Let $m \in \N$, let $g_1,\dots, g_m \in \mathcal{H}$, let $p_1(t),\dots , p_{\ell}(t) \in \R[t]$, and let $(\lambda_{i,j})_{i,j=1}^{\ell,m}\subset \R$ be given by Lemma \ref{lem:richter_A.6}. As in \cite[Section 3]{Richter23}, consider the following definitions. 

    Let $M\in \N$ and let $(c_{i,j})_{i,j=1}^{\ell,M}$ be real numbers such that 
    \begin{equation*}
        p_i(n)=\sum_{j=0}^Mc_{i,j}\binom{n}{j}.
    \end{equation*}
    Define
    \begin{equation*}
        u_j = \prod_{i=1}^{\ell}a_i^{\lambda_{i,j}}\text{ for } j\in \{1,\dots, m\} \text{ and } e_j= \prod_{i=1}^{\ell}a_i^{c_{i,j}}\text{ for } j\in \{1,\dots, M\}.
    \end{equation*}
    Then by (\ref{eq:A.6})
    \begin{equation}\label{eq:change_of_variables_for_nilsequence}
        u_1^{g_1(n)}\cdots u_m^{g_m(n)}e_1^{\binom{n}{1}}\cdots e_M^{\binom{n}{M}} = a_1^{f_1(n)+o_{n\to\oo}(1)}\cdots a_{\ell}^{f_{\ell}(n)+o_{n\to\oo}(1)},
    \end{equation}
    and so it suffices to show that the conclusions of Theorem \ref{thm_generalized_richter_D} hold for the sequence $(w(n)\Gamma)_{n\in \N}$, where
    \begin{equation}
        w(n)= u_1^{g_1(n)}\cdots u_m^{g_m(n)}e_1^{\binom{n}{1}}\cdots e_M^{\binom{n}{M}}.
    \end{equation}

    It is shown in \cite[pg. 441]{Richter23} that there exists a $q\in \N$ such that $\overline{e_j^{q\Z+r}\Gamma}$ is a connected sub-nilmanifold of $X$ for all $j\in \{1,\dots, M\}$ and all $r\in \{0,\dots, q-1\}$. Define $p_{j,r}(n) = \frac{1}{q}\left(\binom{qn+r}{j}-\binom{r}{j}\right)$, $s_j = e_j^q$, $c_r = e_1^{\binom{r}{1}}\cdots e_M^{\binom{r}{M}}$, $h_{i,r}(t) = q^{-\deg(g_i)}g_i(qt+r)$, $z_i = u_i^{q^{\deg(g_i)}}$ and put 
    $$
    w_r(n)=z_1^{h_{1,r}(n)}\cdots z_m^{h_{m,r}(n)}\cdot s_1^{p_{1,r}(n)}\cdots s_M^{p_{M,r}(n)}
    $$
    for $j\in \{1,\dots, M\}$, $i\in \{1,\dots, m\}$, $r\in \{0,\dots, q-1\}$. Then 
    \begin{equation}\label{eq:w_r_equation}
    w(qn+r) = w_r(n)c_r
    \end{equation}
    for all $r\in \{0,\dots, q-1\}$. For each $r\in \{0,\dots, q-1\}$, the polynomials $p_{1,r},\dots, p_{M,r}$ satisfy conditions (1) and (2) of Theorem \ref{thm:generalized_G}, the functions $h_{1,r},\dots, h_{m,r}$ belong to $\mathcal{H}$ and satisfy condition (3) of Theorem \ref{thm:generalized_G} and satisfy condition (5) of Theorem \ref{thm:generalized_G} by the chain rule. The following claim shows that condition (4) of Theorem \ref{thm:generalized_G} is also satisfied.

    \textbf{Claim:} Let $r\in \{0,1,2,\dots, q-1\}$. Let $\eta:X\rightarrow \C$ be a horizontal character which is nontrivial on $\overline{z_1^{\R}\cdots z_m^{\R}\Gamma}$. Then $\E_{n\leq N}^W\eta\left(w_r(n)
    \Gamma\right)=0.$
    
    To see why this claim is true, note that
    \begin{equation}\label{eq:claim_eq_2}
    \eta(w_r(n)\Gamma) = \eta(a_1^{f_1(qn+r)}\cdots a_{\ell}^{f_{\ell}(qn+r)}c_r^{-1}\Gamma)\cdot (1+o_{n\to\oo}(1))
    \end{equation}
     by (\ref{eq:change_of_variables_for_nilsequence}), (\ref{eq:w_r_equation}), and the continuity of $\eta$.
    Put $\eta(a_i\Gamma) = \alpha_i$ for $i\in \{1,\dots, \ell\}$ and $\eta(c_r^{-1}\Gamma) = \rho$. Then the right-hand side of (\ref{eq:claim_eq_2}) becomes
    $$
    \exp(2\pi i (\alpha_1 f_1(qn+r)+\cdots +\alpha_{\ell}f_{\ell}(qn+r)))\cdot \rho\cdot (1+o_{n\to\oo}(1))
    $$
    $\eta$ is nontrivial on $\overline{z_1^{\R}\cdots z_m^{\R}\Gamma}\subset \overline{a_1^{\R}\cdots a_{\ell}^{\R}\Gamma}$ and so there exists $i\in \{1,\dots, \ell\}$ such that $\alpha_i\neq 0$. There exists a nonzero $f\in \text{Span}_{\mathbb{R}}\{f_1,\dots, f_{\ell}\}$ such that $\alpha_1f_1(qn+r)+\cdots +\alpha_{\ell}f_{\ell}(qn+r) = f(qn+r)\cdot (1+o_{n\to\oo}(1))$. Hence
    \begin{align}\label{eq:w_r_and_f}
   \eta(w_r(n)\Gamma) 
   = \exp(f(qn+r)\cdot (1+o_{n\to\oo}(1)))\cdot \rho \cdot (1+o_{n\to\oo}(1)).
    \end{align}
    If $f$ is bounded then $\lim_{n\to\oo}f(qn+r)$ exists and so the right-hand side of (\ref{eq:w_r_and_f}) would tend to a finite limit as $n\to\oo$, which is impossible since the left-hand side of (\ref{eq:w_r_and_f}) does not tend to a limit as $n\to\oo$, because each of the functions $h_{1,r},\dots, h_{m,r},p_{1,r},\dots, p_{M,r}$ tend to $\oo$ and have different growth rates. Therefore $f$ is unbounded and so the sequence $(f(qn+r))_{n\in \N}$ is u.d. mod 1 with respect to $W$-averages by Theorem \ref{thm:W_ud}, because $\{f_1,\dots, f_{\ell}\}$ satisfies property (\ref{eq:W_compatible}) and $\lim_{x\to\oo}\frac{f^{(d)}(qx+r)}{f^{(d)}(x)}\in (0,\oo)$ for $d=\deg^*(f)$. It follows that 
    $$
    \lim_{N\to\oo} \eta(w_r(n)\Gamma) 
   = \lim_{N\to\oo}\exp(f(qn+r))\cdot \rho =0,
    $$
    which proves the claim and shows that condition \ref{condition_(4)} of Theorem \ref{thm:generalized_G} holds.

    Each of the conditions of Theorem \ref{thm:generalized_G} hold, so for each $r\in \{0,1\dots, q-1\}$ the sequence $(w_r(n)c_r\Gamma)_{n\in \N}$ is uniformly distributed in the subnilmanifold
    \begin{align*}
       Y_r:= &\overline{z_1^{\R}\cdots z_m^{\R}s_1^{\Z}\cdots s_{M}^{\Z}c_r\Gamma}\\
       =&c_r\cdot \overline{z_1^{\R}\cdots z_m^{\R}s_1^{\Z}\cdots s_{M}^{\Z}\Gamma}.
    \end{align*}
    Define $Z  =\overline{z_1^{\R}\cdots z_m^{\R}e_1^{\Z}\cdots e_{M}^{\Z}\Gamma}$ so that $c_r\in Z$ and $Y_r\subset Z$ for all $r\in \{0,\dots, r-1\}$. Each $Y_r$ is connected because of our assumption that $\overline{e^{q\Z+r}\Gamma}$ is connected for each $r\in \{0,1,\dots, q-1\}$. Each left translation of $Y_0$ is either disjoint from $Y_0$ or equal to $Y_0$, and we observe that 
    $$
    \bigcup_{0\leq r_1,\dots, r_M\leq q-1}e_1^{r_{1}}\cdots e_{M}^{r_M}Y_0 = Z.
    $$
From this we conclude that $Y_0$ is a connected component of $Z$. 
    Define 
    $$
    \tilde{H} = \{g\in G: gZ=Z\}, \quad H = \tilde{H}^{\circ}
    $$
    so that $Z = \tilde{H}\Gamma$ and $Y_0 = H\Gamma$. Let $x_r = c_r\gamma$ so that $Y_r = Hx_r$ for all $r$. This completes the proof.
\end{proof}

\bibliographystyle{aomalpha}
\bibliography{Paper_references}



\bigskip
\footnotesize
\noindent
Vitaly Bergelson\\
\textsc{The Ohio State University}\\
\href{mailto:vitaly@math.ohio-state.edu}
{\texttt{vitaly@math.ohio-state.edu}}

\bigskip
\noindent
Michael Reilly\\
\textsc{The Ohio State University}\\
\href{mailto:reilly.201@osu.edu}
{\texttt{reilly.201@osu.edu}}

\end{document}